\documentclass[11pt]{elsarticle}

\usepackage{amssymb}
\usepackage{xcolor}
\usepackage{mathrsfs}
\usepackage{amscd}
\usepackage{appendix}
\usepackage{geometry}
\usepackage{setspace}
\usepackage{amsfonts}
\usepackage{amsmath}
\usepackage{lineno,hyperref}
\usepackage{tikz}
\usetikzlibrary{matrix,arrows}

\usepackage{amssymb,amscd,amsthm,verbatim}

\usepackage{bbm}

\date{\today}

\newtheorem{theorem}{Theorem}[section]
\newtheorem{lemma}[theorem]{Lemma}

\newtheorem{coro}[theorem]{Corollary}

\theoremstyle{definition}

\newtheorem*{definition 1}{Definition 1}
\newtheorem*{definition 2}{Definition 2}
\newtheorem*{definition 3}{Definition 3}
\newtheorem*{definition 4}{Definition 4}
\newtheorem*{definition 5}{Definition 5}
\newtheorem*{definition 6}{Definition 6}

\newtheorem{remark}[theorem]{Remark}

\theoremstyle{plain}

\allowdisplaybreaks \numberwithin{equation}{section}

\journal{arXiv}

\begin{document}
\title{Finite and full scale localization for the multi-frequency\\ quasi-periodic CMV matrices}

\author[1]{Daxiong Piao}
\ead{dxpiao@ouc.edu.cn}
\address[1]{School of Mathematical Sciences, Ocean University of China, Qingdao 266100, P.R.China}

\author[2]{Bei Zhang\corref{cor1}}
\ead{beizhang@nankai.edu.cn}
\address[2]{Chern Institute of Mathematics and LPMC, Nankai University, Tianjin 300071, P.R.China}

\cortext[cor1]{Corresponding author}

\begin{abstract}
This paper establishes finite and full scale localization for multi-frequency quasi-periodic CMV matrices, filling a key gap as the CMV counterpart to Goldstein-Schlag-Voda's multi-frequency quasi-periodic Schr\"{o}dinger operator results \cite{GSV16-arXiv, GSV19-Inventiones}. The multi-frequency case presents significant additional challenges due to the complex five-diagonal structure of CMV matrices, the unitary (rather than Hermitian) nature of their finite-volume restrictions. We overcome these obstacles by proving a large deviation theorem for the entries of the transfer matrix, deriving a covering form of the large deviation theorem applicable to CMV matrices. Furthermore, we provide a detailed treatment of double resonance elimination via the theory of semialgebraic sets and no double resonances condition, which is more intricate than in the single-frequency case.


\end{abstract}

\maketitle


\section{Introduction}
Anderson localization (AL), originating from Anderson's seminal work on disordered systems \cite{Anderson-1958}, describes the transition of a medium from a conductor to an insulator. Mathematically, it manifests as the operator possessing pure point spectrum with exponentially decaying eigenfunctions, a property of profound mathematical significance \cite{ChSu14, LTW09}. Here we would like to mention that AL problems for Schr\"{o}dinger operators are widely studied and fruitful results have been obtained, especially for quasi-periodic potentials with single-frequency; see for example, \cite{AJ10-GEMS, Bourgain-book, BG00-Annals, BS01-CMP, ChSu14, GY20-GAFA, MJ} and references therein.

Given the well-developed AL theory for single-frequency quasi-periodic Schr\"{o}dinger operators, extending it to the multi-frequency case is a natural yet challenging direction. The analysis of shifts on a higher-dimensional torus $\mathbb{T}^{d}$ introduces significant complexities, resulting in a less developed theory for multi-frequency quasi-periodic (MF-QP) operators. Chulaevsky and Sina\v{\i} \cite{CS89-CMP} originally established AL  for two-frequency Schr\"{o}dinger operator almost everywhere in phase space. For the frequency of more than two dimensional case, the same result was obtained in \cite{Bourgain-book}. However, these results were established directly in infinite volume, whereas localization in finite volume was not studied. A major breakthrough was achieved by Goldstein, Schlag, and Voda \cite{GSV16-arXiv, GSV19-Inventiones}, who pioneered analytical techniques to establish both finite and full scale localization for MF-QP Schr\"{o}dinger operators. Based on the tools obtained in \cite{GSV16-arXiv}, they further derived a surprising result that the spectrum of this kind of operators is actually a single interval.  Here we would also like to mention the recent papers \cite{ZL22-DCDS} and \cite{Zhao20-JFA} on MF-QP operators.

Substantial progress has also been made for CMV matrices. Localization has been established for various driving dynamics, including skew-shifts \cite{Cedzich-2021,LPG23-JFA}, single-frequency quasi-periodic potentials \cite{WD19-JFA}, random perturbations \cite{Zhu-arXiv}, and hyperbolic toral automorphisms \cite{LPG24-arXiv}. Meanwhile, the spectral theory in the small quasi-periodic regime has been explored, revealing purely absolutely continuous spectrum \cite{Li-Damanik-Zhou} and Cantor spectrum \cite{Li-Damanik-Zhou-C} in different settings. However, the understanding of MF-QP CMV matrices remains notably limited.

To our knowledge, the finite and full scale localization for MF-QP CMV matrices remains unaddressed. To do that, we need large deviation theorem (LDT) for the entries of the transfer matrix. Comparing with the results in \cite{GS08-GAFA}, there are some difficulties we have to overcome. This is primarily due to the more complex structure of CMV matrices. It is known to all, Schr\"{o}dinger operator is a real tri-diagonal matrix, while the CMV matrix is a complex five-diagonal matrix, then the analysis about the corresponding characteristic determinant (entries of the transfer matrix) is more complicated. To treat this issue, we decompose the matrix into two parts and consider two cases; see the proof of Lemma \ref{Lemma3.16} in detail. Consequently, the large deviation theorem for the entries of the transfer matrix follows.

In addition, considering the restriction of the extended CMV matrix to a finite interval, from \cite{Kruger13-IMRN} and \cite{Zhu-arXiv}, the Poisson formula of CMV matrices depends on the parity of the endpoints of this interval. This also makes the analysis of the covering form of LDT more difficult to some extent. Thus we consider a subinterval of the above finite interval instead, then we get the covering lemma. Since the subinterval is a closed interval, the covering lemma can be applied to the original interval. Then the covering form of LDT can be obtained.

Furthermore, the restriction of the Schr\"{o}dinger operator to a finite interval results in a Hermitian matrix with many favorable properties. In contrast, the restriction of the CMV matrix results in a unitary matrix after modifying the boundary conditions, and its properties are less favorable than those of a Hermitian matrix. For example, there is a famous Cauchy Interlacing Theorem for the eigenvalues of Hermitian matrices. But we can not expect such a result for unitary matrices. This directly leads to a challenge in estimating the number of zeros for the characteristic determinant within a disk. Thus, the factorization of the characteristic determinant becomes more complex.

Furthermore, eliminating double resonances in the multi-frequency setting requires a more intricate scheme than in the single-frequency case \cite{WD19-JFA}. Instead of removing a set of frequencies for a fixed phase, we must independently control and eliminate two separate sets of bad frequencies and phases, a process that is inherently more complex.

The rest of this paper is organized as follows. We describe the settings in Section \ref{section2}. In Section \ref{section3}, we give the estimates of the Lyapunov exponent and prove LDT for the transfer matrix and the entries of the transfer matrix respectively. Furthermore, with the aid of Cartan's estimate and Poisson formula, we get the spectral and covering form of LDT. Finally, we prove the \textit{finite scale localization} of the MF-QP CMV matrices, see Theorem \ref{finite-scale-localization}. In Section \ref{section4}, we derive a result on elimination of double resonances by applying the theory of semialgebraic sets. In Section \ref{section5}, we study the elimination of double resonances between  no double resonances intervals. Based on the analysis in previous sections, we obtain the \textit{full scale localization} of the MF-QP CMV matrices in Section \ref{section6}, see Theorem \ref{main-result}. We put some useful definitions and lemmas in the appendix section, that is, Section \ref{section7}.

\section{Preliminaries}\label{section2}
In this section, we introduce the preliminaries of orthogonal polynomials on the unit circle (OPUC) that can be seen in detail in \cite{Simon-book} and describe the settings we study.

Let $\mathbb{D}=\{z: |z|<1\}$ denote the open unit disk in $\mathbb{C}$,  and let $\mu$ be a nontrivial probability measure on $\partial \mathbb{D}=\{z: |z|=1\}$. Consequently, the functions $1, z, z^{2}, \ldots$ are linearly independent in the Hilbert space $L^{2}(\partial \mathbb{D}, d\mu)$. Let $\Phi_{n}(z)$ be the monic orthogonal polynomials, that is,
$\Phi_{n}(z)=P_{n}[z^{n}]$,  $P_{n}\equiv$ projection onto $\{1, z, z^{2}, \ldots, z^{n-1}\}^{\bot}$ in $L^{2}(\partial \mathbb{D}, d\mu)$. The orthonormal polynomials are then given by
$$\varphi_{n}(z)=\frac{\Phi_{n}(z)}{\|\Phi_{n}(z)\|_{\mu}},$$
where $\|\cdot\|_{\mu}$ denotes the norm of $L^{2}(\partial \mathbb{D}, d\mu)$.

For any polynomial $Q_{n}(z)$ of degree $n$, its Szeg\H{o} dual $Q_{n}^{*}(z)$ is defined as
$$Q_{n}^{*}(z)=z^{n}\overline{Q_{n}(1/\overline{z})}.$$
The Verblunsky coefficients $\{\alpha_{n}\}_{n=0}^{\infty}$ ($\alpha_{n} \in \mathbb{D}$) obey the  equation
\begin{equation}\label{Szego-1}
\Phi_{n+1}(z)=z\Phi_{n}(z)-\overline{\alpha}_{n}\Phi_{n}^{*}(z)
\end{equation}
which is known as the Szeg\H{o} recursion.

 Applying the Gram-Schmidt procedure to $\{1, z, z^{-1}, z^{2}, z^{-2},\cdots\}$, we obtain a CMV basis $\{\chi_{j}\}_{j=0}^{\infty}$, and the matrix representation of multiplication by $z$ relative to the CMV basis gives rise to the CMV matrix $\mathcal{C}$,
$$\mathcal{C}_{ij}=\langle \chi_{i},z\chi_{j} \rangle.$$
The matrix $\mathcal{C}$ takes the following form
\begin{equation*}
\mathcal{C}=
\left(
\begin{array}{ccccccc}
\overline{\alpha}_{0}&\overline{\alpha}_{1}\rho_{0}&\rho_{1}\rho_{0}&  &  & & \\
\rho_{0}&-\overline{\alpha}_{1}\alpha_{0}&-\rho_{1}\alpha_{0}& & & &\\
& \overline{\alpha}_{2}\rho_{1} &-\overline{\alpha}_{2}\alpha_{1}& \overline{\alpha}_{3}\rho_{2}&\rho_{3}\rho_{2}& & \\
& \rho_{2}\rho_{1}& -\rho_{2}\alpha_{1}&-\overline{\alpha}_{3}\alpha_{2}&-\rho_{3}\alpha_{2}& & \\
& & &\overline{\alpha}_{4}\rho_{3}&-\overline{\alpha}_{4}\alpha_{3}&\overline{\alpha}_{5}\rho_{4}& \\
& & &\rho_{4}\rho_{3}&-\rho_{4}\alpha_{3}&-\overline{\alpha}_{5}\alpha_{4}& \\
& & & & \ddots & \ddots & \ddots
\end{array}
\right),
\end{equation*}
where $\rho_{j}=\sqrt{1-|\alpha_{j}|^{2}}$, $j\geq 0$.


The extended CMV matrix is a special five-diagonal doubly-infinite matrix in the standard
basis of $L^{2}(\partial \mathbb{D}, d\mu)$ as described in \cite[Subsection 4.5]{Simon-book} and \cite[Subsection 10.5]{Simon-book2}, and is given by
\begin{equation*}
\mathcal{E}=
\left(
\begin{array}{cccccccc}
\ddots&\ddots&\ddots& & & & & \\
& -\overline{\alpha}_{0}\alpha_{-1}&\overline{\alpha}_{1}\rho_{0}&\rho_{1}\rho_{0}&  &  & & \\
& -\rho_{0}\alpha_{-1}&-\overline{\alpha}_{1}\alpha_{0}&-\rho_{1}\alpha_{0}& & & &\\
& & \overline{\alpha}_{2}\rho_{1} &-\overline{\alpha}_{2}\alpha_{1}& \overline{\alpha}_{3}\rho_{2}&\rho_{3}\rho_{2}& & \\
& & \rho_{2}\rho_{1}& -\rho_{2}\alpha_{1}&-\overline{\alpha}_{3}\alpha_{2}&-\rho_{3}\alpha_{2}& & \\
& & & &\overline{\alpha}_{4}\rho_{3}&-\overline{\alpha}_{4}\alpha_{3}&\overline{\alpha}_{5}\rho_{4}& \\
& & & &\rho_{4}\rho_{3}&-\rho_{4}\alpha_{3}&-\overline{\alpha}_{5}\alpha_{4}& \\
& & & & & \ddots & \ddots & \ddots
\end{array}
\right).
\end{equation*}

In this paper, we consider a sequence of Verblunsky coefficients generated by an analytic function $\alpha(\cdot):\mathbb{T}^{d}\rightarrow \mathbb{D}$, $\alpha_{n}(x)=\alpha(T^{n}x)=\alpha(x+n\omega)$, $d>1$. Here, $T$ is the shift map defined by $Tx=x+\omega$, $\mathbb{T}:=\mathbb{R}/\mathbb{Z}$, and $x, \omega\in \mathbb{T}^{d}$ are referred to as phase and frequency respectively. Moreover, the frequency vector $\omega\in \mathbb{T}^{d}$ satisfies the standard Diophantine condition
\begin{equation}\label{standard-DC}
\|k\cdot \omega\|\geq \frac{p}{|k|^{q}}
\end{equation}
for all nonzero $k\in \mathbb{Z}^{d}$, where $p>0$, $q>d$ are some constants. $\|\cdot\|$ denotes the distance to the nearest integer and $|\cdot|$ stands for the sup-norm on $\mathbb{Z}^{d}$, that is, $|k|=|k_{1}|+|k_{2}|+\cdots+|k_{d}|$, $k_{i}$ is the $i$-th element of the vector $k$. For the sake of convenience, we denote $\mathbb{T}^{d}(p,q)\subset \mathbb{T}^{d}$ be the set of $\omega$ satisfying \eqref{standard-DC}.

Assume that the sampling function $\alpha(x)$ satisfies
\begin{equation}\label{sampling-function}
\int_{\mathbb{T}^{d}}\log(1-|\alpha(x)|)dx>-\infty
\end{equation}
and can be extend complex analytically to
\begin{equation*}
\mathbb{T}^{d}_{h}:=\{x+iy:x\in \mathbb{T}^{d}, y\in \mathbb{R}^{d}, |y|<h\}
\end{equation*}
with some $h>0$. Define its norm as
\begin{equation*}
\|\alpha\|_{h}=\sup_{|y|<h}|\alpha(x+iy)|,
\end{equation*}
and let
\begin{equation}\label{C-alpha}
C_{\alpha}=\log\frac{2}{\sqrt{1-\|\alpha\|^{2}_{h}}}.
\end{equation}

The Szeg\H{o} recursion is then equivalent to
\begin{equation}\label{Szego-2}
\rho_{n}(x)\varphi_{n+1}(z)=z\varphi_{n}(z)-\overline{\alpha}_{n}(x)\varphi_{n}^{*}(z),
\end{equation}
where $\rho_{n}(x)=\rho(x+n\omega)$ and $\rho(x)=(1-|\alpha(x)|^{2})^{1/2}$.

Taking the Szeg\H{o} dual to both sides of (\ref{Szego-2}), we obtain
\begin{equation}\label{Szego-3}
\rho_{n}(x)\varphi_{n+1}^{*}(z)=\varphi_{n}^{*}(z)-\alpha_{n}(x)z\varphi_{n}(z).
\end{equation}
Equations (\ref{Szego-2}) and (\ref{Szego-3}) can be can be expressed as
\begin{equation*}
\left(
\begin{array}{c}
\varphi_{n+1}\\
\varphi_{n+1}^{*}
\end{array}
\right)
=
S(\omega,z;x+n\omega)
\left(
\begin{array}{c}
\varphi_{n}\\
\varphi_{n}^{*}
\end{array}
\right),
\end{equation*}
where
\begin{equation*}
S(\omega,z;x)=\frac{1}{\rho(x)}
\left(
\begin{array}{cc}
z&-\overline{\alpha}(x)\\
-\alpha(x)z&1
\end{array}
\right).
\end{equation*}
Since $\det S(\omega,z;x)=z$, it is preferable to study the determinant $1$ matrix
\begin{equation*}
M(\omega,z;x)=\frac{1}{\rho(x)}
\left(
\begin{array}{cc}
\sqrt{z}&-\frac{\overline{\alpha}(x)}{\sqrt{z}}\\
-\alpha(x)\sqrt{z}&\frac{1}{\sqrt{z}}
\end{array}
\right)\in \mathbb{S}\mathbb{U}(1,1),
\end{equation*}
known as the Szeg\H{o} cocycle map.

The monodromy matrix (or $n$-step transfer matrix) is then defined by
\begin{equation}\label{n-step}
M_{n}(\omega,z;x)=\prod_{j=n-1}^{0}M(\omega,z;x+j\omega).
\end{equation}
From (\ref{n-step}), it follows directly that
\begin{equation*}
M_{n_{1}+n_{2}}(\omega,z;x)=M_{n_{2}}(\omega,z;x+n_{1}\omega)M_{n_{1}}(\omega,z;x)
\end{equation*}
and
\begin{equation}\label{log-subadditive}
\log\|M_{n_{1}+n_{2}}(\omega,z;x)\|\leq \log\|M_{n_{1}}(\omega,z;x)\|+\log\|M_{n_{2}}(\omega,z;x+n_{1}\omega)\|.
\end{equation}
Define
\begin{equation}\label{un}
u_{n}(\omega,z;x):=\frac{1}{n}\log\|M_{n}(\omega,z;x)\|
\end{equation}
and
\begin{equation}\label{un}
L_{n}(\omega,z):=\int_{\mathbb{T}^{d}}u_{n}(\omega,z;x)dx.
\end{equation}

Integrating the inequality (\ref{log-subadditive}) with respect to $x$ over $\mathbb{T}^{d}$ yields
\begin{equation*}
L_{n_{1}+n_{2}}(\omega,z)\leq \frac{n_{1}}{n_{1}+n_{2}}L_{n_{1}}(\omega,z)+\frac{n_{2}}{n_{1}+n_{2}}L_{n_{2}}(\omega,z).
\end{equation*}
This implies that
\begin{equation*}
L_{n}(\omega,z)\leq L_{m}(\omega,z)\ \ {\rm if}\ \  m<n,\; m|n
\end{equation*}
and
\begin{equation*}
L_{n}(\omega,z)\leq L_{m}(\omega,z)+C\frac{m}{n} \ \ {\rm if}\ \ m<n.
\end{equation*}

For any irrational $\omega\in \mathbb{T}^{d}$, the transformation $x\rightarrow x+\omega$ is ergodic. Notice that (\ref{log-subadditive}) implies that $\log \|M_{n}(\omega,z;x)\|$ is subadditive, Kingman's subadditive ergodic theorem guarantees the existence of the limit
\begin{equation}\label{Lyapunov-exponent}
L(\omega,z)=\lim_{n\rightarrow\infty}L_{n}(\omega,z),
\end{equation}
known as the Lyapunov exponent. Throughout this paper,  $\gamma>0$ denotes the lower bound of the Lyapunov exponent.

Moreover, the Furstenberg-Kesten theorem states that the limit
\begin{equation}\label{Lyapunov-exponent-1}
\lim_{n\rightarrow\infty}u_{n}(\omega,z;x)=\lim_{n\rightarrow\infty}L_{n}(\omega,z)=L(\omega,z)
\end{equation}
exists for $ a.e.\;x$.
Note that
\begin{equation}\label{log-Mn}
0\leq\log \|M_{n}(\omega,z;x)\|\leq C(\alpha_{i},z)n
\end{equation}
and hence
\begin{equation}\label{Ln-bound}
0\leq L_{n}(\omega,z)\leq C(\alpha_{i},z)
\end{equation}
for $i\in \mathbb{Z}$.

\section{Some Useful Lemmas}\label{section3}
\subsection{LDT Estimate for the Monodromy}
To establish LDT, we rely on the following lemma.
\begin{lemma}\label{Lemma3.1}\cite[Proposition 9.1]{GS01-Annals}
Let $d$ be a positive integer. Suppose $u: D(0,2)^{d}\rightarrow [-1,1]$ is subharmonic in each variable; i.e., $z\rightarrow u(z_{1},z_{2},\ldots,z_{d})$ is subharmonic for any choice of $(z_{2},\ldots,z_{d})\in D(0,2)^{d-1}$ and similarly for each of the other variables. Assume furthermore that for some $n\geq 1$
\begin{equation}\label{GS01-(1)}
\sup_{\theta \in \mathbb{T}^{d}}|u(\theta+\omega)-u(\theta)|<\frac{1}{n}.
\end{equation}
Then there exist $\sigma>0$, $\tau>0$, and $c_{0}$ only depending on $d$ and $\varepsilon_{1}$ such that
\begin{equation}\label{GS01-(2)}
\mathrm{mes}\{\theta\in \mathbb{T}^{d}: |u(\theta)-\langle u\rangle|> n^{-\tau}\}<\exp(-c_{0}n^{\sigma}).
\end{equation}
Here $\langle u\rangle=\int_{\mathbb{T}^{d}}u(\theta)d\theta$. If $d=2$ then the range is $0<\tau<\frac{1}{3}-\varepsilon_{2}$ and $\sigma=\frac{1}{3}-\tau-\varepsilon_{2}$ where $\varepsilon_{2}\rightarrow 0$ as $\varepsilon_{1}\rightarrow 0$.
\end{lemma}
Based on the above lemma, one can obtain the following LDT, a fundamental tool in localization theory.
\begin{theorem}\label{LDT-matrix}
Suppose $\omega\in \mathbb{T}^{d}(p,q)$, $z\in \partial \mathbb{D}$, and $L(\omega,z)>\gamma>0$. There exist  constants $\sigma=\sigma(p,q)$, $\tau=\tau(p,q)$ with $\sigma, \tau \in (0,1)$, and $C_{0}=C_{0}(p,q,h)$ such that for $n\geq 1$,
\begin{equation}\label{LDT1}
\mathrm{mes}\{x\in \mathbb{T}^{d}: |\log\|M_{n}(\omega,z;x)\|-nL_{n}(\omega,z)|> n^{1-\tau}\}<\exp(-C_{0}n^{\sigma}).
\end{equation}
\end{theorem}
\begin{proof}
Fix a dimension $d$ and $z$. For any $x\in D(0,2)^{d}$, define
$$u_{n}(x)=\frac{1}{n}\log\|M_{n}(\omega,z;x)\|.$$
Under the framework of analytic continuation in several complex variables, we note that in Szeg\H{o} cocycle map $\overline{\alpha}(x)=\overline{\alpha(\overline{x}_{1},\overline{x}_{2},\ldots,\overline{x}_{d})}$. Thus, $\overline{\alpha}(x)$ is analytic. Then $u_{n}$ is a continuous subharmonic function bounded by $1$ in $D(0,1)^{d}$. Moreover, $u_{n}$ satisfies the conditions in the above lemma, and  \eqref{LDT1} follows directly from \eqref{GS01-(2)}.
\end{proof}

\subsection{Estimation of the Lyapunov Exponent}
The Lyapunov exponent is a core indicator characterizing the spectral properties and localization behavior of operators. First, we need to clarify the deviation estimate between the finite scale approximation of the Lyapunov exponent $L_{n}(\omega,z)$ and its limit $L(\omega,z)$.

\begin{lemma}\label{Lemma3.4}
Assume $\omega\in \mathbb{T}^{d}(p,q)$, $z\in \partial \mathbb{D}$, and $L(\omega,z)>\gamma>0$. For any $n\geq 2$, we have
\begin{equation*}
0\leq L_{n}(\omega,z)-L(\omega,z)<C\frac{(\log n)^{1/\sigma}}{n},
\end{equation*}
where $C=C(p,q,z,\gamma)$ and $\sigma$ is as in LDT.
\end{lemma}
\begin{proof}
Clearly, $0\leq L_{n}(\omega,z)\leq C(z)$ for all $n$. Let $k$ be a positive integer satisfying $k\gamma>16C(z)$. For $n>10$, set $l=[C_{1}(\log n)^{1/\sigma}]$ with some large $C_{1}$. Consider the integers $l_{0}, 2l_{0}, \ldots,2^{k}l_{0}$. There exists some $0\leq j<k$ such that (with $l_{j}=2^{j}l_{0}$)
\begin{equation}\label{lem3.4-(1)}
L_{l_{j}}(z)-L_{l_{j+1}}(z)<\frac{\gamma}{16}.
\end{equation}
Otherwise, then $C(z)>L_{l_{0}}(z)-L_{2^{k}l_{0}}(z)\geq \frac{k\gamma}{16}>C(z)$, leading to a contradiction. Set $l=2^{j}l_{0}$ where $j$ satisfies \eqref{lem3.4-(1)}.

It is straightforward to verify that to verify that $M(\omega,z;x)$ is conjugate to an $\mathbb{SL}(2,\mathbb{R})$ matrix $T(\omega,z;x)$  via
\begin{equation*}
Q=-\frac{1}{1+i}
\left(
\begin{array}{cc}
1&-i\\
1&i
\end{array}
\right)\in \mathbb{U}(2),
\end{equation*}
that is,
\begin{equation*}
T(\omega,z;x)=Q^{*}M(\omega,z;x)Q\in \mathbb{SL}(2,\mathbb{R}),
\end{equation*}
where ``$*$'' denotes the conjugate transpose of a matrix.

We now apply the Avalanche Principle (AP) (see Appendix) to the matrices $B_{j}=T_{l}(\omega,z;x+(j-1)l\omega)$ for $1\leq j\leq m=[n/l]$, where $\mu=\exp(l\gamma/2)$. Note that $\mu>n^{2}$ if $C_{1}>4/\gamma$. By LDT, the phases $x$ satisfying
\begin{equation*}
\min_{1\leq j\leq m} \|B_{j}\|\geq \exp(lL_{l}(\omega,z)-l^{1-\tau})>\exp(l\gamma/2)=\mu>n^{2}
\end{equation*}
form a set $\mathcal{G}_{1}$ with measure at most $m \exp(-C_{0}l^{\sigma})\leq n\exp(-C_{0}(2^{j}l_{0})^{\sigma})<n^{-2}$, provided $C_{1}>n^{1/\sigma}$.

Moreover, combining \eqref{lem3.4-(1)} and LDT, we have
\begin{align*}
\max_{1\leq j< m}\Big|\log\|B_{j+1}\|+\log\|B_{j}\|-\log\|B_{j+1}B_{j}\| \Big|
&\leq 2l(L_{l}(z)+\frac{\gamma}{32})-2l(L_{2l}(z)-\frac{\gamma}{32})\\
&=2l[(L_{l}(z)-L_{2l}(z))+\frac{\gamma}{16})]\\
&\leq \exp(l\gamma/4)=\frac{1}{2}\log \mu
\end{align*}
up to a set $\mathcal{G}_{2}$ of $x$ with measure at most $2m \exp(-C_{0}l^{\sigma})\leq 2n\exp(-C_{0}(l)^{\sigma})<n^{-2}$, provided $C_{1}>n^{1/\sigma}$. Let $\mathcal{G}=\mathcal{G}_{1}\cup \mathcal{G}_{2}$. Then $\mathrm{mes}\mathcal{G}<2n^{-2}$, and \eqref{AP-1}, \eqref{AP-2} hold for all $x \in\mathcal{G}$.

By (AP), we have
\begin{equation*}
\Big|\log\|B_{m}\cdots B_{1}\|+\sum_{j=2}^{m-1}\log\|B_{j}\|-\sum_{j=1}^{m-1}\log\|B_{j+1}B_{j}\|\Big|<C\frac{n}{\mu}<Cn^{-1}
\end{equation*}
for all $x\in\mathcal{G}$. That is,
\begin{align*}
&\Big|\log\|T_{lm}(\omega,z;x)+\sum_{j=2}^{m-1}\log\|T_{l}(\omega,z;x+jl\omega)\|\\
&\ \ \ -\sum_{j=1}^{m-1}\log\| T_{l}(\omega,z;x+(j+1)l\omega)T_{l}(\omega,z;x+jl\omega) \|\Big|< Cn^{-1}
\end{align*}
for all $x\in\mathcal{G}$.

Similarly, for $\log \|T_{lm}(\omega,z;x+lm\omega)\|$ and $\log \|T_{lm}(\omega,z;x)\|$, the phases $x$ satisfying
\begin{align}	
&\Big| \|\log\| T_{2lm}(\omega,z;x)\|-\|\log\| T_{lm}(\omega,z;x+lm\omega)\|-\|\log\| T_{lm}(\omega,z;x)\| \nonumber\\
&+\|\log\| T_{l}(\omega,z;x+lm\omega)\|+\|\log \|T_{l}(\omega,z;x+(m-1)l\omega)\| \nonumber \\
&-\|\log \|T_{l}(\omega,z;x+lm\omega)T_{l}(\omega,z;x+(m-1) l\omega)\|\Big|\leq \frac{C}{n} \label{Lem3.4-(2)}
\end{align}
form a set with measure  at most $Cn^{-2}$.\\
Since
\begin{equation}\label{Lem3.4-(3)}
\big|\|\log\| T_{n}(\omega,z;x)\|-\log\| T_{lm}(\omega,z;x)\| \big|\leq C(\alpha_{n},\alpha_{lm},z)n
\end{equation}
and
\begin{equation}\label{Lem3.4-(4)}
\big|\|\log\| T_{l}(\omega,z;x)\| \big|\leq C(z)l,
\end{equation}
we can conclude from \eqref{Lem3.4-(2)} that
\begin{equation*}
\big|\|\log\| T_{2n}(\omega,z;x)\|-\log\| T_{n}(\omega,z;x+n\omega)\|-\log\| T_{n}(\omega,z;x)\| \big|\leq C(\log n)^{1/\sigma}
\end{equation*}
up to a set of $x$ not exceeding $Cn^{-2}$ in measure.

Integrating the above inequality over $x$, then we have
\begin{equation*}
|L_{2n}(z)-L_{n}(z)|\leq C\frac{(\log n)^{1/\sigma}}{n},
\end{equation*}
where $C=C(\gamma,z,\omega)$.

We can finally prove this lemma by summing over $2^{k}n$ .
\end{proof}
Based on the continuity result in the above lemma, we can further refine the stability of the logarithm of the transfer matrix norm under small perturbations.

\begin{lemma}\label{Lemma3.5}
Let $n\geq 1$, $(\omega_{i},z_{i},x_{i})\in \mathbb{C}^{d}\times\partial \mathbb{D}\times \mathbb{C}^{d}$, $i=1,2$, such that
\begin{equation*}
|\mathrm{Im} x_{i}|<h,\quad n|\mathrm{Im} \omega_{i}|<h.
\end{equation*}
Then
\begin{equation*}
\|M_{n}(\omega_{1},z_{1};x_{1})-M_{n}(\omega_{2},z_{2};x_{2})\|\leq \Big(\frac{4}{\sqrt{1-\|\alpha\|}_{h}}\Big)^{n}(|\omega_{1}-\omega_{2}|+|z_{1}-z_{2}|+|x_{1}-x_{2}|).
\end{equation*}
In particular, we have
\begin{equation*}
\big|\log\|M_{n}(\omega_{1},z_{1};x_{1})\|-\log \|M_{n}(\omega_{2},z_{2};x_{2})\| \big|\leq \Big(\frac{4}{\sqrt{1-\|\alpha\|}_{h}}\Big)^{n}(|\omega_{1}-\omega_{2}|+|z_{1}-z_{2}|+|x_{1}-x_{2}|),
\end{equation*}
provided the right-hand side is less than $\frac{1}{2}$.
\end{lemma}
\begin{proof}
Let
\begin{equation*}
M_{k,n}=\frac{1}{\rho(x_{k}+n\omega_{k})}
\left(
\begin{array}{cc}
\sqrt{z_{k}}&-\frac{\overline{\alpha}(x_{k}+n\omega_{k})}{\sqrt{z_{k}}}\\
-\alpha(x_{k}+n\omega_{k})&1/\sqrt{z_{k}}
\end{array}
\right),\quad k=1,2.
\end{equation*}
Then
\begin{align*}	
 M_{n}(\omega_{1},z_{1};x_{1})-M_{n}(\omega_{2},z_{2};x_{2})&=\sum_{j=1}^{n}M_{2,n}\cdots M_{2,j+1}(M_{1,j}-M_{2,j})M_{1,j-1}\cdots M_{1,1} \\
&=\sum_{j=1}^{n}M_{n-j}(\omega_{2},z_{2};x_{2})(M_{1,j}-M_{2,j})M_{j-1}(\omega_{1},z_{1};x_{1}).
\end{align*}
It is not difficult to check that $M_{k,n}$ is conjugate to an $\mathbb{SL}(2,\mathbb{R})$ matrix $T_{k,n}$ through
$Q$, that is,
\begin{equation*}
T_{k,n}=Q^{*}M_{k,n}Q.
\end{equation*}

According to the polar decomposition of $\mathbb{SL}(2,\mathbb{C})$ matrices,
\begin{equation*}
\mathrm{tr}(T_{k,n}^{*}T_{k,n})=|a|^{2}+|b|^{2}+|c|^{2}+|d|^{2}=\|T_{k,n}\|^{2}+\|T_{k,n}\|^{-2},
\end{equation*}
where $a,b,c,d$ are the entries of matrix $T_{k,n}$. Then a direct calculation implies
\begin{align*}	
\|T_{k,n}\|+\|T_{k,n}\|^{-1}&=\|M_{k,n}\|+\|M_{k,n}\|^{-1} \\
&=\sqrt{\mathrm{tr}(M_{k,n}^{*}M_{k,n})+2}  \\
&= \frac{\sqrt{2+2|\alpha(x_{k}+n\omega_{k})|^{2}+2\rho^{2}(x_{k}+n\omega_{k})}}{\rho(x_{k}+n\omega_{k})}\\
&=\frac{2}{\sqrt{1-|\alpha(x_{k}+n\omega_{k})|^{2}}}.
\end{align*}
Thus, $\sup\limits_{|y|<h}\|T_{k,n}(\omega_{k},z_{k};x_{k}+iy_{k})\|$, $\sup\limits_{|y|<h}\|T_{k,n}(\omega_{k},z_{k};x_{k}+iy_{k})\|^{-1}\leq \frac{2}{\sqrt{1-\|\alpha\|_{h}^{2}}}$ and
\begin{equation*}
\sup\limits_{|y|<h}\frac{1}{n}\log\|T_{n}(\omega,z;x)\|=\sup\limits_{|y|<h}\frac{1}{n}\log\|M_{n}(\omega,z;x)\|\leq
\log\frac{2}{\sqrt{1-\|\alpha\|_{h}^{2}}}=C_{\alpha},
\end{equation*}
where $T_{n}(\omega,z;x)=\mathop{\prod}\limits_{j=n-1}^{0}T(\omega,z;x+j\omega)$, $T(\omega,z;x+j\omega)=Q^{*}M(\omega,z;x+j\omega)Q$.

Then according to the Mean Value Theorem, one can obtain that
\begin{equation*}
\|M_{n}(\omega_{1},z_{1};x_{1})-M_{n}(\omega_{2},z_{2};x_{2})\|\leq \Big(\frac{4}{\sqrt{1-\|\alpha\|^{2}}_{h}}\Big)^{n}(|\omega_{1}-\omega_{2}|+|z_{1}-z_{2}|+|x_{1}-x_{2}|).
\end{equation*}
Due to the fact that $|\log x|\leq 2|x-1|$ provided $|x-1|\leq\frac{1}{2}$, we get
\begin{align*}
\big|\log\|M_{n}(\omega_{1},z_{1};x_{1})\|-\log \|M_{n}(\omega_{2},z_{2};x_{2})\|\big| &\leq \Big|\frac{\|M_{n}(\omega_{1},z_{1};x_{1})\|}{\|M_{n}(\omega_{2},z_{2};x_{2})\|}-1\Big|\\
&\leq \frac{\|M_{n}(\omega_{1},z_{1};x_{1})-M_{n}(\omega_{2},z_{2};x_{2})\|}{\|M_{n}(\omega_{2},z_{2};x_{2})\|}
\end{align*}
provided the right-hand side is less than $\frac{1}{2}$.

Since $\|M_{n}(\omega,z;x)\|\geq 1$ for any $\omega,z,x$, then we have
\begin{align*}
&\big|\log\|M_{n}(\omega_{1},z_{1};x_{1})\|-\log \|M_{n}(\omega_{2},z_{2};x_{2})\|\big|\\
&\leq \|M_{n}(\omega_{1},z_{1};x_{1})-M_{n}(\omega_{2},z_{2};x_{2})\|\\
&\leq \Big(\frac{4}{\sqrt{1-\|\alpha\|^{2}}_{h}}\Big)^{n}(|\omega_{1}-\omega_{2}|+|z_{1}-z_{2}|+|x_{1}-x_{2}|).
\end{align*}
\end{proof}

Based on (AP), we can refine the above lemma.
\begin{lemma}\label{Lemma3.6}
Suppose $\omega_{0}\in \mathbb{T}^{d}(p,q)$, $z_{0}\in \partial \mathbb{D}$, and $L(\omega_{0},z_{0})>\gamma>0$. Let $\sigma$ be as in LDT and let $A$ be a constant such that $\sigma A\geq 1$. Then there exists a set $\mathcal{B}_{n,\omega_{0},z_{0}}$, $\mathrm{mes}(\mathcal{B}_{n,\omega_{0},z_{0}})<\exp(-C_{0}(\log n)^{\sigma A})$, such that the following holds. For all $n\geq N_{0}(\alpha_{i},p,q,\gamma,A)$,
\begin{equation}\label{lem3.6-(1)}
\big|\log\|M_{n}(\omega,z;x)\|-\log \|M_{n}(\omega_{0},z_{0};x_{0})\|\big|< \exp(-(\log n)^{A})
\end{equation}
for any $x_{0}\in \mathbb{T}^{d}\backslash \mathcal{B}_{n,\omega_{0},z_{0}}$ and $(\omega,z,x)\in \mathbb{C}^{d}\times \partial \mathbb{D}\times\mathbb{C}^{d}$ such that
\begin{equation}\label{lem3.6-(2)}
|\omega-\omega_{0}|,\,|z-z_{0}|,\,|x-x_{0}|<\exp(-(\log n)^{4A}).
\end{equation}
\end{lemma}
\begin{proof}
Let $l\simeq (\log n)^{A+1}$, $m=[n/l]$, $B_{j}=T_{l}(\omega,z;x+(j-1)l\omega)$ be as in the proof of Lemma \ref{Lemma3.4}. For convenience, we also assume that $n=ml$.

Take $\mathcal{B}_{n,\omega_{0},z_{0}}\subset \mathbb{T}^{d}$ be the set in \eqref{LDT1}. By LDT and Lemma \ref{Lemma3.4}, we get that
\begin{equation*}
\min_{1\leq j\leq m}\|B_{j}(\omega_{0},z_{0};x_{0})\|\geq \exp(lL_{l}(\omega_{0},z_{0})-l^{1-\tau})\geq \exp(l\gamma/2)>m,
\end{equation*}
\begin{equation*}
\max_{1\leq j< m}(\log\|B_{j+1}(\omega_{0},z_{0};x_{0})\|+ \log\|B_{j}(\omega_{0},z_{0};x_{0})\|-\log\|B_{j+1}(\omega_{0},z_{0};x_{0})B_{j}(\omega_{0},z_{0};x_{0})\|)\leq \frac{l\gamma}{16}
\end{equation*}
for all $x_{0}$ outside $\mathcal{B}_{n,\omega_{0},z_{0}}$ with
\begin{equation*}
\mathrm{mes}(\mathcal{B}_{n,\omega_{0},z_{0}})\lesssim m\exp(-C_{0}l^{\sigma})<\exp(-C_{0}(\log n)^{\sigma A}),
\end{equation*}
provided that $\sigma A\geq 1$. Take $x_{0}\in \mathbb{T}^{d}\backslash\mathcal{B}_{n,\omega_{0},z_{0}} $. Applying (AP) with $\mu=\exp(l\gamma/2)$, we have
\begin{align}\label{lem3.6-(3)}
&\Big|\log\|M_{m}(\omega_{0},z_{0};x_{0})\|+\sum_{j=2}^{m-1}\log\|B_{j}(\omega_{0},z_{0};x_{0})\|-
\sum_{j=1}^{m-1}\log\|B_{j+1}(\omega_{0},z_{0};x_{0})B_{j}(\omega_{0},z_{0};x_{0})\|\Big|\nonumber\\
&<Cn\exp(-l\gamma/2).
\end{align}
Take $\omega,z,x$ satisfying \eqref{lem3.6-(2)}. From Lemma \ref{Lemma3.5}, we know that
\begin{equation*}
\min_{1\leq j\leq m}\|B_{j}(\omega,z;x)\|\geq \exp(l\gamma/4)>m
\end{equation*}
and
\begin{equation*}
\max_{1\leq j< m}(\log\|B_{j+1}(\omega,z;x)\|+\log\|B_{j}(\omega,z;x)\|-\log\|B_{j+1}(\omega,z;x)B_{j}(\omega,z;x)\|)\leq\frac{l\gamma}{8}.
\end{equation*}
Now applying (AP) again with $\mu=\exp(l\gamma/4)$, we have
\begin{align}\label{lem3.6-(4)}
&\Big|\log\|M_{m}(\omega,z;x)\|+\sum_{j=2}^{m-1}\log\|B_{j}(\omega,z;x)\|-
\sum_{j=1}^{m-1}\log\|B_{j+1}(\omega,z;x)B_{j}(\omega,z;x)\|\Big|\nonumber\\
&<Cn\exp(-l\gamma/4).
\end{align}
Subtracting \eqref{lem3.6-(3)} from \eqref{lem3.6-(4)} and applying Lemma \ref{Lemma3.5}, we get
\begin{align*}
\big|\log\|M_{m}(\omega,z;x)\|-\log\|M_{m}(\omega_{0},z_{0};x_{0})\|\big|
&\leq \Big(\frac{4}{\sqrt{1-\|\alpha\|}_{h}}\Big)^{n}(|\omega-\omega_{0}|+|z-z_{0}|+|x-x_{0}|)\\
&\leq 3\Big(\frac{4}{\sqrt{1-\|\alpha\|}_{h}}\Big)^{n}\exp(-(\log n)^{4A})\\
&<\exp(-(\log n)^{A}).
\end{align*}
\end{proof}

Now we study continuity from the transfer matrix level to the Lyapunov exponent level.

\begin{lemma}\label{Lemma3.7}
Suppose $\omega_{0}\in \mathbb{T}^{d}(p,q)$, $z_{0}\in \partial \mathbb{D}$, and $L(\omega_{0},z_{0})>\gamma>0$. Let $\sigma$ be as in LDT and $A$ be a constant satisfying $\sigma A\geq1$. Then for all $n\geq N_{0}(p,q,\gamma,A)$,
\begin{equation*}
|L_{n}(\omega,z)-L_{n}(\omega_{0},z_{0})|<\exp(-(\log n)^{\sigma A})
\end{equation*}
provided that $|\omega-\omega_{0}|$, $|z-z_{0}|<\exp(-(\log n)^{4A})$.
\end{lemma}
\begin{proof}
Due to
\begin{align*}
|L_{n}(\omega,z)-L_{n}(\omega_{0},z_{0})|&=\Big|\frac{1}{n} \Big(\int_{\mathbb{T}^{d}}\log\|M_{n}(\omega,z;x)\|dx-\int_{\mathbb{T}^{d}}\log\|M_{n}(\omega_{0},z_{0};x_{0})\|dx\Big)\Big|\\
&\leq \frac{1}{n}\int_{\mathbb{T}^{d}}\Big|\log\|M_{n}(\omega,z;x)\|-\log\|M_{n}(\omega_{0},z_{0};x_{0})\|\Big| dx,
\end{align*}
then according to Lemma \ref{Lemma3.6} we have
\begin{equation*}
|L_{n}(\omega,z)-L_{n}(\omega_{0},z_{0})|<\exp(-(\log n)^{\sigma A}).
\end{equation*}
\end{proof}

Now we focus on the local continuity of the Lyapunov exponent with respect to frequency and spectral parameters.

\begin{lemma}\label{Lemma3.8}
Assume $\omega\in\mathbb{T}^{d}(p,q)$, $z\in \partial \mathbb{D}$, and $L(\omega,z)>\gamma>0$. Let $\sigma$ be as in LDT. There exists $\varepsilon_{0}=\varepsilon_{0}(p,q,z_{0},\gamma)$ such that if $|z-z_{0}|<\varepsilon_{0}$, then $L(\omega_{0},z)>\frac{\gamma}{2}$ and
\begin{equation*}
|L(\omega_{0},z)-L(\omega_{0},z_{0})|\leq \exp(\frac{1}{2}(-\log|z-z_{0}|)^{\sigma/4}).
\end{equation*}
Furthermore, if $\omega\in \mathbb{T}^{d}(p,q)\cap (\omega_{0}-\varepsilon_{0},\omega_{0}+\varepsilon_{0})$, then $L(\omega,z_{0})>\frac{\gamma}{2}$ and
\begin{equation*}
|L(\omega,z_{0})-L(\omega_{0},z_{0})|\leq \exp(\frac{1}{2}(-\log|\omega-\omega_{0}|)^{\sigma/4}).
\end{equation*}
\end{lemma}
\begin{proof}
Combining Lemma \ref{Lemma3.4} and Lemma \ref{Lemma3.7}, we have that
\begin{align*}
|L(\omega_{0},z)-L(\omega_{0},z_{0})|&<C\frac{(\log n)^{1/\sigma}}{n}\leq \exp(-\frac{1}{2}\log(n+1))\\
&\leq\exp(\frac{1}{2}(-\log|z-z_{0}|)^{\sigma/4})<\frac{\gamma}{2},
\end{align*}
provided
\begin{equation*}
\exp(-(\log(n+1))^{4/\sigma})\leq |z-z_{0}|\leq \exp(-(\log n)^{4/\sigma})
\end{equation*}
for $n\geq n_{0}(p,q,z_{0},\gamma)$. Then we get the first statement by taking $\varepsilon_{0}=\exp(-(\log n_{0})4/\sigma)$. Similarly, we can obtain the second conclusion.
\end{proof}

We further extend to the case of perturbations in the imaginary part of the phase with the help of the integral properties of subharmonic functions.

\begin{lemma}\label{Lemma-GAFA}\cite[Lemma 4.1]{GS08-GAFA}
Let $0<\rho<1$ and suppose $f$ is subharmonic on
\begin{equation*}
\mathcal{A}_{\rho}:=\{x+iy:x\in \mathbb{T}, |y|<\rho\},
\end{equation*}
such that $\sup\limits_{\mathcal{A}_{\rho}}f \leq 1$ and $\int_{\mathbb{T}}f(x)dx\geq 0$. Then for any $y$, $y'$ so that $-\frac{\rho}{2}<y,y'<\frac{\rho}{2}$ one has
\begin{equation*}
\Big|\int_{\mathbb{T}}f(x+iy)dx-\int_{\mathbb{T}}f(x+iy')dx\Big|\leq C_{\rho}|y-y'|.
\end{equation*}
\end{lemma}
\begin{coro}\label{Coro3.10}
Let $\omega\in \mathbb{T}^{d}$, $z\in \partial \mathbb{D}$, and $L(\omega,z)>\gamma>0$. There exists $C=C(z)$ such that
\begin{equation*}
|L_{n}(\omega,z;y)-L_{n}(\omega,z)|\leq C\sum_{i=1}^{d}|y_{i}|
\end{equation*}
for all $|y|<h$ uniformly in $n$. Particularly, the bound holds equally true with $L$ in place of $L_{n}$.
\end{coro}
\begin{proof}
When $d=1$, let $f(x+iy)=\frac{1}{C(z)n}\log\|M_{n}(\omega,z;y)\|$. According to the above lemma,
\begin{align*}
&|L_{n}(\omega,z;y)-L_{n}(\omega,z)|\\
&=\Big|\int_{\mathbb{T}}\frac{1}{n} \log\|M_{n}(\omega,z;x+iy)\|dx-\int_{\mathbb{T}}\frac{1}{n} \log\|M_{n}(\omega,z;x)\|dx\Big|\\
&=C(z)\Big|\int_{\mathbb{T}}\frac{1}{C(z)n} \log\|M_{n}(\omega,z;x+iy)\|dx-\int_{\mathbb{T}}\frac{1}{C(z)n} \log\|M_{n}(\omega,z;x)\|dx\Big|\\
&\leq C(z)C_{h}|y|.
\end{align*}
Now we check the statement for $d=2$. Let
\begin{equation*}
f(x_{1}+iy_{1};y_{2}):=\int_{\mathbb{T}}\frac{1}{C(z)n} \log\|M_{n}(\omega,z;x_{1}+iy_{1},x_{2}+iy_{2})\|dx_{2}.
\end{equation*}
By Lemma \ref{Lemma-GAFA},
\begin{equation*}
|f(x_{1}+iy_{1};y_{2})-f(x_{1}+iy_{1};0)|\leq C|y_{2}|.
\end{equation*}
Fix $y_{2}$. $f(x_{1}+iy_{1};y_{2})$ is a subharmonic function. Lemma \ref{Lemma-GAFA} implies
\begin{equation*}
\Big|\int_{\mathbb{T}}f(x_{1}+iy_{1};y_{2})dx_{1}-\int_{\mathbb{T}}f(x_{1};y_{2})dx_{1}\Big|\leq C|y_{1}|.
\end{equation*}
Thus,
\begin{align*}
&|L_{n}(\omega,z;y)-L_{n}(\omega,z)|\\
&=\Big|\int_{\mathbb{T}^{2}}\frac{1}{n} \log\|M_{n}(\omega,z;x+iy)\|dx-\int_{\mathbb{T}^{2}}\frac{1}{n} \log\|M_{n}(\omega,z;x)\|dx\Big|\\
&=C(z)\Big|\int_{\mathbb{T}}\int_{\mathbb{T}}\frac{1}{C(z)n}\log\|M_{n}(\omega,z;x_{1}+iy_{1},x_{2}+iy_{2})\|dx_{2}dx_{1}\\
&\quad-\int_{\mathbb{T}}\int_{\mathbb{T}}\frac{1}{C(z)n}\log\|M_{n}(\omega,z;x_{1},x_{2})\|dx_{2}dx_{1}\Big|\\
&=C(z)\Big|\int_{\mathbb{T}}f(x_{1}+iy_{1};y_{2})dx_{1}-\int_{\mathbb{T}}f(x_{1};0)dx_{1}\Big|\\
&=C(z)\Big|\int_{\mathbb{T}}(f(x_{1}+iy_{1};y_{2})-f(x_{1}+iy_{1};0))dx_{1}+\int_{\mathbb{T}}(f(x_{1}+iy_{1};0)-f(x_{1};0))dx_{1}\Big|\\
&\leq C(z)\Big(\int_{\mathbb{T}}|f(x_{1}+iy_{1};y_{2})-f(x_{1}+iy_{1};0)|dx_{1}+\int_{\mathbb{T}}|f(x_{1}+iy_{1};0)-f(x_{1};0)|dx_{1}\Big)\\
&\leq C(|y_{1}+y_{2}|).
\end{align*}

For general $d$, the proof is similar with the aid of induction over $d$.
\end{proof}

The following lemma establishes the upper bound estimate of the logarithm of the transfer matrix norm.

\begin{lemma}\label{Lemma3.11}
Assume $\omega\in\mathbb{T}^{d}(p,q)$, $z\in \partial \mathbb{D}$, and $L(\omega,z)>\gamma>0$. Then for all $n\geq 1$,
\begin{equation}\label{Lem3.11-(1)}
\sup\limits_{x\in \mathbb{T}^{d}}\log\|M_{n}(\omega,z;x)\|\leq nL_{n}(\omega,z)+Cn^{1-\tau},
\end{equation}
where $C=C(p,q,z,\gamma)$ and $\tau$ as in LDT.
\end{lemma}
\begin{proof}
To check \eqref{Lem3.11-(1)}, we only need to consider the case that $n$ is large enough because for smaller $n$ we can take $C$ large enough.

According to LDT,
\begin{equation}\label{Lem3.11-(2)}
\mathrm{mes}\{x\in\mathbb{T}^{d}:|\log\|M_{n}(\omega,z;x+iy)\|-nL_{n}(\omega,z;y)|>n^{1-\tau}\}<\exp(-C_{0}n^{\sigma}).
\end{equation}
By Corollary \ref{Coro3.10}, we can get
\begin{equation}\label{Lem3.11-(3)}
\mathrm{mes}\{x\in\mathbb{T}^{d}:|\log\|M_{n}(\omega,z;x+iy)\|-nL_{n}(\omega,z)|>2n^{1-\tau}\}<\exp(-C_{0}n^{\sigma})
\end{equation}
provided $|y|<\frac{1}{n}$. With the help of the sub-mean value property applicable to subharmonic functions, we obtain that
\begin{equation}\label{Lem3.11-(4)}
\log\|M_{n}(\omega,z;x)\|\leq (\pi r^{2})^{-d}\int_{\mathcal{P}}\log \|M_{n}(\omega,z;\xi+iy)\|d\xi dy,
\end{equation}
where the polydisk $\mathcal{P}=\mathop{\prod}\limits_{j=1}^{d}D(x_{j},r)$, $x=(x_{1},\ldots,x_{d})$, $r=\frac{1}{n}$.\\
Let $\mathcal{A}_{y}\subset \mathbb{T}^{d}$ be the set in \eqref{Lem3.11-(3)}, i.e.,
\begin{equation*}
\mathcal{A}_{y}=\{x\in\mathbb{T}^{d}:|\log\|M_{n}(\omega,z;x+iy)\|-nL_{n}(\omega,z)|>2n^{1-\tau}\}.
\end{equation*}
Let $\mathcal{A}=\{(\xi,y)\in [0,1]^{d}\times (-r,r)^{d}:\xi\in \mathcal{A}_{y}\}$. From \eqref{Lem3.11-(3)}, we have
\begin{equation}\label{Lem3.11-(5)}
(\pi r^{2})^{-d}\int_{\mathcal{P}\backslash\mathcal{A}}\log \|M_{n}(\omega,z;\xi+iy)\|d\xi dy\leq nL_{n}(\omega,z)+2n^{1-\tau}.
\end{equation}
Due to \eqref{log-Mn},
\begin{small}
\begin{equation}\label{Lem3.11-(6)}
(\pi r^{2})^{-d}\int_{\mathcal{P}\cap\mathcal{A}}\log \|M_{n}(\omega,z;\xi+iy)\|d\xi dy\leq (\pi r^{2})^{-d}C(z)n\mathrm{mes}(\mathcal{A})<\exp(-C_{0}n^{\sigma}/2).
\end{equation}
\end{small}
Combining \eqref{Lem3.11-(4)}, \eqref{Lem3.11-(5)} and \eqref{Lem3.11-(6)}, we obtain the conclusion.
\end{proof}

Lemma \ref{Lemma3.11} establishes the global upper bound of the logarithm of the transfer matrix norm, and the following corollary extends this result to the case of small perturbations.

\begin{coro}\label{Coro3.12}
Suppose $\omega_{0}\in \mathbb{T}(p,q)$, $z_{0}\in \partial \mathbb{D}$ and $L(\omega_{0},z_{0})>\gamma>0$. Let $\sigma$, $\tau$ be as in LDT. Then for all $n\geq N_{0}(p,q,z_{0},\gamma)$ and $(\omega,z,y)\in \mathbb{C}^{d}\times \partial \mathbb{D}\times \mathbb{R}^{d}$ such that
\begin{equation*}
|y|<\frac{1}{n}, \quad |\omega-\omega_{0}|,\quad |z-z_{0}|<\exp(-(\log n)^{8/\sigma}),
\end{equation*}
we have
\begin{equation*}
\sup\limits_{x\in \mathbb{T}^{d}}\log\|M_{n}(\omega,z;x+iy)\|\leq nL_{n}(\omega_{0},z_{0})+Cn^{1-\tau},
\end{equation*}
where $C=C(p,q,z_{0},\gamma)$.
\end{coro}
\begin{proof}
Take $\omega, z, y$ satisfying the assumptions. Let $A=2/\sigma$. According to Lemma \ref{Lemma3.11} and Lemma \ref{Lemma3.6}, we obtain
\begin{equation*}
\log \|M_{n}(\omega,z;x+iy)\|\leq nL_{n}(\omega_{0},z_{0};y)+Cn^{1-\tau}
\end{equation*}
for all $x$ outside a set $B_{y}\in \mathbb{T}^{d}$ with $\mathrm{mes}(B_{y})<\exp(-(\log n)^{2})$.

Similar to the proof of Lemma \ref{Lemma3.11}, we can get the conclusion.
\end{proof}

Then we can further obtain the norm estimate of the difference between transfer matrices.

\begin{coro}\label{Coro3.13}
Assume $\omega_{0}\in \mathbb{T}(p,q)$, $z_{0}\in\partial \mathbb{D}$, and $L(\omega_{0},z_{0})>\gamma>0$. Let $\sigma, \tau$ be as in LDT. For all $n>N_{0}(p,q,z_{0},\gamma)$ and $(\omega_{i},z_{i},x_{i})\in \mathbb{C}^{d}\times \partial \mathbb{D}\times \mathbb{R}^{d}$ such that
\begin{equation*}
|\mathrm{Im} x_{i}|<\frac{1}{n},\, |\omega_{i}-\omega_{0}|,\, |z_{i}-z_{0}|<\exp(-(\log n)^{8/\sigma}),\,i=1,2,
\end{equation*}
we have
\begin{equation*}
\|M_{n}(\omega_{1},z_{1};x_{1})-M_{n}(\omega_{2},z_{2};x_{2})\|\leq (|\omega_{1}-\omega_{2}|+|z_{1}-z_{2}|+|x_{1}-x_{2}|)\exp(nL(\omega_{0},z_{0})+Cn^{1-\tau}),
\end{equation*}
where $C=C(p,q,z_{0},\gamma)$. Particularly,
\begin{small}
\begin{equation}\label{Coro3.12-(1)}
\big|\log\|M_{n}(\omega_{1},z_{1};x_{1})\|-\log \|M_{n}(\omega_{2},z_{2};x_{2})\|\big|\leq (|\omega_{1}-\omega_{2}|+|z_{1}-z_{2}|+|x_{1}-x_{2}|)\exp(nL(\omega_{0},z_{0})+Cn^{1-\tau}),
\end{equation}
\end{small}
provided the right-hand side of \eqref{Coro3.12-(1)} is less than $\frac{1}{2}$.
\end{coro}
\begin{proof}
Similar to the proof of Lemma \ref{Lemma3.5}, based on the Corollary \ref{Coro3.12}, the conclusion follows.
\end{proof}
\subsection{LDT Estimate for the Entries of the Monodromy}
Define the unitary matrices
\begin{equation*}
\Theta_{n}=
\left(
\begin{array}{cc}
\overline{\alpha}_{n}&\rho_{n}\\
\rho_{n}&-\alpha_{n}
\end{array}
\right).
\end{equation*}
Then it is valid to factorize the matrix $\mathcal{C}$ as follows,
\begin{equation*}
\mathcal{C}=\mathcal{L}_{+}\mathcal{M}_{+},
\end{equation*}
where
\begin{equation*}
\mathcal{L}_{+}=\left(
\begin{matrix}
\Theta_0 &~ & ~\\
~& \Theta_2 & ~\\
~ & ~& \ddots
\end{matrix}
\right),\quad
\mathcal{M}_{+}=\left(
\begin{matrix}
\mathbf{1} &~ & ~\\
~& \Theta_1 & ~\\
~ & ~& \ddots
\end{matrix}
\right).
\end{equation*}

Similarly, the extended CMV matrix can be written as
\begin{equation*}
\mathcal{E}=\mathcal{L}\mathcal{M},
\end{equation*}
where
\begin{equation*}
\mathcal{L}=\bigoplus_{j\in \mathbb{Z}}\Theta_{2j}, \quad \mathcal{M}=\bigoplus_{j\in \mathbb{Z}}\Theta_{2j+1}.
\end{equation*}

We let $\mathcal{E}_{[a,b]}$ denote the restriction of an extended CMV matrix to the finite interval $[a,b]$, defined by
\begin{equation*}
\mathcal{E}_{[a,b]}=P_{[a,b]}\mathcal{E}(P_{[a,b]})^{*},
\end{equation*}
where $P_{[a,b]}$ is the projection $\ell^{2}(\mathbb{Z})\rightarrow \ell^{2}([a,b])$. $\mathcal{L}_{[a,b]}$ and $\mathcal{M}_{[a,b]}$ are defined similarly.

However, the matrix $\mathcal{E}_{[a,b]}$ will no longer be unitary due to the fact that $|\alpha_{a}|<1$ and $|\alpha_{b}|<1$. Thus, we need to modify the boundary conditions briefly. With $\beta,\eta \in \partial \mathbb{D}$, define the sequence of Verblunsky coefficients
\begin{equation*}
\hat{\alpha}_{n}=
\begin{cases}
\beta, \quad & n =a;\\
\eta,& n =b;\\
\alpha_{n},& n \notin  \{a,b\}.
\end{cases}
\end{equation*}
Without loss of generality, we restrict the extended CMV matrix to the interval $[0,n-1]$. We denote the characteristic determinant of matrix $\mathcal{E}^{\beta,\eta}_{[0,n-1]}$ as $\varphi^{\beta,\eta}_{[0,n-1]}(\omega,z;x)=\det \big(z-\mathcal{E}^{\beta,\eta}_{[0,n-1]}\big)$. According to the results in \cite[Theorem 2]{Wang-JMAA}, the relation between this characteristic determinant and the $n$-step transfer matrix is
\begin{equation}\label{relation}
M_n(\omega,z;x)=(\sqrt{z})^{-n}\Big(\prod_{j=0}^{n-1}\frac{1}{\rho_j}\Big)
\left(
\begin{matrix}
z\varphi^{\beta,\eta}_{[1,n-1]} & \frac{z\varphi^{\beta,\eta}_{[1,n-1]}-\varphi^{\beta,\eta}_{[0,n-1]}}{\alpha_{-1}}\\
z\big(\frac{z\varphi^{\beta,\eta}_{[1,n-1]}-\varphi^{\beta,\eta}_{[0,n-1]}}{\alpha_{-1}}\big)^* &(\varphi^{\beta,\eta}_{[1,n-1]})^{*}
\end{matrix}
\right).
\end{equation}
In order to obtain the large deviation theorem for the entries of the monodromy matrix, we need to estimate the above characteristic determinant.

\begin{remark}\label{Remark3.14}
According to the above matrix, we have the following two conclusions
$$\sup\limits_{x\in \mathbb{T}^{d}}\log|\varphi^{\beta,\eta}_{[0,n-1]}(x+iy)|\leq nL_{n}(\omega_{0},z_{0})+Cn^{1-\tau}$$
and
\begin{align}
&\big|\log|\varphi^{\beta,\eta}_{[0,n-1]}(\omega_{1},z_{1};x_{1})|-\log |\varphi^{\beta,\eta}_{[0,n-1]}(\omega_{2},z_{2};x_{2})|\big|\nonumber\\
&\leq (|\omega_{1}-\omega_{2}|+|z_{1}-z_{2}|+|x_{1}-x_{2}|)\frac{\exp(nL(\omega_{0},z_{0})+Cn^{1-\tau})}{\mathop{\max}\limits_{i}|\varphi^{\beta,\eta}_{[1,n-1]}(\omega_{i},z_{i};x_{i})|}\label{Remark3.14-(1)}
\end{align}
provided the right-hand side of \eqref{Remark3.14-(1)} is less than $\frac{1}{2}$, which can be derived from Corollary \ref{Coro3.12} and Corollary \ref{Coro3.13}.
\end{remark}

To ensure the characteristic determinant does not vanish on a large set, we first establish a lower bound on its measure of small values using properties of real-analytic functions and matrix perturbations.

\begin{lemma}\label{Lemma3.15}\cite[Lemma 11.4]{GS01-Annals}
Suppose $V$ is a nonconstant real-analytic function on $Q_{0}=[-2,2]^{d}$ with $\sup\limits_{Q_{0}}|V|\leq 1$. Then there exists $\varepsilon=\varepsilon(V,d)>0$ and $C=C(V,d)$ so that
\begin{equation*}
\mathrm{mes}\{(x_{1},\ldots,x_{d})\in [-1,1]^{d}:|V(x_{1},\ldots,x_{d})-E|<t\}\leq Ct^{\varepsilon}
\end{equation*}
for all $-1\leq E\leq 1$ and $0<t<1$.
\end{lemma}

\begin{lemma}\label{Lemma3.16}
Fix some $z$ and let $\varphi_{[0,n-1]}^{\beta,\eta}=\det \big(z-\mathcal{E}^{\beta,\eta}_{[0,n-1]}\big)$. Then
\begin{equation}\label{lem3.16-(1)}
\mathrm{mes}\{x\in \mathbb{T}^{d}:|\varphi_{[0,n-1]}^{\beta,\eta}|\leq \exp(-n^{d+2})\}\leq \exp(-n)
\end{equation}
for large $n$.
\end{lemma}
\begin{proof}
Suppose \eqref{lem3.16-(1)} fails and set $u(x)=\log\big|\varphi_{[0,n-1]}^{\beta,\eta}\big|$. According to Lemma \ref{Lemma3.14} with $M\leq C(\alpha_{n},z)n$, $L=n^{d+2}$ and $\delta=\exp(-n)$, then we have
\begin{equation*}
\sup\limits_{\mathbb{T}^{d}}u\leq C_{\rho,d}C(\alpha_{n},z)n-\frac{n^{d+2}}{C_{\rho,d}\log^{d}(C_{d}/\delta)}\leq-C_{1}n^{2}
\end{equation*}
for some small constant $C_{1}$. Thus,
\begin{equation}\label{Lem3.16-(2)}
\sup\limits_{\mathbb{T}^{d}}\big|\varphi_{[0,n-1]}^{\beta,\eta}\big|\leq \exp(-C_{1}n^{2}).
\end{equation}
On the other hand, we can write $\varphi_{[0,n-1]}^{\beta,\eta}$ as $\det(D_{[0,n-1]}+B_{[0,n-1]})$,
where
\begin{equation*}
D_{[0,n-1]}=\mathrm{diag}\{z+\overline{\alpha}_{0}\alpha_{-1},z+\overline{\alpha}_{1}\alpha_{0},\ldots,z+\overline{\alpha}_{n-1}\alpha_{n-2}\}.
\end{equation*}
Due to $|\alpha_{j}|, \rho_{j}<1$, we know that $\|B_{[0,n-1]}\|<3$.

Therefore,
\begin{align*}
\Big|\frac{1}{n}\log\big|\varphi_{[0,n-1]}^{\beta,\eta}\big|\Big|&=\Big|\frac{1}{n}\log\big|\det(D_{[0,n-1]}+B_{[0,n-1]}) \big|\Big|\\
&\leq\Big|\frac{1}{n}\log\big|\det D_{[0,n-1]}\big|\Big|+\Big|\frac{1}{n}\log\big|\det(I+D^{-1}_{[0,n-1]}B_{[0,n-1]}) \big|\Big|\\
&=\Big|\frac{1}{n}\sum_{j=0}^{n-1}\log|z+\overline{\alpha}_{j}\alpha_{j-1}|\Big|+\Big|\frac{1}{n}\log\big|\det(I+D^{-1}_{[0,n-1]}B_{[0,n-1]}) \big|\Big|.
\end{align*}
Let $\hat{\alpha}:=\sup\limits_{j\in \mathbb{Z},\, x\in \mathbb{T}^{d}}|\overline{\alpha}_{j}\alpha_{j-1}|$. Now we consider the following two cases.\\
\textbf{Case 1}. Suppose $\hat{\alpha}\geq \frac{1}{2}$. According to Lemma \ref{Lemma3.15}, we have that
\begin{equation*}
\mathrm{mes}\{x\in \mathbb{T}^{d}:|\overline{\alpha}_{j}\alpha_{j-1}-1|<e^{-\rho}\}<C e^{-\varepsilon\rho}
\end{equation*}
for some $\rho>0$.

Since
\begin{equation*}
\|D^{-1}_{[0,n-1]}\|\leq N_{0}^{-M}(z+\overline{\alpha}_{j}\alpha_{j-1})|^{-1},
\end{equation*}
we get
\begin{align*}
\mathrm{mes}\Big\{x\in \mathbb{T}^{d}:\|D^{-1}_{[0,n-1]}\|>\frac{2}{3}\Big\}&\leq n \mathrm{mes}\Big\{x\in \mathbb{T}^{d}:|N_{0}^{-M}(z+\overline{\alpha}(Tx)\alpha(x))|<\frac{3}{2}N_{0}^{-M}\Big\}\\
&=n \mathrm{mes}\Big\{x\in \mathbb{T}^{d}:|N_{0}^{-M}z^{-1}\overline{\alpha}(Tx)\alpha(x)-N_{0}^{-M}|<\frac{3}{2}N_{0}^{-M}\Big\}\\
&\leq CnN_{0}^{-\varepsilon M}.
\end{align*}
Thus,
\begin{equation*}
\mathrm{mes}\{x\in \mathbb{T}^{d}:\|D^{-1}_{[0,n-1]}B_{[0,n-1]}\|\geq 2\}\leq CnN_{0}^{-\varepsilon M}.
\end{equation*}
Then we have
\begin{align*}
\Big|\frac{1}{n}\log\big|\varphi_{[0,n-1]}^{\beta,\eta}\big|\Big|&\leq\Big|\frac{1}{n}\log\big|\det D_{[0,n-1]}\big|\Big|+\Big|\frac{1}{n}\log\big|\det(I+D^{-1}_{[0,n-1]}B_{[0,n-1]}) \big|\Big|\\
&\leq \log(1+\hat{\alpha})+\log 3=\log(3(1+\hat{\alpha}))=:C(\hat{\alpha}).
\end{align*}
That is, $e^{-C(\hat{\alpha})n}\leq |\varphi_{[0,n-1]}^{\beta,\eta}|\leq e^{C(\hat{\alpha})n}$, which contradicts \eqref{Lem3.16-(2)}.\\
\textbf{Case 2}. Suppose $\hat{\alpha}< \frac{1}{2}$. According to the definition of $D_{[0,n-1]}$,
\begin{equation*}
\Big|\frac{1}{n}\log|\det D_{[0,n-1]}|\Big|=\Big|\frac{1}{n}\sum_{j=0}^{n-1}\log|z+\overline{\alpha}_{j}\alpha_{j-1}|\Big|.
\end{equation*}
If $|z+\overline{\alpha}_{j}\alpha_{j-1}|<1$, then
\begin{equation*}
|z+\overline{\alpha}_{j}\alpha_{j-1}|\geq |z|-|\overline{\alpha}_{j}\alpha_{j-1}|\geq |z|-\hat{\alpha}\geq \frac{1}{2}.
\end{equation*}
If $|z+\overline{\alpha}_{j}\alpha_{j-1}|>1$, then
\begin{equation*}
|z+\overline{\alpha}_{j}\alpha_{j-1}|\leq |z|+|\overline{\alpha}_{j}\alpha_{j-1}|\leq 2.
\end{equation*}
Thus, $\Big|\frac{1}{n}\log|\det D_{[0,n-1]}|\Big|<\log 2<2$ and
\begin{align*}
\Big|\frac{1}{n}\log\big|\varphi_{[0,n-1]}^{\beta,\eta}\big|\Big|&\leq\Big|\frac{1}{n}\log\big|\det D_{[0,n-1]}\big|\Big|+\Big|\frac{1}{n}\log\big|\det(I+D^{-1}_{[0,n-1]}B_{[0,n-1]}) \big|\Big|\\
&\leq 2+\log 7<4.
\end{align*}
This implies that $e^{-4n}\leq |\varphi_{[0,n-1]}^{\beta,\eta}|\leq e^{4n}$, which contradicts \eqref{Lem3.16-(2)}.

In conclusion, $\mathrm{mes}\{x\in \mathbb{T}^{d}:|\varphi_{[0,n-1]}^{\beta,\eta}|\leq \exp(-n^{d+2})\}\leq \exp(-n)$.
\end{proof}

Extending the previous result, we generalize the lower bound to smaller thresholds of the characteristic determinant, ensuring robustness across a range of small values that are relevant for LDT.

\begin{lemma}\label{Lemma3.17}
Let $N_{0}$ be as in Lemma \ref{Lemma3.16} and $0<\theta<1$. Fix some $z$ and let $\varphi_{[0,n-1]}^{\beta,\eta}=\det \big(z-\mathcal{E}^{\beta,\eta}_{[0,n-1]}\big)$. Then for any $N_{0}\leq n<N^{\frac{1-\theta}{2}}$,
\begin{equation}\label{Lem3.17-(1)}
\mathrm{mes}\{x\in \mathbb{T}^{d}:|\varphi_{[0,n-1]}^{\beta,\eta}|\leq \exp(-N^{1-\theta})\}\leq \exp(-N^{\frac{1-\theta}{d}}n^{-\frac{2}{d}}),
\end{equation}
where $N$ is large.
\end{lemma}
\begin{proof}
Suppose \eqref{Lem3.17-(1)} fails for some $n$ and $N$. Applying Lemma \ref{Lemma3.14} with $M=n$, $\delta=\exp(-N^{\frac{1-\theta}{d}}n^{-\frac{2}{d}})$ and $L=N^{1-\theta}$, we obtain
\begin{equation*}
\sup\limits_{\mathbb{T}^{d}}u\leq Cn-\frac{N^{1-\theta}}{C\log^{d}(C/\theta)}\leq -C_{1}n^{2}
\end{equation*}
for some constant $C_{1}$. In other words,
\begin{equation*}
\sup\limits_{\mathbb{T}^{d}}|\varphi_{[0,n-1]}^{\beta,\eta}| \leq \exp( -C_{1}n^{2}).
\end{equation*}
Then this leads to a contradiction in the same way as the derivation of the contradiction in the  proof of Lemma \ref{Lemma3.16}. Similar to the proof of Lemma \ref{Lemma3.16}, the conclusion follows.
\end{proof}
Now we introduce the properties of matrices in $\mathbb{SL}(2,\mathbb{C})$. For any $A\in \mathbb{SL}(2,\mathbb{C})$, there are unit vectors $\underline{u}^{+}_{A}$, $\underline{u}^{-}_{A}$, $\underline{v}^{+}_{A}$, $\underline{v}^{-}_{A}$ such that $A\underline{u}^{+}_{A}=\|A\|\underline{v}^{+}_{A}$, $A\underline{u}^{-}_{A}=\|A\|^{-1}\underline{v}^{-}_{A}$. Moreover, $\underline{u}^{+}_{A}\bot\underline{u}^{-}_{A}$ and $\underline{v}^{+}_{A}\bot\underline{v}^{-}_{A}$. In what follows, we denote that $a\wedge b:=\min\{a,b\}$ for positive numbers $a, b$ and $\varphi\wedge\psi$ stands for the exterior product of vectors $\varphi, \psi$.

\begin{lemma}\label{Lemma3.18}
From the Theorem \ref{LDT-matrix}, we have
\begin{equation}\label{Lem3.18-(1)}
\mathrm{mes}\{x\in \mathbb{T}^{d}: |\log\|M_{n}(\omega,z;x)\|-nL_{n}(\omega,z)|> n^{1-\tau}\}<\exp(-C_{0}n^{\sigma}).
\end{equation}
Then
\begin{align}
&\mathrm{mes}\{x\in \mathbb{T}^{d}: |\varphi_{[0,n-1]}^{\beta,\eta}(x)|+|\varphi_{[0,n-1]}^{\beta,\eta}(T^{j_{1}}x)|+
|\varphi_{[0,n-1]}^{\beta,\eta}(T^{j_{2}}x)|\leq \exp(nL_{n}(\omega,z)-100n^{1-\sigma})\}\nonumber\\
&\leq \exp(-n^{\frac{1-\sigma}{2d}}\wedge\sigma)\label{Lem3.18-(2)}
\end{align}
for any $l_{0}=l_{0}(T,z,\sigma)\leq j_{1}\leq j_{1}+l_{0}\leq j_{2}\leq n^{\frac{1-\sigma}{8}}$ and $n>N(T,C_{0},z,\sigma)$. Moreover, to obtain \eqref{Lem3.18-(2)} for some $n$ only requires \eqref{Lem3.18-(1)} with the same $n$.
\end{lemma}
\begin{proof}
From \eqref{relation}, the transfer matrix $M_{n}(\omega,z;x)\in \mathbb{S}\mathbb{U}(1,1)$, then we can verify that
\begin{equation}\label{relation-SL(2,R)}
T_n(\omega,z;x)=Q^{*}M_n(\omega,z;x)Q=(\sqrt{z})^{-n}\Big(\prod_{j=0}^{n-1}\frac{1}{\rho_j}\Big)
\left(
\begin{matrix}
t_{1} & t_{2}\\
t_{3} & t_{4}
\end{matrix}
\right)\in \mathbb{SL}(2,\mathbb{R}),
\end{equation}
where
\begin{align*}
&t_{1}=\frac{1}{2}z\varphi^{\beta,\eta}_{[1,n-1]}+\frac{1}{2}z\Big(\frac{z\varphi^{\beta,\eta}_{[1,n-1]}-\varphi^{\beta,\eta}_{[0,n-1]}}{\alpha_{-1}}\Big)^*
+\frac{1}{2}\frac{z\varphi^{\beta,\eta}_{[1,n-1]}-\varphi^{\beta,\eta}_{[0,n-1]}}{\alpha_{-1}}+\frac{1}{2}(\varphi^{\beta,\eta}_{[1,n-1]})^{*},\\
&t_{2}=-\frac{i}{2}z\varphi^{\beta,\eta}_{[1,n-1]}-\frac{i}{2}z\Big(\frac{z\varphi^{\beta,\eta}_{[1,n-1]}-\varphi^{\beta,\eta}_{[0,n-1]}}{\alpha_{-1}}\Big)^*
+\frac{i}{2}\frac{z\varphi^{\beta,\eta}_{[1,n-1]}-\varphi^{\beta,\eta}_{[0,n-1]}}{\alpha_{-1}}+\frac{i}{2}(\varphi^{\beta,\eta}_{[1,n-1]})^{*},\\
&t_{3}=\frac{i}{2}z\varphi^{\beta,\eta}_{[1,n-1]}-\frac{i}{2}z\Big(\frac{z\varphi^{\beta,\eta}_{[1,n-1]}-\varphi^{\beta,\eta}_{[0,n-1]}}{\alpha_{-1}}\Big)^*
+\frac{i}{2}\frac{z\varphi^{\beta,\eta}_{[1,n-1]}-\varphi^{\beta,\eta}_{[0,n-1]}}{\alpha_{-1}}-\frac{i}{2}(\varphi^{\beta,\eta}_{[1,n-1]})^{*},\\
&t_{4}=\frac{1}{2}z\varphi^{\beta,\eta}_{[1,n-1]}-\frac{1}{2}z\Big(\frac{z\varphi^{\beta,\eta}_{[1,n-1]}-\varphi^{\beta,\eta}_{[0,n-1]}}{\alpha_{-1}}\Big)^*
-\frac{1}{2}\frac{z\varphi^{\beta,\eta}_{[1,n-1]}-\varphi^{\beta,\eta}_{[0,n-1]}}{\alpha_{-1}}+\frac{1}{2}(\varphi^{\beta,\eta}_{[1,n-1]})^{*}
\end{align*}
and ``$*$" stands for the Szeg\H{o} dual of the corresponding polynomial.

According to \cite[Section 3]{Wang-JMAA}, we know that
\begin{equation}\label{Lem3.18-(0)}
\varphi^{\beta,\eta}_{[0,n-1]}=(z+\overline{\alpha}_{0}\alpha_{-1})\varphi^{\beta,\eta}_{[1,n-1]}-\rho_{0}\alpha_{-1}\det \mathcal{P}_{n-1},
\end{equation}
where
\begin{equation*}
\mathcal{P}_{n-1}=
\left(
\begin{array}{ccccccc}
-\overline{\alpha}_{1}\rho_{0}&-\rho_{1}\rho_{0}&  &  & & \\
-\overline{\alpha}_{2}\rho_{1}&z+\overline{\alpha}_{2}\alpha_{1}&-\overline{\alpha}_{3}\rho_{2} &-\rho_{3}\rho_{2} & &\\
-\rho_{2}\rho_{1} &\rho_{2}\alpha_{1}&z+\overline{\alpha}_{3}\alpha_{2}&\rho_{3}\alpha_{2}& & \\
  & & -\overline{\alpha}_{4}\rho_{3}&z+\overline{\alpha}_{4}\alpha_{3}& &\\
  & & & & \ddots &\\
  & & & & &z+\overline{\eta}\alpha_{n-2}
\end{array}
\right).
\end{equation*}
Then for unit vectors $\underline{e}_{1}=(1,0)^{T}$ and $\underline{e}_{2}=(0,1)^{T}$, we can derive
\begin{equation*}
T_n(\omega,z;x)\underline{e}_{1}\wedge\underline{e}_{2}=(\sqrt{z})^{-n}\Big(\prod_{j=0}^{n-1}\frac{1}{\rho_j}\Big)t_{1}.
\end{equation*}
Due to the fact that $(z\varphi^{\beta,\eta}_{[1,n-1]})^{*}=(\varphi^{\beta,\eta}_{[1,n-1]})^{*}$ and
$$\Big(\frac{z\varphi^{\beta,\eta}_{[1,n-1]}-\varphi^{\beta,\eta}_{[0,n-1]}}{\alpha_{-1}}\Big)^{*}=z\Big(\frac{z\varphi^{\beta,\eta}_{[1,n-1]}-\varphi^{\beta,\eta}_{[0,n-1]}}{\alpha_{-1}}\Big)^{*},$$
we have
\begin{align*}
|T_n(\omega,z;x)\underline{e}_{1}\wedge\underline{e}_{2}|&\leq \Big(\prod_{j=0}^{n-1}\frac{1}{\rho_j}\Big)\Big(|\varphi^{\beta,\eta}_{[1,n-1]}|+\Big|\frac{z\varphi^{\beta,\eta}_{[1,n-1]}-\varphi^{\beta,\gamma}_{[0,n-1]}}{\alpha_{-1}} \Big|\Big )\\
&=\Big(\prod_{j=0}^{n-1}\frac{1}{\rho_j}\Big)\Big(|\varphi^{\beta,\eta}_{[1,n-1]}|+|\rho_{0}\det \mathcal{P}_{n-1}-\overline{\alpha}_{0}\varphi^{\beta,\eta}_{[1,n-1]}|\Big)\\
&\leq \Big(\prod_{j=0}^{n-1}\frac{1}{\rho_j}\Big)(2|\varphi^{\beta,\eta}_{[1,n-1]}|+|\det \mathcal{P}_{n-1}|).
\end{align*}
Therefore, from \eqref{Lem3.18-(0)} we have
\begin{equation*}
|\varphi^{\beta,\eta}_{[0,n-1]}|\leq 2|\varphi^{\beta,\eta}_{[1,n-1]}|+|\det \mathcal{P}_{n-1}|\leq \Big(\prod_{j=0}^{n-1}\frac{1}{\rho_j}\Big)(2|\varphi^{\beta,\eta}_{[1,n-1]}|+|\det \mathcal{P}_{n-1}|).
\end{equation*}
Suppose
\begin{equation*}
|T_n(\omega,z;x)\underline{e}_{1}\wedge\underline{e}_{2}|\leq \Big(\prod_{j=0}^{n-1}\frac{1}{\rho_j}\Big)(2|\varphi^{\beta,\eta}_{[1,n-1]}|+|\det \mathcal{P}_{n-1}|)<\exp(nL_{n}(z)-100n^{1-\sigma}).
\end{equation*}
For convenience, we write $T_n(\omega,z;x)$ as $T_n$. Since
\begin{equation*}
T_{n}\underline{u}^{+}_{n}\wedge \underline{e}_{2}(\underline{u}^{+}_{n}\cdot\underline{e}_{1})+T_{n}\underline{u}^{-}_{n}\wedge \underline{e}_{2}(\underline{u}^{-}_{n}\cdot\underline{e}_{1})=\|T_{n}\|\underline{v}^{+}_{n}\wedge \underline{e}_{2}(\underline{u}^{+}_{n}\cdot\underline{e}_{1})+\|T_{n}\|^{-1}\underline{v}^{-}_{n}\wedge \underline{e}_{2}(\underline{u}^{-}_{n}\cdot\underline{e}_{1}),
\end{equation*}
one can obtain that
\begin{equation*}
\|T_{n}\||\underline{u}^{+}_{n}\cdot\underline{e}_{1}||\underline{v}^{+}_{n}\wedge \underline{e}_{2}|
-\|T_{n}\|^{-1}|\underline{u}^{-}_{n}\cdot\underline{e}_{1}||\underline{v}^{-}_{n}\wedge \underline{e}_{2}|\leq \exp(nL_{n}(z)-100n^{1-\sigma}).
\end{equation*}
Then the proof of \eqref{Lem3.18-(2)} is the same as the proof of the Lemma 2.9 in \cite{GS08-GAFA}, we just need to replace the use of $f_{N}$ there by our $\varphi^{\beta,\eta}_{[0,n-1]}$ and replace $N$ by $n$.
\end{proof}

Next, we establish a lower bound on the average of the logarithm of the characteristic determinant, ensuring it is close to the finite scale Lyapunov exponent.

\begin{lemma}\label{Lemma3.19}
There exists some constant $\kappa>0$ (depending on $\sigma$ in LDT) such that
\begin{equation}\label{Lem3.19-(1)}
\int_{\mathbb{T}^{d}}\frac{1}{n}\log|\varphi^{\beta,\eta}_{[0,n-1]}(x)|dx>L_{n}(z)-n^{-\kappa}
\end{equation}
for $n\geq N(T,z,\sigma)$.
\end{lemma}
\begin{proof}
Replace $f_{N}(x)$ in the proof of Lemma 2.10 in \cite{GS08-GAFA} by $\varphi^{\beta,\eta}_{[0,n-1]}(x)$, with the aid of the Lemma \ref{Lemma3.17} and Lemma \ref{Lemma3.18}, one can get the statement.
\end{proof}

\begin{lemma}\label{Lemma3.20}
From the Lemma \ref{Lemma3.11}, we have the upper bound
\begin{equation}\label{Lem3.20-(1)}
\sup\limits_{\mathbb{T}^{d}}\log\|M_{n}(\omega,z;x)\|\leq nL_{n}(\omega,z)+Cn^{1-\tau}.
\end{equation}
Fix a large positive integer $n\geq N_{0}$, then for some small constant $\nu$ depending on $\sigma$ and $\tau$,
\begin{equation}\label{Lem3.20-(2)}
\mathrm{mes}\Big\{x\in \mathbb{T}^{d}:\frac{1}{n}\log|\varphi^{\beta,\eta}_{[0,n-1]}(x)|<L_{n}(z)-n^{-\nu}\Big\}\leq \exp(-Cn^{\nu}).
\end{equation}
\end{lemma}
\begin{proof}
We just consider the case that $d=2$. The general case is similar.

Let $u=\frac{1}{n}\log|\varphi^{\beta,\eta}_{[0,n-1]}(x)|$, $\langle u \rangle :=\int_{\mathbb{T}^{2}}u(x)dx$. From \eqref{Lem3.19-(1)}, \eqref{Lem3.20-(1)} and \eqref{relation}, one can get
$$\langle u\rangle>L_{n}-n^{-\kappa},\quad \sup\limits_{\mathbb{T}^{2}} u<L_{n}+Cn^{-\tau}.$$
Let $g(z)=\int_{\mathbb{T}}u(z,y)dy$. Then $\sup\limits_{\mathbb{T}}g<\sup\limits_{\mathbb{T}^{2}} u<L_{n}+Cn^{-\tau}$
and $\langle g\rangle=\langle u\rangle>L_{n}-n^{-\kappa}$. This implies that
\begin{equation}\label{Lem3.20-(3)}
\|g-\langle g\rangle\|_{1}\leq C n^{-\nu}
\end{equation}
for some constant $C$ and $\nu$ (smaller than $\tau$ and $\kappa$).

Since $\sup\limits_{\mathbb{T}}g\geq \langle g\rangle >\frac{\gamma}{2}>0$ for large $n$ and $\sup\limits_{T^{2}_{h}}u\leq C$, one can conclude from the proof of  Lemma \ref{Lemma3.14} that there is the Riesz representation
\begin{equation*}
g(z)=\int_{\mathbb{T}_{h/2}}\log|z-\zeta|d\mu(\zeta)+h(z),
\end{equation*}
where
\begin{equation*}
\mu(\mathbb{T}_{h/2})+\|h\|_{L^{\infty}(\mathbb{T}_{h/2})}\leq C_{h}(m+M),
\end{equation*}
provided that $\sup\limits_{y\in \mathbb{T}}u(z,y)>-m$.

According to the definition of the norm on BMO($\mathbb{T}$), one can derive that
\begin{equation}\label{Lem3.20-(4)}
\|g\|_{BMO(\mathbb{T})}\leq C\|g-\langle g\rangle\|^{1/2}_{1}\leq Cn^{-\frac{\nu}{2}}.
\end{equation}
Let $\Omega=\{x\in \mathbb{T}:g(x)=\int_{\mathbb{T}}u(x,y)dy>L_{n}-n^{-\nu/4}\}$. By \eqref{Lem3.20-(4)}, we have
\begin{equation}\label{Lem3.20-(5)}
\mathrm{mes}(\mathbb{T}\backslash \Omega)\leq \exp(-n^{\nu/4}).
\end{equation}
Fix some $x\in \Omega$. Then
\begin{equation}\label{Lem3.20-(6)}
\sup\limits_{y\in \mathbb{T}}u(x,y)\leq \sup\limits_{\mathbb{T}^{2}}u<L_{n}+Cn^{-\tau}<\int_{\mathbb{T}}u(x,y)dy+Cn^{-\nu/4}.
\end{equation}
Since $\sup\limits_{y\in\mathbb{T}}u(x,y)\geq 0$ and $\sup\limits_{z\in\mathbb{T}_{h}}u(x,z)\leq C$, combining \eqref{Lem3.20-(6)} and the definition of the norm on  BMO($\mathbb{T}$), we have
$$\|u(x,\cdot)\|_{BMO(\mathbb{T})}\leq Cn^{-\nu/8}$$
for any $x\in \Omega$.

Therefore, from the John-Nirenberg inequality,
\begin{equation}\label{Lem3.20-(7)}
\mathrm{mes}\{y\in \mathbb{T}|u(x,y)<L_{n}-Cn^{-\nu/16}\}\leq \exp(-Cn^{\nu/16})
\end{equation}
for large $n$. Then the statement follows from \eqref{Lem3.20-(5)} and \eqref{Lem3.20-(7)} via Fubini.
\end{proof}

\begin{remark}\label{Remark3.22}
From the above lemma, we get the LDT for the determinant $\varphi^{\beta,\eta}_{[0,n-1]}(\omega,z;x)$ as follows
\begin{equation}\label{LDT2}
\mathrm{mes}\{x\in \mathbb{T}^{d}:|\log|\varphi^{\beta,\eta}_{[0,n-1]}(\omega,z;x)|-nL_{n}(\omega,z)|>n^{1-\tau}\}<\exp(-Cn^{\nu}).
\end{equation}
\end{remark}

In what follows, we denote $\mathcal{D}(z_{0},r)=\{z\in \mathbb{C}:|z-z_{0}|<r\}$.

\begin{lemma}\label{Lemma3.25}
Suppose that $x\in \mathbb{T}^{d}$, $\omega\in\mathbb{T}^{d}(p,q)$, $z\in  \partial \mathbb{D}$ and $L(\omega,z)>\gamma>0$. Let $H\gg 1$ and $\tau$, $\nu$ as in LDT. There exists $C_{0}=C_{0}(p,q,z,\gamma)$, such that for any $n\geq N(p,q,z,\gamma)$, if
\begin{equation*}
\log|\varphi^{\beta,\eta}_{[0,n-1]}(x)|\leq nL_{n}(\omega,z)-C_{0}Hn^{1-\tau},
\end{equation*}
then there exists $x'\in \mathbb{C}^{d}$, $|x'-x|\lesssim \exp(-(H+n^{\nu}/d))$ such that $\varphi^{\beta,\eta}_{[0,n-1]}(x')=0$. Furthermore,
\begin{equation*}
\|(\mathcal{E}^{\beta,\eta}_{[0,n-1]}(\omega,z;x)-z)^{-1}\|\geq C\exp(H+n^{\nu}/d)
\end{equation*}
for some $C$.
\end{lemma}
\begin{proof}
From Remark \ref{Remark3.22}, there exists $x_{0}$ outside the set in \eqref{LDT2} with $|x-x_{0}|\lesssim\exp(-n^{\nu}/d)$ such that
\begin{equation*}
\log|\varphi^{\beta,\eta}_{[0,n-1]}(x)|>nL_{n}(\omega,z)-n^{1-\tau}=:m.
\end{equation*}
Let
\begin{equation*}
\phi(\zeta)=\varphi^{\beta,\eta}_{[0,n-1]}\Big(x_{0}+\frac{C\exp(-n^{\nu}/d)\zeta}{|x-x_{0}|}(x-x_{0})\Big)
\end{equation*}
and $\zeta_{0}\in \mathcal{D}(0,\frac{1}{7})$ such that $\phi(\zeta_{0})=\varphi^{\beta,\eta}_{[0,n-1]}(x)$.

From the assumption, one has
\begin{equation*}
\log|\varphi^{\beta,\eta}_{[0,n-1]}(x)|\leq nL_{n}(\omega,z)-C_{0}Hn^{1-\tau}<nL_{n}(\omega,z)=:M.
\end{equation*}
According to the Cartan's estimate, there exists a set $\mathcal{B}\in \mathrm{Car}_{1}(H,CHn^{1-\tau})$ such that
\begin{equation*}
\log|\phi(\zeta)|>M-CH(M-m)=CHn^{1-\tau}
\end{equation*}
for all $\zeta\in \mathcal{D}(0,\frac{1}{6})\backslash \mathcal{B}$.

On the other hand,
\begin{equation*}
\log|\phi(\zeta_{0})|=\log|\varphi^{\beta,\eta}_{[0,n-1]}(x)|\leq nL_{n}(\omega,z)-C_{0}Hn^{1-\tau}.
\end{equation*}
Thus, there exists $\zeta_{j}\in \mathcal{B}$ such that $\zeta_{0}\in \mathcal{D}(\zeta_{j},r_{j})\subset \mathcal{B}$, where $r_{j}<\exp(-H)$. Then Lemma \ref{Lemma3.23} implies that there exists $\zeta'\in\mathcal{D}(\zeta_{j},r_{j})$
such that $\phi(\zeta')=0$.

Taking $x'=x_{0}+\frac{C\exp(-n^{\nu}/d)\zeta'}{|x-x_{0}|}(x-x_{0})$, then we can obtain $\varphi^{\beta,\eta}_{[0,n-1]}(x')=0$.

Due to the structure of the matrix $\mathcal{E}^{\beta,\eta}_{[0,n-1]}$, there exists some constant $C$ such that
\begin{equation*}
\|\mathcal{E}^{\beta,\eta}_{[0,n-1]}(\omega,z;x)-\mathcal{E}^{\beta,\eta}_{[0,n-1]}(\omega,z;x')\|\leq C|x-x'|.
\end{equation*}
In addition, if
\begin{equation*}
\|(\mathcal{E}^{\beta,\eta}_{[0,n-1]}(\omega,z;x')-E)^{-1}\| \|\mathcal{E}^{\beta,\eta}_{[0,n-1]}(\omega,z;x)-\mathcal{E}^{\beta,\eta}_{[0,n-1]}(\omega,z;x')\|<1,
\end{equation*}
then we have
\begin{align*}
&\|(\mathcal{E}^{\beta,\eta}_{[0,n-1]}(\omega,z;x')-E)^{-1}\| \|\mathcal{E}^{\beta,\eta}_{[0,n-1]}(\omega,z;x)-\mathcal{E}^{\beta,\eta}_{[0,n-1]}(\omega,z;x')\| \\
&=\|(\mathcal{E}^{\beta,\eta}_{[0,n-1]}(\omega,z;x')-E)^{-1}\| \|(\mathcal{E}^{\beta,\eta}_{[0,n-1]}(\omega,z;x)-E)-(\mathcal{E}^{\beta,\eta}_{[0,n-1]}(\omega,z;x')-E)\| \\
&<1
\end{align*}
and
\begin{equation*}
\|E-(\mathcal{E}^{\beta,\eta}_{[0,n-1]}(\omega,z;x')-E)^{-1}(\mathcal{E}^{\beta,\eta}_{[0,n-1]}(\omega,z;x)-E)\|<1,
\end{equation*}
where $E$ denotes the identity matrix.

Thus, $E-[E-(\mathcal{E}^{\beta,\eta}_{[0,n-1]}(\omega,z;x')-E)^{-1}(\mathcal{E}^{\beta,\eta}_{[0,n-1]}(\omega,z;x)-E)]$
is invertible and this implies that $\mathcal{E}^{\beta,\eta}_{[0,n-1]}(\omega,z;x)-E$ is invertible. Then the second statement follows.
\end{proof}
As a consequence, we have the spectral form of LDT.
\begin{lemma}\label{spectrl-form-(LDT)}
Assume $x\in \mathbb{T}^{d}$, $\omega\in \mathbb{T}^{d}(p,q)$, $z\in\partial \mathbb{D}$ and $L(\omega,z)>\gamma>0$. Let $\tau$, $\nu$ as in LDT. If $n\geq N(p,q,z,\gamma)$ and
\begin{equation*}
\|(\mathcal{E}^{\beta,\eta}_{[0,n-1]}-z)^{-1}\|\leq C\exp(n^{\nu/2}),
\end{equation*}
then
\begin{equation*}
\log|\varphi^{\beta,\eta}_{[0,n-1]}(x)|>nL_{n}(\omega,z)-n^{1-\tau/2}.
\end{equation*}
\end{lemma}
\begin{proof}
Let $n^{\nu/2}=H+\frac{n^{\nu/2}}{d}$, then $H=\frac{d-1}{d}n^{\nu/2}$. Applying the above lemma, one can obtain
\begin{align*}
\log|\varphi^{\beta,\eta}_{[0,n-1]}(x)|&>nL_{n}(\omega,z)-C_{0}Hn^{1-\tau}=nL_{n}(\omega,z)-C_{0}\frac{d-1}{d}n^{\nu/2}n^{1-\tau}\\
&>nL_{n}(\omega,z)-C_{0}\frac{d-1}{d}n^{\nu/2}n^{\frac{1}{2}-\frac{\tau}{2}}>nL_{n}(\omega,z)-C_{0}\frac{d-1}{d}n^{1-\frac{\tau}{2}}\\
&>nL_{n}(\omega,z)-n^{1-\frac{\tau}{2}}.
\end{align*}
\end{proof}
\subsection{Poisson Formula for CMV Matrix}
For $z\in \mathbb{C}$, $\beta, \eta\in\partial \mathbb{D}$, we can define the polynomials
\begin{equation*}
\varphi^{\beta,\eta}_{[a,b]}(z):=\det(z-\mathcal{E}^{\beta,\eta}_{[a,b]}), \quad \phi^{\beta,\eta}_{[a,b]}(z):=(\rho_{a}\cdots\rho_{b})^{-1}\varphi^{\beta,\eta}_{[a,b]}(z).
\end{equation*}
Note that when $a>b$, $\phi^{\beta,\eta}_{[a,b]}(z)=1$.

Since the equation $\mathcal{E}u=zu$ is equivalent to $(z\mathcal{L}^{*}-\mathcal{M})u=0$. Then the associated finite-volume Green's functions are as follows
\begin{align*}
G^{\beta,\eta}_{[a,b]}(z)=\big(z(\mathcal{L}^{\beta,\eta}_{[a,b]})^{*}-\mathcal{M}^{\beta,\eta}_{[a,b]}\big)^{-1},\\
G^{\beta,\eta}_{[a,b]}(j,k;z)=\langle \delta_{j},G^{\beta,\eta}_{[a,b]}(z)\delta_{k}\rangle,\quad j,k\in [a,b].
\end{align*}
According to \cite[Proposition 3.8]{Kruger13-IMRN} and \cite[Section B.1]{Zhu-arXiv}, for $\beta, \eta\in\partial \mathbb{D}$, the Green's function has the following expression
\begin{equation*}
|G^{\beta,\eta}_{[a,b]}(j,k;z)|=\frac{1}{\rho_{k}}\Big| \frac{\phi^{\beta,\cdot}_{[a,j-1]}(z)\phi^{\cdot,\eta}_{[k+1,b]}(z)}{\phi^{\beta,\eta}_{[a,b]}(z)}\Big|,\quad a\leq j\leq k\leq b,
\end{equation*}
where ``$\cdot$'' stands for the unchanged Verblunsky coefficient.

From \cite[Lemma 3.9]{Kruger13-IMRN}, if $u$ satisfies $\mathcal{E}u=zu$, Poisson formula reads
\begin{align*}
u(m)=&G^{\beta,\eta}_{[a,b]}(a, m;z)
\begin{cases}
(z\overline{\beta}-\alpha_{a})u(a)-\rho_{a}u(a+1), \quad & a\text{ even}\\
(z\alpha_{a}-\beta)u(a)+z\rho_{a}u(a+1),& a \text { odd}
\end{cases}\\
&+G^{\beta,\eta}_{[a,b]}(m,b;z)
\begin{cases}
(z\overline{\eta}-\alpha_{b})u(b)-\rho_{b}u(b-1), \quad & b\text{ even}\\
(z\alpha_{b}-\eta)u(b)+z\rho_{b-1}u(b-1),& b \text { odd}
\end{cases}
\end{align*}
for $a<m<b$.
\begin{lemma}\label{covering-lemma}
Let $x, \omega\in \mathbb{T}^{d}$, $z\in \partial \mathbb{D}$ and $[a,b]\subset \mathbb{Z}$. If for any $m\in [a+1,b-1]$, there exists an interval $I_{m}=[a_{m},b_{m}]\subset[a+1,b-1]$ containing $m$ such that
\begin{small}
\begin{equation*}
|G^{\beta,\eta}_{I_{m}}(a_{m},m;z)|
\begin{cases}
|z\overline{\beta}-\alpha_{a}|+\rho_{a}, & a\text{ even}\\
|z\alpha_{a}-\beta|+\rho_{a},& a \text { odd}
\end{cases}
+|G^{\beta,\eta}_{I_{m}}(m,b_{m};z)|
\begin{cases}
|z\overline{\eta}-\alpha_{b}|+\rho_{b},  &b\text{ even}\\
|z\alpha_{b}-\eta|+\rho_{b-1},& b \text { odd}
\end{cases}
<1,
\end{equation*}
\end{small}
then $z\notin \sigma(\mathcal{E}^{\beta,\eta}_{[a,b]})$.
\end{lemma}
\begin{proof}
Assume to the contrary that $z\in \sigma(\mathcal{E}^{\beta,\eta}_{[a,b]})$ and let $u$ be a corresponding eigenvector, i.e., $\mathcal{E}^{\beta,\eta}_{[a,b]}u=zu$. Suppose that $|u(m)|=\max\limits_{n\in(a,b)}|u(n)|$. If $m\in [a+1,b-1]$, then there is a vector $\tilde{u}$ satisfying $\mathcal{E}^{\beta,\eta}_{[a,b]}\tilde{u}=z\tilde{u}$ and $\tilde{u}|_{[a+1,b-1]}=u|_{[a+1,b-1]}$. Thus, from the hypothesis and the Poisson formula, one can get that
\begin{equation*}
|u(m)|\leq \max\{|u(a_{m})|,|u(a_{m}+1)|,|u(b_{m})|,|u(b_{m}-1)|\},
\end{equation*}
which is a contradiction.
\end{proof}

Next, we use the LDT for the characteristic determinant to bound Green's functions and their operator norms.

\begin{lemma}\label{Lemma3.28}
Suppose that $x_{0}\in \mathbb{T}^{d}$, $\omega_{0}\in \mathbb{T}^{d}(p,q)$, $z_{0}\in \partial \mathbb{D}$ and $L(\omega_{0},z_{0})>\gamma>0$. Let $K\in \mathbb{R}$ and $\tau$ be as in LDT. There exists $C_{0}=C_{0}(p,q,z_{0},\gamma)$ such that if $n\geq N(p,q,z_{0},\gamma)$ and
\begin{equation}\label{Lem3.28-(1)}
\log|\varphi^{\beta,\eta}_{[0,n-1]}(\omega_{0},z_{0};x_{0})|>nL_{n}(\omega_{0},z_{0})-K,
\end{equation}
then for any $(\omega,z,x)\in \mathbb{T}^{d}\times\partial \mathbb{D}\times\mathbb{T}^{d}$ with $|x-x_{0}|, |\omega-\omega_{0}|, |z-z_{0}|<\exp(-(K+C_{0}n^{1-\tau}))$ we have
\begin{equation}\label{Lem3.28-(2)}
|G^{\beta,\eta}_{[0,n-1]}(j,k;z)|\leq \exp(-\frac{\gamma}{2}|k-j|+K+2C_{0}n^{1-\tau})
\end{equation}
and
\begin{equation}\label{Lem3.28-(3)}
\|G^{\beta,\eta}_{[0,n-1]}(z)\|\leq \exp(K+3C_{0}n^{1-\tau}).
\end{equation}
\end{lemma}
\begin{proof}
Take $|x-x_{0}|, |\omega-\omega_{0}|, |z-z_{0}|<\exp(-(K+C_{0}n^{1-\tau}))$ with $C$ large enough.

According to Remark \ref{Remark3.14}, we have
\begin{align*}
\log|\varphi^{\beta,\eta}_{[0,n-1]}(\omega,z;x)|&\geq\log|\varphi^{\beta,\eta}_{[0,n-1]}(\omega_{0},z_{0};x_{0})|-1\geq nL_{n}(\omega_{0},z_{0})-K-1\\
&\geq nL_{n}(\omega,z)-K-2\geq nL(\omega,z)-K-2.
\end{align*}
Thus, according to Lemma \ref{Lemma3.11}, for $0\leq j\leq k\leq n-1$, one can obtain
\begin{align*}
|G^{\beta,\eta}_{[0,n-1]}(j,k;z)|&=\frac{1}{\rho_{k}}\Big|\frac{\phi^{\beta,\cdot}_{[0,j-1]}(z)\phi^{\cdot,\eta}_{[k+1,n-1]}(z)}{\phi^{\beta,\eta}_{[0,n-1]}(z)}\Big|\\
&=\frac{1}{\rho_{k}}\Big|\frac{(\rho_{0}\cdots\rho_{j-1})^{-1}\varphi^{\beta,\cdot}_{[0,j-1]}(z)(\rho_{k+1}\cdots\rho_{n-1})^{-1}\varphi^{\cdot,\eta}_{[k+1,n-1]}(z)}{(\rho_{0}\cdots\rho_{n-1})^{-1}\varphi^{\beta,\eta}_{[0,n-1]}(z)}\Big|\\
&=(\rho_{j}\cdots\rho_{k-1})\Big|\frac{\varphi^{\beta,\cdot}_{[0,j-1]}(z)\varphi^{\cdot,\eta}_{[k+1,n-1]}(z)}{\varphi^{\beta,\eta}_{[0,n-1]}(z)}\Big|\\
&\leq\Big|\frac{\varphi^{\beta,\cdot}_{[0,j-1]}(z)\varphi^{\cdot,\eta}_{[k+1,n-1]}(z)}{\varphi^{\beta,\eta}_{[0,n-1]}(z)}\Big|\\
&\leq\exp(-(k-j)L(\omega,z)+Cn^{1-\tau}+K+2)\\
&\leq\exp(-\frac{\gamma}{2}(k-j)+K+2C_{0}n^{1-\tau}).
\end{align*}
Therefore, $|G^{\beta,\eta}_{[0,n-1]}(j,k;z)|\leq\exp(-\frac{\gamma}{2}|k-j|+K+2C_{0}n^{1-\tau})$ for any choice of $j, k$ and then estimate \eqref{Lem3.28-(3)} follows.
\end{proof}
Based on the above analysis, we can obtain the covering form of LDT.
\begin{lemma}\label{covering-form-(LDT)}
Suppose that $n\gg 1$, $x_{0}\in \mathbb{T}^{d}$, $\omega_{0}\in \mathbb{T}^{d}(p,q) $, $z_{0}\in \partial \mathbb{D}$ and $L(\omega_{0},z_{0})>\gamma>0$. Let $\tau, \nu$ be as in LDT. Suppose that for each point $m\in [0,n-1]$, there exists an interval $I_{m}\subset [0,n-1]$ such that:

(i) $\mathrm{dist}(m, [0,n-1]\backslash I_{m})\geq |I_{m}|/100$;

(ii) $|I_{m}|\geq C(p,q,z_{0},\gamma)$;

(iii) $\log|\varphi^{\beta,\eta}_{I_{m}}(\omega_{0},z_{0};x_{0})|>|I_{m}|L_{|I_{m}|}(\omega_{0},z_{0})-|I_{m}|^{1-\tau/4}$.\\
Then for any $(\omega,z,x)\in \mathbb{T}^{d}\times \partial \mathbb{D}\times\mathbb{T}^{d} $ such that
$$|x-x_{0}|,\, |\omega-\omega_{0}|,\, |z-z_{0}|<\exp(-2\max\limits_{m}|I_{m}|^{1-\tau/4}),$$
we have
$$\mathrm{dist}(z,\sigma(\mathcal{E}^{\beta,\eta}_{[0,n-1]}))\geq \exp (-2\max\limits_{m}|I_{m}|^{1-\tau/4}).$$
In addition, if $\omega\in \mathbb{T}^{d}(p,q)$ and $\max\limits_{m}|I_{m}|\leq n^{\nu/2}$, then
$$\log|\varphi^{\beta,\eta}_{[0,n-1]}|>nL_{n}(\omega,z)-n^{1-\tau/2}.$$
\end{lemma}
\begin{proof}
According to Lemma \ref{Lemma3.28}, from (iii) we can obtain that
\begin{equation*}
|G^{\beta,\eta}_{[0,n-1]}(m,k;z)|\leq \exp(-\frac{\gamma}{2}|m-k|+\frac{3}{2}|I_{m}|^{1-\tau/4})
\end{equation*}
provided
$$|x-x_{0}|,\, |\omega-\omega_{0}|,\, |z-z_{0}|<\exp(-\frac{3}{2}|I_{m}|^{1-\tau/4}).$$
Then Lemma \ref{covering-lemma} implies that $z\notin \sigma(E^{\beta,\eta}_{[0,n-1]})$ for any
$|z-z_{0}|<\exp(-\frac{3}{2}|I_{m}|^{1-\tau/4})$.

Thus, if $|z-z_{0}|<\exp(-\frac{3}{2}|I_{m}|^{1-\tau/4})/2$, then
\begin{equation*}
\mathrm{dist}(z,\sigma(\mathcal{E}^{\beta,\eta}_{[0,n-1]}))\geq \exp(-\frac{3}{2}|I_{m}|^{1-\tau/4})/2.
\end{equation*}
This implies the first statement.

If $\max\limits_{m}|I_{m}|\leq n^{\nu/2}$, we have that
\begin{equation*}
\mathrm{dist}(z,\sigma(\mathcal{E}^{\beta,\eta}_{[0,n-1]}))\geq \exp(-n^{\nu/2}).
\end{equation*}
Then the second statement can be shown according to the spectral form of LDT.
\end{proof}

\subsection{Finite Scale Localization}
In this section, we derive the finite scale localization independently of the elimination of resonances.
\begin{theorem}\label{finite-scale-localization}
Assume $x_{0}\in \mathbb{T}^{d}$, $\omega_{0}\in \mathbb{T}^{d}(p,q)$, $z_{0}\in \partial \mathbb{D}$ and $L(\omega_{0},z_{0})>\gamma>0$. Let $\tau, \nu$ be as in LDT and $l, n$ be integers such that $C(p,q,z_{0},\gamma)\leq l \leq n^{\nu/2}$. Denote $z_{j}^{\Lambda}(\omega,x)$, $u_{j}^{\Lambda}(\omega,x)$ be the eigenpairs of $\mathcal{E}_{\Lambda}^{\beta,\eta}$ with $u_{j}^{\Lambda}(\omega,x)$ a unit vector, where $\mathcal{E}_{\Lambda}^{\beta,\eta}$ is the restriction to a finite interval $\Lambda$ of the extended matrix $\mathcal{E}$ and $\beta, \eta \in \partial\mathbb{D}$. Suppose that there exists an interval $I=[n',n'']\subset[0,n-1]$ such that
\begin{equation}\label{local-localization-(3)}
\log\big|\varphi_{[0,l-1]}^{\beta,\eta}(\omega_{0},z_{0};x_{0}+(m-1)\omega_{0})\big|>lL_{l}(\omega_{0},z_{0})-l^{1-\tau/4}
\end{equation}
for any $m\in[0,n-l]\backslash I$. Then for any $(\omega,x)\in \mathbb{T}^{d}(p,q)\times\mathbb{T}^{d}$ satisfying $|x-x_{0}|, |\omega-\omega_{0}|<\exp(-l)$ and any eigenvalue $|z_{j}^{[0,n-1]}(\omega,x)-z_{0}|<\exp(-l)$, we have
\begin{equation}\label{local-localization-(2)}
|u_{j}^{[0,n-1]}(\omega,x;s)|<\exp(-\frac{\gamma}{12}\mathrm{dist}(s,I))
\end{equation}
provided $\mathrm{dist}(s,I)\geq l^{2/\nu}$.
\end{theorem}
\begin{proof}
Take $x, \omega, z=z_{j}^{[0,n-1]}(\omega,x)$ satisfying the assumptions and $s\in[0,n-1]\backslash I$ such that $d=\mathrm{dist}(s,I)\geq l^{2/\nu}$. Suppose $k\in[0,n']$. Let $J=[s-d,s+d]\cap[0,n-1]=:[t,n']$.
Obviously, $|J|\geq l^{2/\nu}$.

According to the covering form of LDT, we can obtain
\begin{equation}\label{local-localization-(3)}
\log|\varphi_{J}^{\beta,\eta}(\omega,z;x)|\geq |J|L_{|J|}(\omega,z)-|J|^{1-\tau/2}.
\end{equation}
Let $u=u_{j}^{[0,n-1]}(\omega,x)$. From Poisson formula, for $s>t$,
\begin{align*}
u(s)=&G^{\beta,\eta}_{J}(s,t;z)
\begin{cases}
(z\overline{\beta}-\alpha_{t})u(t)-\rho_{t}u(t+1), \quad & t\text{ even,}\\
(z\alpha_{t}-\beta)u(t)+z\rho_{t}u(t+1),& t \text { odd,}
\end{cases}\\
&+G^{\beta,\eta}_{J}(s,n';z)
\begin{cases}
(z\overline{\eta}-\alpha_{n'})u(n')-\rho_{n'}u(n'-1), \quad & n'\text{ even,}\\
(z\alpha_{n'}-\eta)u(n')+z\rho_{n'-1}u(n'-1),& n' \text { odd}.
\end{cases}
\end{align*}
From Lemma \ref{Lemma3.28}  we have
\begin{equation*}
|G^{\beta,\eta}_{J}(t, s;z)|\leq \exp(-\frac{\gamma}{2}|s-t|+|J|^{1-\tau/2}+2C_{0}|J|^{1-\tau})
\end{equation*}
and
\begin{equation*}
|G^{\beta,\eta}_{J}(s,n';z)|\leq \exp(-\frac{\gamma}{2}|s-n'|+|J|^{1-\tau/2}+2C_{0}|J|^{1-\tau}).
\end{equation*}
Therefore, we can get that
\begin{align*}
|u(s)|&\leq 6\exp(-\frac{\gamma}{2}d+|J|^{1-\tau/2}+2C_{0}|J|^{1-\tau})\\
& <6\exp(-\frac{\gamma}{2}d+C|J|^{1-\tau/2})<\exp(-\frac{\gamma}{12}d).
\end{align*}
For $s\in [n'',n-1]$, the statement can be shown analogously.
\end{proof}

To complement the finite scale localization of eigenfunctions, we next establish a separation result for eigenvalues.

\begin{lemma}\label{Lemma4.2}
Assume $x_{0}\in \mathbb{T}^{d}(p,q)$, $z_{0}\in \partial \mathbb{D}$ and $L(\omega_{0},z_{0})>\gamma>0$. Let $\tau, \nu$ be as in LDT and $l, n$ be integers such that $C(p,q,z_{0},\gamma)\leq l \leq n^{\nu/2}$. Assume that there exists an interval $I=[n',n'']\subset[0,n-1]$ such that \eqref{finite-scale-localization} holds and $|I|\geq l^{2/\nu}+\log n$. Then for any $(\omega,x)\in \mathbb{T}^{d}(p,q)\times\mathbb{T}^{d}$ satisfying $|x-x_{0}|, |\omega-\omega_{0}|<\exp(-l)$ and any eigenvalue $|z_{j}^{[0,n-1]}(\omega,x)-z_{0}|<\exp(-l)/2$, we have
\begin{equation}\label{Lem4.2-(1)}
|z_{j}^{[0,n-1]}(\omega,x)-z_{k}^{[0,n-1]}(\omega,x)|\geq \exp(-C|I|)
\end{equation}
for any $k\neq j$, where $C=C(\alpha_{i},z_{0})$.
\end{lemma}
\begin{proof}
We argue by contradiction. Let $z_{1}=z_{j}^{[0,n-1]}(\omega,x)$, $z_{2}=z_{k}^{[0,n-1]}(\omega,x)$ and $u_{1}=u_{j}^{[0,n-1]}(\omega,x)$, $u_{2}=u_{k}^{[0,n-1]}(\omega,x)$ be the corresponding eigenfunctions and assume
$$|z_{1}-z_{2}|<\exp(-C|I|).$$
Since $u_{1}, u_{2}$ are eigenfunctions corresponding to different eigenvalues, they are orthogonal and
$$\|u_{1}-u_{2}\|^{2}=\|u_{1}\|^{2}+\|u_{2}\|^{2}.$$
Let $\tilde{I}=\{s\in [0,n-1]: \mathrm{dist}(s,I)<l^{2/\nu}\}$. According to Theorem \ref{finite-scale-localization}, one can obtain
\begin{equation*}
\sum_{s\notin \tilde{I}}|u_{i}(s)|^{2}\leq \|u_{i}\|^{2}\sum_{s\notin \tilde{I}}\exp(-\frac{\gamma}{12}\mathrm{dist}(s,I))\leq \|u_{i}\|^{2}\exp(-\frac{\gamma}{24} l^{2/\nu})
\end{equation*}
(we used the fact that $l$ is taken to be large enough).

Therefore,
\begin{equation}\label{Lem4.2-(2)}
\sum_{s\notin \tilde{I}}|u_{1}(s)-u_{2}(s)|^{2}\leq 2\exp(-\frac{\gamma}{24} l^{2/\nu})(\|u_{1}\|^{2}+\|u_{2}\|^{2}).
\end{equation}
For those $s\in \tilde{I}:=(I_{1},I_{2})$, applying Poisson formula, we have
\begin{align*}
&|u_{1}(s)-u_{2}(s)|\\
& \leq 2\big|G_{[I_{1},I_{2}]}^{\beta,\eta}(I_{1},s;z_{1})-G_{[I_{1},I_{2}]}^{\beta,\eta}(I_{1},s;z_{2})\big|(|u_{1}(I_{1})|+|u_{1}(I_{1}+1)|+|u_{2}(I_{1})|+|u_{2}(I_{1}+1)|)\\
&\quad+2\big|G_{[I_{1},I_{2}]}^{\beta,\eta}(s,I_{2};z_{1})-G_{[I_{1},I_{2}]}^{\beta,\eta}(s,I_{2};z_{2})\big|(|u_{1}(I_{2}-1)|+|u_{1}(I_{2})|+|u_{2}(I_{2})|+|u_{2}(I_{2}-1)|)\\
& \leq 2(\|u_{1}\|+\|u_{2}\|)(\big|G_{[I_{1},I_{2}]}^{\beta,\eta}(I_{1},s;z_{1})-G_{[I_{1},I_{2}]}^{\beta,\eta}(I_{1},s;z_{2})\big|+\big|G_{[I_{1},I_{2}]}^{\beta,\eta}(s,I_{2};z_{1})-G_{[I_{1},I_{2}]}^{\beta,\eta}(s,I_{2};z_{2})\big|).
\end{align*}
Note that $|z_{j}^{[0,n-1]}(\omega,x)-z_{0}|<\exp(-l)/2$, then Lemma \ref{Lemma3.28} can be applied. From \eqref{Lem3.28-(2)}, one can obtain that
\begin{equation*}
\big|G_{[I_{1},I_{2}]}^{\beta,\eta}(I_{1},s;z_{1})-G_{[I_{1},I_{2}]}^{\beta,\eta}(I_{1},s;z_{2})\big|+\big|G_{[I_{1},I_{2}]}^{\beta,\eta}(s,I_{2};z_{1})-G_{[I_{1},I_{2}]}^{\beta,\eta}(s,I_{2};z_{2})\big|<\frac{1}{4}.
\end{equation*}
Thus, we have
\begin{equation}\label{Lem4.2-(3)}
\sum_{s\in \tilde{I}}|u_{1}(s)-u_{2}(s)|^{2}<(\|u_{1}\|^{2}+\|u_{2}\|^{2}).
\end{equation}
Combining \eqref{Lem4.2-(2)} and \eqref{Lem4.2-(3)}, we can obtain that
\begin{equation*}
\|u_{1}-u_{2}\|^{2}<\|u_{1}\|^{2}+\|u_{2}\|^{2},
\end{equation*}
which contradicts the fact that $u_{1}, u_{2}$ are orthogonal.
\end{proof}
\section{Removing Double Resonances by Using Semialgebraic Sets}\label{section4}

In this section, semialgebraic sets will be introduced by approximating the Verblunsky coefficients $\alpha_{i}$ with a polynomial $\tilde{\alpha}_{i}$. Then we obtain a result on elimination of double resonances using semialgebraic sets. More precisely, given $n\geq 1$, by truncating $\alpha_{i}$'s Fourier series and the Taylor series of the trigonometric functions, one can obtain a polynomial $\tilde{\alpha}_{i}$ of degree $\lesssim n^{4}$ such that
\begin{equation}\label{VCA}
\|\alpha_{i}-\tilde{\alpha}_{i}\|_{\infty}\lesssim \exp(-n^{2}).
\end{equation}
Let $\tilde{\mathcal{E}}_{[0,n-1]}^{\beta,\eta}$ be the matrix with this truncated Verblunsky coefficient and the refined boundary condition $\beta, \eta\in \partial\mathbb{D}$. If we let $z_{j}^{[0,n-1]}(x,\omega)$ and $\tilde{z}_{j}^{[0,n-1]}(x,\omega)$ be the eigenvalues of $\mathcal{E}_{[0,n-1]}^{\beta,\eta}$ and $\tilde{\mathcal{E}}_{[0,n-1]}^{\beta,\eta}$ respectively, then
\begin{equation}\label{eigenvalue-estimate}
|z_{j}^{[0,n-1]}(x,\omega)-\tilde{z}_{j}^{[0,n-1]}(x,\omega)|\leq \|\mathcal{E}_{[0,n-1]}^{\beta,\eta}(x,\omega)-\tilde{\mathcal{E}}_{[0,n-1]}^{\beta,\eta}(x,\omega)\|\lesssim 8\|\alpha_{i}-\tilde{\alpha}_{i}\|_{\infty}.
\end{equation}

For the purpose of semialgebraic approximation, we take the following set into consideration
\begin{equation*}
\mathbb{T}_{n}^{d}(p,q)=\{\omega\in\mathbb{T}^{d}: \|k\cdot \omega\|\geq \frac{p}{|k|^{q}},\, \mathrm{for}\,\, \mathrm{all}\,\, \mathrm{nonzero}\,\, k \in \mathbb{Z}^{d}, 0< |k|\leq n\}.
\end{equation*}

We first recall the definition of semialgebraic sets, which are finite unions of sets defined by polynomial equalities and inequalities. This structure is critical for measuring exceptional resonant sets.
\begin{definition 1}\cite[Definition 9.1]{Bourgain-book}\label{semialgebraic-set}
A set $\mathcal{S}\subset \mathbb{R}^{n}$ is called semialgebraic if it is a finite union of sets defined by a finite number of polynomial equalities and inequalities. More precisely, let $\mathcal{P}=\{P_{1},\ldots,P_{s}\}\subset R[X_{1},\ldots,X_{n}]$ be a family of real polynomials whose degrees are bounded by $d$. A (closed) semialgebraic set $\mathcal{S}$ is given by an expression
\begin{equation}\label{semialgebraic(1)}
\mathcal{S}=\bigcup\limits_{j}\bigcap\limits_{l\in \mathcal{L}_{j}}\{R^{n}|P_{l}s_{jl}0\},
\end{equation}
where $\mathcal{L}_{j}\subset\{1,\ldots,s\}$ and $s_{jl}\in \{\geq,\leq,=\}$ are arbitrary. We say that $\mathcal{S}$ has degree at most $sd$, and its degree is the infimum of $sd$ over all representations as in \eqref{semialgebraic(1)}.
\end{definition 1}

To control the number of resonant points along orbits of the shift map, we use the following lemma, which bounds the density of semialgebraic sets under Diophantine frequency shifts.

\begin{lemma}\label{Lemma3.31}\cite[Corollary 9.7]{Bourgain-book}
Let $\mathcal{S}\subset [0,1]^{n}$ be semialgebraic of degree $B$ and $\mathrm{mes}_{n}\mathcal{S}<\eta$. Let $\omega\in T^{n}$ satisfy a DC and $N$ be a large integer,
\begin{equation*}
\log B\ll \log N< \log\frac{1}{\eta}.
\end{equation*}
Then, for any $x_{0}\in T^{n}$
\begin{equation}\label{Lem3.31-(1)}
\#\{k=1,\ldots,N| x_{0}+k\omega\in\mathcal{S}(\mathrm{mod} \,1) \}<N^{1-\delta}
\end{equation}
for some $\delta=\delta(\omega)$.
\end{lemma}

Next, we extend this density control to intersections of semialgebraic sets, which is necessary for handling multiple resonant conditions simultaneously.

\begin{lemma}\cite[Lemma 1.18]{Bourgain-GAFA-2007}\label{Lemma-NDR-1}
Let $\mathcal{A}\subset [0,1]^{s+r}$ be semialgebraic of degree $B$ and such that for each $t\in [0,1]^{r}$, $\mathrm{mes}_{s}(\mathcal{A}(t))<\eta$. Then
\begin{equation*}
\{(x_{1},\ldots,x_{2^{r}}):\mathcal{A}(x_{1})\cap\ldots\cap\mathcal{A}(x_{2^{r}})\neq \emptyset\}\subset [0,1]^{s^{2^{r}}}
\end{equation*}
is semialgebraic of degree at most $B^{C}$ and measure at most
\begin{equation*}
\eta_{r}=B^{C}\eta^{s^{-r}2^{-r(r-1)/2}} \,\, with \,\, C=C(r).
\end{equation*}
\end{lemma}
Note that in the above lemma $\mathcal{A}(t)=\{x\in [0,1]^{s}:(x,t)\in\mathcal{A} \}$ and $\mathcal{A}(x)=\{t\in [0,1]^{r}:(x,t)\in\mathcal{A} \}$.

To control resonant frequencies, we use the following lemma, which bounds the measure of frequencies for which orbits pass through a semialgebraic set. This is key for excluding exceptional frequencies that support double resonances.

\begin{lemma}\cite[Lemma 1.20]{Bourgain-GAFA-2007}\label{Lemma-NDR-2}
Let $\mathcal{A}\subset[0,1]^{rs}$ be a semialgebraic set of degree $B$ and $\mathrm{mes}_{rs}(\mathcal{A})<\eta$. Let $\mathcal{N}_{1},\ldots,\mathcal{N}_{s-1}\subset \mathbb{Z}$ be finite sets with the property that
\begin{equation*}
|n_{i}|>(B|n_{i-1}|)^{C},\,\, if\,\, n_{i}\in \mathcal{N}_{i} \,\, and\,\, n_{i-1}\in \mathcal{N}_{i-1},\,\, 2\leq i\leq s-1,
\end{equation*}
where $C=C(s,r)$. Assume also
\begin{equation*}
\max\limits_{n\in \mathcal{N}_{s-1}}|n|^{C}<\frac{1}{\eta}.
\end{equation*}
Then
\begin{equation*}
\mathrm{mes}\{\omega\in[0,1]^{r}:(\omega,n_{1}\omega,\ldots,n_{s-1}\omega)\in \mathcal{A}\,\, for\,\, some\,\, n_{i}\in \mathcal{N}_{i}\}<B^{C}(\min\limits_{n\in \mathcal{N}_{1}}|n|)^{-1}.
\end{equation*}
\end{lemma}

\begin{lemma}\cite[Lemma 9.9]{Bourgain-book}\label{Lemma7.1}
Let $\mathcal{S} \subset[0,1]^{d}$, $d=d_{1}+d_{2}$, be a semialgebraic set of degree $B$ and $\mathrm{mes}_{d}(\mathcal{S})<\vartheta$, $\log B\ll \log\frac{1}{\vartheta}$. We denote $(x,\omega)\in [0,1]^{d_{1}}\times[0,1]^{d_{2}}$ the product variable. Fix $\varepsilon>\vartheta^{\frac{1}{d}}$. Then there is a decomposition
$$\mathcal{S}=\mathcal{S}_{1}\cup \mathcal{S}_{2}$$
$\mathcal{S}_{1}$ satisfying
$$\mathrm{mes}_{d_{2}}(\mathrm{Proj}_{\omega}\mathcal{S}_{1})<B^{C}\varepsilon$$
and $\mathcal{S}_{2}$ satisfying the transversality property
$$\mathrm{mes}_{d_{1}}(\mathcal{S}_{2}\cap L)<B^{C}\varepsilon^{-1}\vartheta^{\frac{1}{d}}$$
for any $d_{1}$-dimensional plane $L$ s.t. $\max\limits_{1\leq j\leq d_{2}}|\mathrm{Proj}_{L}(e_{j})|<\frac{1}{100}\varepsilon$ (we denote by $e_{1},\ldots,e_{d_{2}}$ the $\omega$-coordinate vectors).
\end{lemma}

From Corollary \ref{Coro3.13}, we can get the following statements.
\begin{coro}\label{Coro3.30}
Assume $\omega\in \mathbb{T}^{d}(p,q)$, $z\in \partial \mathbb{D}$ and $L(\omega,z)>\gamma>0$. Then for any $x\in \mathbb{T}^{d}$ and $n\geq 1$ we have
\begin{equation*}
\|M_{n}(\omega,z;x)-\tilde{M}_{n}(\omega,z;x)\|\leq 8\|\alpha_{i}-\tilde{\alpha}_{i}\|_{\infty}\exp(nL(\omega,z)+Cn^{1-\tau}),
\end{equation*}
where $C=C(p,q,z,\gamma)$. In particular,
$$\big|\log\|M_{n}(\omega,z;x)\|-\log\|\tilde{M}_{n}(\omega,z;x)\|\big|\leq 8\|\alpha_{i}-\tilde{\alpha}_{i}\|_{\infty}\exp(nL(\omega,z)+Cn^{1-\tau}),$$
\begin{small}
$$\big|\log|\varphi_{[0,n-1]}^{\beta,\eta}(\omega,z;x)|-\log|\tilde{\varphi}_{[0,n-1]}^{\beta,\eta}(\omega,z;x)|\big|\leq 8\|\alpha_{i}-\tilde{\alpha}_{i}\|_{\infty}\frac{\exp(nL(\omega,z)+Cn^{1-\tau})}{\max\{|\varphi_{[0,n-1]}^{\beta,\eta}(\omega,z;x)|,|\tilde{\varphi}_{[0,n-1]}^{\beta,\eta}(\omega,z;x)|\}},$$
\end{small}
provided the right-hand sides are less than $\frac{1}{2}$.
\end{coro}

We now show that the set of phases where the transfer matrix norm deviates from its LDT average is contained in a semialgebraic set of small measure.

\begin{lemma}\label{Lemma3.32}
Let $n\geq 1$, $\omega\in \mathbb{T}^{d}_{n}(p,q)$, $z\in \partial \mathbb{D}$, $\sigma, \tau$ be as in LDT, and
\begin{equation*}
\mathcal{B}_{n}:=\{x\in \mathbb{T}^{d}:|\log\|M_{n}(\omega,z;x)\|-nL_{n}(\omega,z)|\geq 4n^{1-\tau}\}.
\end{equation*}
Then there exists a semialgebraic set $\mathcal{S}_{n}$ such that $\mathcal{B}_{n}\subset\mathcal{S}_{n}$, $\mathrm{deg}(\mathcal{S}_{n})\leq n^{C}$, and $\mathrm{mes}(\mathcal{S}_{n})<\exp(-C_{0}n\sigma)$ provided $n\geq n(p,q,z)$.
\end{lemma}
\begin{proof}
Take $\tilde{\alpha}_{n}$ as in \eqref{VCA} and let
\begin{equation*}
\mathcal{S}_{n}:=\{x\in \mathbb{T}^{d}:|\log\|\tilde{M}_{n}(\omega,z;x)\|_{\mathrm{HS}}-nL_{n}(\omega,z)|\geq 2n^{1-\tau}\},
\end{equation*}
with $\|\cdot\|_{\mathrm{HS}}$ standing for the Hilbert-Schmidt norm. Obviously, $\mathrm{deg}(\mathcal{S}_{n})\leq n^{C}$. According to Corollary \ref{Coro3.30}, $\mathcal{B}_{n}\subset\mathcal{S}_{n}$. In addition,
\begin{equation*}
\mathcal{S}_{n}\subset\{x\in \mathbb{T}^{d}:|\log\|M_{n}(\omega,z;x)\|-nL_{n}(\omega,z)|> n^{1-\tau}\}
\end{equation*}
follows from Corollary \ref{Coro3.30}.

Therefore, by LDT, the measure estimate of $\mathcal{S}_{n}$ follows.
\end{proof}

We extend the LDT average to time-averages along the shift orbit, ensuring that the average of the transfer matrix norm over a long orbit is close to the finite scale Lyapunov exponent.

\begin{lemma}\label{Lemma3.33}
Let $\sigma, \tau$  be as in LDT. Then for any $n\geq n_{0}(p,q)$,
\begin{equation*}
C(p,q)\log n\leq \log J<n^{\sigma},
\end{equation*}
$x\in \mathbb{T}^{d}$, $\omega\in \mathbb{T}^{d}_{J}(p,q)$, $z\in \partial \mathbb{D}$, we have
\begin{equation*}
\big|\frac{1}{J}\sum_{j=1}^{J}\log \|M_{n}(\omega,z;x+j\omega)\|-nL_{n}(\omega,z)\big|\leq 5n^{1-\tau}.
\end{equation*}
\end{lemma}
\begin{proof}
Let $\mathcal{B}_{n}$ and $\mathcal{S}_{n}$ be as in Lemma \ref{Lemma3.33}. For any $x\in \mathbb{T}^{d}\backslash \mathcal{S}_{n}$, we have that
\begin{equation*}
nL_{n}(\omega,z)-4n^{1-\tau}\leq \log \|M_{n}(\omega,z;x)\|\leq nL_{n}(\omega,z)+4n^{1-\tau}.
\end{equation*}
While for $x\in \mathcal{S}_{n}$, one can obtain that
\begin{equation*}
0\leq\log \|M_{n}(\omega,z;x)\|\leq C(\alpha_{i},z)n.
\end{equation*}
Select $J$ in accordance with the required assumptions. Combining Lemma \ref{Lemma3.31}, one can get
\begin{align*}
&\frac{1}{J}\sum_{j=1}^{J}\log \|M_{n}(\omega,z;x+j\omega)\|-nL_{n}(\omega,z)\\
&\leq\frac{J-J^{1-\delta}}{J}(nL_{n}(\omega,z)+4n^{1-\tau})+\frac{J^{1-\delta}}{J}Cn-nL_{n}(\omega,z)\\
&\leq \frac{J-J^{1-\delta}}{J}(4n^{1-\tau})+\frac{J^{1-\delta}}{J}Cn\\
&\leq 5n^{1-\tau}.
\end{align*}
On the other hand,
\begin{align*}
\frac{1}{J}\sum_{j=1}^{J}\log \|M_{n}(\omega,z;x+j\omega)\|-nL_{n}(\omega,z)&\geq\frac{J-J^{1-\delta}}{J}(nL_{n}(\omega,z)-4n^{1-\tau})-nL_{n}(\omega,z)\\
&= \frac{J-J^{1-\delta}}{J}(-4n^{1-\tau})+\frac{J^{1-\delta}}{J}(-nL_{n}(\omega,z))\\
&\geq -5n^{1-\tau}.
\end{align*}
Thus, we complete the proof.
\end{proof}

We characterize the set of resonant $(\omega,z,x)$ where the characteristic determinant is small as a semialgebraic set of small measure.

\begin{lemma}\label{Lemma3.34}
Let $n\geq 1$ and $\tau, \nu$ be as in LDT. Let $\mathcal{B}_{n}$ be the set of $(\omega,z,x)\in \mathbb{T}^{d}(p,q)\times \partial \mathbb{D}\times \mathbb{T}^{d}$ such that $L(\omega,z)>\gamma>0$ and
\begin{equation*}
\log|\varphi_{[0,n-1]}^{\beta,\eta}(\omega,z;x)|\leq nL_{n}(\omega,z)-n^{1-\tau/2}.
\end{equation*}
Then there exists a semialgebraic set $\mathcal{S}_{n}$ such that $\mathcal{B}_{n}\subset \mathcal{S}_{n}$, $\mathrm{deg}(\mathcal{S}_{n})\leq n^{C(p,q)}$, and $\mathrm{mes}(\mathcal{S}_{n})<\exp(-n^{\nu}/2)$, provided $n\geq N_{0}(\alpha_{i},p,q,\gamma)$. Furthermore,
\begin{equation*}
\mathrm{mes}(\mathcal{S}_{n}(\omega,z))<\exp(-n^{\nu}),\quad \mathcal{S}_{n}(\omega,z)=\{x:(\omega,z,x)\in \mathcal{S}_{n}\}.
\end{equation*}
\end{lemma}
\begin{proof}
Take $\tilde{\alpha}_{i}$ as in \eqref{VCA} and let $\mathcal{S}_{n}$ be the set of
\begin{equation*}
(\omega,z,x)\in \mathbb{T}_{J}^{d}(p,q)\times \partial \mathbb{D}\times \mathbb{T}^{d}, \quad J=n^{C(a,b)}
\end{equation*}
such that
\begin{equation}\label{Lem3.34-(1)}
\frac{1}{nJ}\sum_{j=1}^{J}\log\|\tilde{M}_{n}(\omega,z;x+jw)\|_{\mathrm{HS}}\geq \frac{\gamma}{2}
\end{equation}
and
\begin{equation}\label{Lem3.34-(2)}
\log|\varphi_{[0,n-1]}^{\beta,\eta}(\omega,z;x)|\leq \frac{1}{J}
\sum_{j=1}^{J}\log\|\tilde{M}_{n}(\omega,z;x+jw)\|_{\mathrm{HS}}-n^{1-\tau/2}/2.
\end{equation}
Obviously, $\mathrm{deg}(\mathcal{S}_{n})\leq n^{C}$. According to Corollary \ref{Coro3.30} and Lemma \ref{Lemma3.32}, $\mathcal{B}_{n}\subset\mathcal{S}_{n}$.

Combining Corollary \ref{Coro3.30}, \eqref{Lem3.34-(1)} and \eqref{Lem3.34-(2)}, we have $L_{n}(\omega,z)\geq\frac{\gamma}{4}$ and
$$\log|\varphi_{[0,n-1]}^{\beta,\eta}(\omega,z;x)|\leq nL_{n}(\omega,z)-n^{1-\tau/2}/4.$$
Therefore, $\mathrm{mes}(\mathcal{S}_{n}(\omega,z))<\exp(-n^{\nu})$ follows from LDT. In addition, $$\mathrm{mes}(\mathcal{S}_{n})<2\pi \exp(-n^{\nu})<\exp(-n^{\nu}/2).$$
\end{proof}

We use the semialgebraic characterization to show that for most frequencies, resonant phases are rare along orbits.
\begin{lemma}\label{Lemma5.3}
Let $l\geq 1$, $\gamma>0$, and $\tau, \nu$ be as in LDT. Let $\mathcal{N}_{1},\ldots,\mathcal{N}_{s-1}\subset \mathbb{Z}$, $s=2^{2d+1}$, be finite sets with the property that
\begin{equation*}
|n_{i}|>(l^{C(d)}|n_{i-1}|)^{C'(d)},\,\, if\,\, n_{i}\in \mathcal{N}_{i}\,\,and\,\, n_{i-1}\in \mathcal{N}_{i-1},\,\, 2\leq i\leq s-1
\end{equation*}
and
\begin{equation*}
\max\limits_{n\in \mathcal{N}_{s-1}} |n|<\exp(cl^{\nu}),\,\,c=c(d).
\end{equation*}
For any $l\geq l_{0}(\alpha_{i},p,q,\gamma)$ there exists a set $\Omega_{l}$, based on the selection of finite sets $\mathcal{N}_{i}$, such that
\begin{equation*}
\mathrm{mes}(\Omega_{l})\leq l^{C(p,q)}(\min\limits_{n\in \mathcal{N}_{1}}|n|)^{-1}
\end{equation*}
and the subsequent statement is valid. For any $x\in \mathbb{T}^{d}$, $\omega\in \mathbb{T}^{d}(p,q)\backslash \Omega_{l}$, $z\in \partial \mathbb{D}$, if $L(\omega,z)>\gamma$ and
\begin{equation*}
\log|\varphi_{[0,l-1]}^{\beta,\eta}(\omega,z;x)|\leq lL_{l}(\omega,z)-l^{1-\tau/2},
\end{equation*}
then there exists $i\in \{1,\ldots,s-1\}$, depending on $ \omega, z, x$, such that
\begin{equation*}
\log|\varphi_{[0,l-1]}^{\beta,\eta}(\omega,z;x+(n-1)\omega)|> lL_{l}(\omega,z)-l^{1-\tau/2}, \,\,for \,\,all\,\, n\in \mathcal{N}_{i}.
\end{equation*}
\end{lemma}
\begin{proof}
Let $\mathcal{B}_{l}$, $\mathcal{S}_{l}$ be the sets from \ref{Lemma3.34}. Then $\mathcal{S}_{l}$ is semialgebraic, $\mathcal{B}_{l}\subset\mathcal{S}_{l}$, $\mathrm{deg}(\mathcal{S}_{l})\leq l^{C(p,q)}$, $\mathrm{mes}(\mathcal{S}_{l}(\omega,z))<\exp(-l^{\nu})$.

Let $\mathcal{T}_{l}$ be the set of $(\omega,z,x,y)\in \mathbb{T}^{d}_{l^{C}}(p,q)\times \partial \mathbb{D}\times\mathbb{T}^{d}\times\mathbb{T}^{d}$ such that $(\omega,z,x+y-\omega)\in\mathcal{S}_{l}$. Then $\mathcal{T}_{l}$ is a semialgebraic set and
\begin{equation*}
\mathrm{deg}(\mathcal{T}_{l})\leq l^{C},\,\, \mathrm{mes}(\mathcal{T}_{l}(\omega,z,x))<\exp(-l^{\nu}).
\end{equation*}
According to Lemma \ref{Lemma-NDR-1},
\begin{equation*}
\mathcal{A}:=\{(y_{1},\ldots,y_{s}):\mathcal{T}_{l}(y_{1})\cap\cdots\cap\mathcal{T}_{l}(y_{s})\neq \emptyset\}
\end{equation*}
is semialgebraic of degree at most $l^{C}$ and measure at most $\exp(-cl^{\nu})$, where $c=c(d)$.

Let $\Omega_{l}=\{\omega:(\omega,n_{1}\omega,\ldots,n_{s-1}\omega)\in \mathcal{A}\,\,\mathrm{for}\,\,\mathrm{some}\,\,n_{i}\in \mathcal{N}_{i}\}$. Then we can obtain the conclusion with the aid of Lemma \ref{Lemma-NDR-2}.
\end{proof}

We extend the previous result to two resonant phases, ensuring that for most frequencies, two distinct resonant phases cannot have long orbits without non-resonant points.
\begin{coro}\label{Coro5.4}
Let $l\geq 1$ and $\tau, \nu$ be as in LDT. Let
$\mathcal{N}_{1},\ldots,\mathcal{N}_{(s-1)^{2}}\subset \mathbb{Z}$ be finite sets, with $s=2^{2d+1}$, possessing the property that
\begin{equation*}
|n_{i}|>(l^{C(d)}|n_{i-1}|)^{C'(d)},\,\, if\,\, n_{i}\in \mathcal{N}_{i}\,\,and\,\, n_{i-1}\in \mathcal{N}_{i-1},\,\, 2\leq i\leq (s-1)^{2}
\end{equation*}
and
\begin{equation*}
\max\limits_{n\in \mathcal{N}_{(s-1)^{2}}} |n|<\exp(cl^{\nu}),\,\,c=c(d).
\end{equation*}
For any $l\geq l_{0}(\alpha_{i},p,q,\gamma)$ there exists a set $\Omega_{l}$, based on the selection of finite sets $\mathcal{N}_{i}$, such that
\begin{equation*}
\mathrm{mes}(\Omega_{l})\leq l^{C(p,q)}(\min\limits_{n\in \mathcal{N}_{1}}|n|)^{-1}
\end{equation*}
and the following statement holds. For any $x_{1}, x_{2}\in \mathbb{T}^{d}$, $\omega\in \mathbb{T}^{d}(p,q)\backslash \Omega_{l}$, $z\in \partial \mathbb{D}$, if $L(\omega,z)>\gamma$ and
\begin{equation*}
\log|\varphi_{[0,l-1]}^{\beta,\eta}(\omega,z;x_{j})|\leq lL_{l}(\omega,z)-l^{1-\tau/2},\,\, j=1,2,
\end{equation*}
then there exists $i\in \{1,\ldots,(s-1)^{2}\}$, depending on $ \omega, z, x_{1}, x_{2}$, such that
\begin{equation*}
\log|\varphi_{[0,l-1]}^{\beta,\eta}(\omega,z;x_{j}+(n-1)\omega)|> lL_{l}(\omega,z)-l^{1-\tau/2}, \,\,for \,\,all\,\, n\in \mathcal{N}_{i}.
\end{equation*}
\end{coro}
\begin{proof}
Let $\Omega_{l}$ be the union of the sets obtained from Lemma \ref{Lemma5.3} with the following choices of finite sets:
\begin{equation}\label{Coro5.4-(1)}
\mathcal{N}_{k(s-1)+1},\ldots,\mathcal{N}_{(k+1)(s-1)},\,\,k=0,\ldots,s-2,
\end{equation}
\begin{equation}\label{Coro5.4-(2)}
\mathop{\cup}\limits_{j=1}^{s-1}\mathcal{N}_{j}, \mathop{\cup}\limits_{j=1}^{s-1}\mathcal{N}_{(s-1)+j},\ldots,\mathop{\cup}\limits_{j=1}^{s-1}\mathcal{N}_{(s-2)(s-1)+j}.
\end{equation}
It is evident that $\Omega_{l}$ complies with the prescribed measure bound. Let $\omega\in \mathbb{T}^{d}(p,q)\backslash\Omega_{l}$.

From Lemma \ref{Lemma5.3} and \eqref{Coro5.4-(2)}, there exists $k\in \{0,\ldots,s-2\}$ such that
\begin{equation*}
\log|\varphi_{[0,l-1]}^{\beta,\eta}(\omega,z;x_{1}+(n-1)\omega)|> lL_{l}(\omega,z)-l^{1-\tau/2}
\end{equation*}
for all $n\in \mathop{\cup}\limits_{j=1}^{s-1}\mathcal{N}_{k(s-1)+j}$. Besides, from  \ref{Lemma5.3} and \eqref{Coro5.4-(1)}, there exists $j\in \{1,\ldots,s-1\}$ such that
\begin{equation*}
\log|\varphi_{[0,l-1]}^{\beta,\eta}(\omega,z;x_{2}+n\omega)|> lL_{l}(\omega,z)-l^{1-\tau/2}
\end{equation*}
for all $n\in \mathop{\cup}\limits_{j=1}^{s-1}\mathcal{N}_{k(s-1)+j}$. Taking $i=k(s-1)+j$, the conclusion follows.
\end{proof}

For our purpose, we need the following Wegner's estimate.
\begin{lemma}\label{Wegner's-estimate}
Let $\nu$ be as in LDT. Assume $x_{0}\in \mathbb{T}^{d}$, $\omega_{0}\in \mathbb{T}^{d}(p,q) $, $z_{0}\in \partial \mathbb{D}$ and $L(\omega_{0},z_{0})>\gamma>0$. Let $l, n$ be integers such that $(2\log n)^{1/\nu}\leq Cl \leq n$. Then for any $n\geq N_{0}(p,q,\gamma)$ there exists  a set $\mathcal{B}_{n,\omega_{0},z_{0}}$, $\mathrm{mes}(\mathcal{B}_{n,\omega_{0},z_{0}})<\exp(-Cl^{\nu}/2)$ such that for any $x\in \mathbb{T}^{d}\backslash \mathcal{B}_{n,\omega_{0},z_{0}}$ and any $(\omega,z)\in \mathbb{T}^{d}\times \partial \mathbb{D}$, $|x-x_{0}|$, $|\omega-\omega_{0}|$, $|z-z_{0}|<\exp(-l)$, we have
\begin{equation*}
\mathrm{dist}(z,\sigma(\mathcal{E}_{[0,n-1]}^{\beta,\eta}))\geq \exp(-l).
\end{equation*}
\end{lemma}
\begin{proof}
Take $\omega, z$ satisfying the assumptions. Let $\mathcal{B}_{n,\omega_{0},z_{0}}$ be the set of $x$ such that
\begin{equation*}
\log|\varphi^{\beta,\eta}_{[0,l-1]}(x+(m-1)\omega_{0},\omega_{0},z_{0})|<lL_{l}(\omega_{0},z_{0})-l^{1-\tau},\quad m\in [0,n-1].
\end{equation*}
According to LDT, we can obtain
\begin{equation*}
\mathrm{mes}(\mathcal{B}_{n,\omega_{0},z_{0}})<\exp(-Cl^{v})<n\exp(-Cl^{v})<\exp(-\frac{Cl^{v}}{2}).
\end{equation*}
By the covering form of LDT, for any $x\notin \mathcal{B}_{n,\omega_{0},z_{0}}$, we have
\begin{equation*}
\mathrm{dist}(z,\sigma(\mathcal{E}_{[0,n-1]}^{\beta,\eta}))\geq \exp(-l).
\end{equation*}
\end{proof}

We extend the Wegner estimate to control resonances between two distant intervals, showing that for most frequencies and phases, at least one of the intervals has a bounded Green's function (i.e., no resonance). This directly eliminates double resonances.
\begin{lemma}\label{Lemma7.2}
Let $\varpi\in (0,1)$, $\bar{l}\geq 1$ and $\nu$ be as in LDT. Let $\Lambda_{0}, \Lambda_{1}$ be intervals in $\mathbb{Z}$ such that $\bar{l}/10\leq |\Lambda_{0}|, |\Lambda_{1}|\leq 10\bar{l}$ and $0\in \Lambda_{0}, \Lambda_{1}$.
For any $x_{0}\in \mathbb{T}^{d}$, $\bar{l}\geq \bar{l_{0}}(\alpha_{i},p,q,\varpi,\gamma)$, $1\leq t\leq \exp(c_{0}\bar{l}^{\varpi\nu})$, there exists a set $\Omega_{\bar{l},t_{0},x_{0}}$, $\mathrm{mes}(\Omega_{\bar{l},t_{0},x_{0}})<\bar{l}^{C(p,q)}/t_{0}$ such that the following statement holds. For all $\omega\in \mathbb{T}^{d}(p,q)\backslash\Omega_{\bar{l},t_{0},x_{0}}$, $z\in \partial \mathbb{D}$, such that $L(\omega,z)>\gamma$, we have
\begin{equation*}
\min(\|(\mathcal{E}_{\Lambda_{0}}^{\beta,\eta}(\omega,x_{0})-z)^{-1}\|,\|(\mathcal{E}_{\Lambda_{1}}^{\beta,\eta}(\omega,x_{0}+t\omega)-z)^{-1}\|)\leq\exp(\bar{l}^{\varpi})
\end{equation*}
for all $t_{0}\leq |t|\leq \exp(c_{0}\bar{l}^{\varpi\nu})$.
\end{lemma}
\begin{proof}
Fix $x_{0}\in \mathbb{T}^{d}$. Let $\mathcal{B}$ be the set of $(\omega,z,x)\in \mathbb{T}^{d}(p,q)\times \partial \mathbb{D}\times \mathbb{T}^{d}$, such that $L(\omega,z)>\gamma$,
\begin{equation*}
\|(\mathcal{E}_{\Lambda_{0}}^{\beta,\eta}(\omega,x_{0})-z)^{-1}\|>\exp(\bar{l}^{\varpi}),\,\, \mathrm{and}\,\,\|(\mathcal{E}_{\Lambda_{1}}^{\beta,\eta}(\omega,x)-z)^{-1}\|>\exp(\bar{l}^{\varpi}).
\end{equation*}
Let $\tilde{\alpha}_{i}$ be as in \eqref{VCA}. Then $\mathcal{B}$ is contained in a semialgebraic set $\tilde{\mathcal{S}}$ of
\begin{equation*}
(\omega,z,x)\in \mathbb{T}_{J}^{d}(p,q)\times \partial \mathbb{D}\times \mathbb{T}^{d},\,\,J=\bar{l}^{C(p,q)}
\end{equation*}
satisfying
\begin{equation*}
\frac{1}{|\Lambda_{1}|J}\sum_{j=1}^{J}\log\|\tilde{M}_{|\Lambda_{1}|}(\omega,z;x+j\omega)\|_{\mathrm{HS}}\geq\frac{\gamma}{2},
\end{equation*}
\begin{equation*}
\|(\tilde{\mathcal{E}}_{\Lambda_{0}}^{\beta,\eta}(\omega,x_{0})-z)^{-1}\|>\frac{1}{2}\exp(\bar{l}^{\varpi}),\,\, \mathrm{and}\,\,\|(\tilde{\mathcal{E}}_{\Lambda_{1}}^{\beta,\eta}(\omega,x)-z)^{-1}\|>\frac{1}{2}\exp(\bar{l}^{\varpi}).
\end{equation*}
It is clear that the degree of $\tilde{\mathcal{S}}$ is bounded above by $\bar{l}^{C}$. Furthermore, for $(\omega,z,x)\in\tilde{\mathcal{S}}$ one can obtain that
$$L_{|\Lambda_{1}|}(\omega,z)\geq\frac{\gamma}{4},$$
\begin{equation}\label{Lem7.2-(1)}
\|(\mathcal{E}_{\Lambda_{0}}^{\beta,\eta}(\omega,x_{0})-z)^{-1}\|>\frac{1}{4}\exp(\bar{l}^{\varpi}),\,\, \mathrm{and}\,\,\|(\mathcal{E}_{\Lambda_{1}}^{\beta,\eta}(\omega,x)-z)^{-1}\|>\frac{1}{4}\exp(\bar{l}^{\varpi}).
\end{equation}
For $(\omega,z,x)$ satisfying \eqref{Lem7.2-(1)} and some $z_{0}\in \sigma(\mathcal{E}_{\Lambda_{0}}^{\beta,\eta}(\omega,x_{0}))$, we have $\mathrm{dist}(z_{0},\sigma(\mathcal{E}_{\Lambda_{1}}^{\beta,\eta}(\omega,x)))<8\exp(-\bar{l}^{\varpi})$. Taking a union over $z_{0}$ and according to Lemma \ref{Wegner's-estimate}, we have
$$\mathrm{mes}(\mathrm{Proj}_{(\omega,x)}\tilde{\mathcal{S}})<\exp(-c\bar{l}^{\varpi\nu}).$$
Set $\mathcal{S}:=\mathrm{Proj}_{(\omega,x)}\tilde{\mathcal{S}}.$ Let $\mathcal{S}=\mathcal{S}_{1}\cup\mathcal{S}_{2}$ be the decomposition of $\mathcal{S}$ corresponding to Lemma \ref{Lemma7.1} with $\varepsilon=\frac{200}{t_{0}}$. In order to get the conclusion, we just need that $(\omega,\{x_{0}+t\omega\})\notin \mathcal{S}$.

Let $T$ be integers in the range $t_{0}\leq|t|\leq\exp(c_{0}\bar{l}^{\varpi\nu})$. Define
\begin{equation*}
\Omega_{\bar{l},t_{0},x_{0}}(\Lambda_{0},\Lambda_{1})=\mathrm{Proj}_{\omega}\mathcal{S}_{1}\cup(\mathop{\cup}\limits_{t\in T}\Omega_{\bar{l},t,x_{0}}(\Lambda_{0},\Lambda_{1})),
\end{equation*}
\begin{equation*}
\Omega_{\bar{l},t,x_{0}}(\Lambda_{0},\Lambda_{1}):=\{\omega:(\omega,\{x_{0}+t\omega\})\in\mathcal{S}_{2}\},
\end{equation*}
\begin{equation*}
\Omega_{\bar{l},t_{0},x_{0}}=\mathop{\cup}\limits_{\Lambda_{0},\Lambda_{1}}\Omega_{\bar{l},t_{0},x_{0}}(\Lambda_{0},\Lambda_{1}).
\end{equation*}
Note that $\{\omega:(\omega,\{x_{0}+t\omega\})\in\mathcal{S}_{1}\}\subset \mathrm{Proj}_{\omega}\mathcal{S}_{1}$. From the assumptions, there are less than $C\bar{l}^{4}$ possible choices for $\Lambda_{0}, \Lambda_{1}$. Then we only need to estimate the measure of $\Omega_{\bar{l},t_{0},x_{0}}$. Note that the set of $(\omega,\{x_{0}+t\omega\})$, $\omega\in[0,1]^{d}$, is contained in a union of planes $L_{i}$, $i\leq |t|^{d}$. The planes $L_{i}$ are parallel to $(\omega,t\omega)$, $\omega\in \mathbb{R}^{d}$, then
\begin{equation*}
|\mathrm{Proj}_{L_{i}}e_{j}|\leq\frac{1}{|t|}\leq\frac{1}{t_{0}}<\frac{\varepsilon}{100}\,\,\mathrm{for}\,\,\mathrm{all}\,\,i,\,j,
\end{equation*}
where $e_{j}$ are as in Lemma \ref{Lemma7.1}. By Lemma \ref{Lemma7.1},
$$\mathrm{mes}(\mathrm{Proj}_{\omega}\mathcal{S}_{1})<\bar{l}^{C}/t_{0},$$
\begin{equation*}
\mathrm{mes}(\Omega_{\bar{l},t,x_{0}}(\Lambda_{0},\Lambda_{1}))=\sum_{i}\mathrm{mes}(\mathcal{S}_{2}\cap L_{i})\lesssim |t|^{d}\bar{l}^{C}t_{0}\exp(-c\bar{l}^{\varpi\nu})\leq\exp(-c'\bar{l}^{\varpi\nu}),
\end{equation*}
then we can obtain the conclusion.
\end{proof}

We further extend the result to hold for most phases, ensuring that for most frequencies and phases, double resonances between distant intervals are excluded.
\begin{coro}\label{Coro7.3}
We use the notation and assumptions of Lemma \ref{Lemma7.2}. For any $\bar{l}\geq \bar{l}_{0}(\alpha_{i},p,q,\varpi,\gamma)$ and $1\leq t_{0}\leq\exp(c_{0}\bar{l}^{\varpi\nu})$, there exists a set $\Omega_{\bar{l},t_{0}}$, $\mathrm{mes}(\Omega_{\bar{l},t_{0}})<\bar{l}^{C(p,q)}t_{0}^{-1/2}$, such that for any $\omega\notin \Omega_{\bar{l},t_{0}}$ there exists a set $\mathcal{B}_{\bar{l},t_{0},\omega}$, $\mathrm{mes}(\mathcal{B}_{\bar{l},t_{0},\omega})<\bar{l}^{C(p,q)}t_{0}^{-1/2}$, and the following statement holds. For any $\omega\in \mathbb{T}^{d}(p,q)\backslash\Omega_{\bar{l},t_{0}}$, $x\in \mathbb{T}^{d}\backslash\mathcal{B}_{\bar{l},t_{0},\omega}$, $z\in \partial \mathbb{D}$, such that $L(\omega,z)>\gamma$, we have
\begin{equation*}
\min(\|(\mathcal{E}_{\Lambda_{0}}^{\beta,\eta}(\omega,x)-z)^{-1}\|,\|(\mathcal{E}_{\Lambda_{1}}^{\beta,\eta}(\omega,x+t\omega)-z)^{-1}\|)\leq\exp(\bar{l}^{\varpi})
\end{equation*}
for all $t_{0}\leq |t|\leq\exp(c_{0}\bar{l}^{\varpi\nu})$.
\end{coro}
\begin{proof}
Let $\mathcal{B}_{\bar{l},t_{0}}$ be the set of $x\in \mathbb{T}^{d}$, $\omega\in\mathbb{T}^{d}(p,q)$ such that $\omega\in \Omega_{\bar{l},t_{0},x}$, with $\Omega_{\bar{l},t_{0},x}$ as in Lemma \ref{Lemma7.2}. With the aid of Chebyshev's inequality, there exists a set $\Omega_{\bar{l},t_{0}}$, $\mathrm{mes}(\Omega_{\bar{l},t_{0}})\leq (\mathrm{mes}(\mathcal{B}_{\bar{l},t_{0}}))^{\frac{1}{2}}$ such that for $\omega\notin\Omega_{\bar{l},t_{0}}$ we have
$\mathrm{mes}(\mathcal{B}_{\bar{l},t_{0},\omega})\leq(\mathrm{mes}(\mathcal{B}_{\bar{l},t_{0}}))^{\frac{1}{2}}$, where $\mathcal{B}_{\bar{l},t_{0},\omega}=\{x:(x,\omega)\in \mathcal{B}_{\bar{l},t_{0}}\}$. Then the conclusion follows by applying Lemma \ref{Lemma7.2}.
\end{proof}

\section{Removing Double Resonances under NDR Condition}\label{section5}
To begin with, we consider the local factorization for $\varphi_{\Lambda}^{\beta,\eta}(\omega,z;x)$ about the spectral variable, where $\Lambda$ is a finite interval in $\mathbb{Z}$. Inspired by \cite[Section 4]{GSV16-arXiv}, we will introduce the theory about ``no double resonances NDR" condition in this section.
\begin{definition 2}
Let $\tau, \nu$ be as in LDT. An interval $\Lambda\subset \mathbb{Z}$ is said to be $(K,l,C)$-NDR with respect to $x_{0}, \omega_{0}, z_{0}$ if there exists a subset $\underline{\Lambda}\subset\Lambda$ with $|\underline{\Lambda}|\leq K$ such that
\begin{equation*}
\log|\varphi_{[0,l-1]}^{\beta,\eta}(\omega_{0},z_{0};x_{0}+(n-1)\omega_{0})|>lL_{l}(\omega_{0},z_{0})-Cl^{1-\tau/3}
\end{equation*}
for all $n\in \Lambda\backslash \underline{\Lambda}$. Moreover, the connected components of $\Lambda\backslash \underline{\Lambda}$ must have length  exceeding $l^{2/\nu}$. If $C=1$, $\Lambda$ is referred to as $(K,l)$-NDR.
\end{definition 2}

For our purpose, we need the following Weierstrass' Preparation Theorem and the perturbation theory of matrices.

Consider an analytic function $f(z,\omega_{1},\ldots,\omega_{d})$ defined in a polydisk
\begin{equation*}
\mathcal{P}_{*}=\mathcal{D}(z_{0},R_{0})\times \prod_{j=1}^{d}\mathcal{D}(\omega_{j,0},R_{0})\,\,z_{0},\,\omega_{j,0}\in \mathbb{C}, \,\, R_{0}>0.
\end{equation*}
\begin{lemma}\cite[Lemma 2.28]{GSV16-arXiv}\label{Weierstrass}
Assume that $f(\cdot,\omega_{1},\ldots,\omega_{d})$ has no zeros on some circle
\begin{equation*}
\{z:|z-z_{0}|=r\}\,\, 0<r<R_{0}/2,
\end{equation*}
for any $\underline{\omega}=(\omega_{1},\ldots,\omega_{d})\in \mathcal{P}=\mathop{\prod}\limits_{j=1}^{d}\mathcal{D}(\omega_{j,0},r_{j,0})$ where $0<r_{j,0}<R_{0}$. Then there exist a polynomial $P(z,\underline{\omega})=z^{k}+a_{k-1}(\underline{\omega})z^{k-1}+\cdots+a_{0}(\underline{\omega})$ with $a_{j}(\underline{\omega})$ analytic in $\mathcal{P}$ and an analytic function $g(z,\underline{\omega})$, $(z,\underline{\omega})\in\mathcal{D}(z_{0},r)\times\mathcal{P}$ so that the following properties hold:\\
(a) $f(z,\underline{\omega})=P(z,\underline{\omega})g(z,\underline{\omega})$ for any $(z,\underline{\omega})\in \mathcal{D}(z_{0},r)\times\mathcal{P}$,\\
(b) $g(z,\underline{\omega})\neq 0 $ for any $(z,\underline{\omega})\in \mathcal{D}(z_{0},r)\times\mathcal{P}$,\\
(c) for any $\underline{\omega}\in\mathcal{P}$, $P(\cdot,\underline{\omega})$ has no zeros in $\mathbb{C}\backslash\mathcal{D}(z_{0},r)$.
\end{lemma}

We also use matrix perturbation theory to ensure that eigenvalues of NDR intervals are stable under small perturbations of frequency and phase.
\begin{lemma}\cite[Lemma 2.37]{GSV16-arXiv}\label{perturbation(1)}
Let $H, H_{0}$ be $n\times n$ matrices, $H_{0}$ is Hermitian, $E_{0}\in \mathbb{R}$, $r_{0}>0$. Assume the number of eigenvalues of $H_{0}$ in $(E_{0}-r_{0},E_{0}+r_{0})$ is at most $K$ and
\begin{equation*}
\|H-H_{0}\|\leq\frac{r_{0}}{32(K+1)^{2}}.
\end{equation*}
Then there exists $r_{0}/2<r<r_{0}$, which depends only on $H_{0}$, such that $H$ and $H_{0}$ have the same number of eigenvalues in the disk $\mathcal{D}(E_{0},r)$. Moreover, neither $H$ nor $H_{0}$ have eigenvalues in the region
\begin{equation*}
r-\frac{r_{0}}{8(K+1)}\leq|\zeta-E_{0}|\leq r+\frac{r_{0}}{8(K+1)}.
\end{equation*}
\end{lemma}

\begin{remark}
In \cite{GSV16-arXiv}, during the proof of the above lemma, the authors used the basic property of Hermitian matrix $H_{0}$, that is, $\|(H_{0}-E)^{-1}\|=\frac{1}{\mathrm{dist}(E,\mathrm{spec}(H_{0}))}$ for $E\notin \mathrm{spec}(H_{0})$. Actually, this formula also holds true for unitary matrices. Therefore, we have the following corollary. Since the proof of the corollary below is similar to the above lemma, we will no longer prove it.
\end{remark}

\begin{coro}\label{Coro6.3}
Let $\mathcal{E}_{0}, \mathcal{E}_{1}$ be $n\times n$ matrices, $\mathcal{E}_{0}$ is unitary, $z_{0}\in \partial \mathbb{D}$, $r_{0}>0$. Assume the number of eigenvalues of $\mathcal{E}_{0}$ in $\mathcal{D}(z_{0},r_{0})\cap \partial \mathbb{D}$ is at most $K$ and
\begin{equation*}
\|\mathcal{E}_{1}-\mathcal{E}_{0}\|\leq\frac{r_{0}}{32(K+1)^{2}}.
\end{equation*}
Then there exists $r_{0}/2<r<r_{0}$, which depends only on $\mathcal{E}_{0}$, such that $\mathcal{E}_{1}$ and $\mathcal{E}_{0}$ have the same number of eigenvalues in the arc $\mathcal{D}(z_{0},r)\cap \partial \mathbb{D}$. Moreover, neither $\mathcal{E}_{1}$ nor $\mathcal{E}_{0}$ have eigenvalues in the region
\begin{equation*}
\Big\{\zeta\in\partial \mathbb{D}:r-\frac{r_{0}}{8(K+1)}\leq|\zeta-z_{0}|\leq r+\frac{r_{0}}{8(K+1)}\Big\}.
\end{equation*}
\end{coro}

Using the unitary perturbation result (Corollary \ref{Coro6.3}), we show that eigenvalues of NDR intervals are isolated from a fixed circle under small perturbations.
\begin{lemma}\label{Lemma6.4}
Suppose $x_{0}\in \mathbb{T}^{d}$, $\omega_{0}\in \mathbb{T}^{d}(p,q)$, $z_{0}\in \partial \mathbb{D}$ and $L(\omega_{0},z_{0})>\gamma>0$. Let $K\geq 1$, $l\geq l_{0}(\alpha_{i},p,q,z_{0},\gamma)$, $r_{0}=\exp(-l)$. Assume $\Lambda$ is $(K,l)$-NDR with respect to $x_{0}, \omega_{0}, z_{0}$. There exists $\frac{r_{0}}{2}<r<r_{0}$ such that for any $(\omega,x)\in \mathbb{C}^{d}\times\mathbb{C}^{d}$ satisfying
\begin{equation}\label{Lem6.4-(1)}
|x-x_{0}|<\frac{c(\alpha_{i})r_{0}}{(K+1)^{2}},\,\, |\omega-\omega_{0}|<\frac{c(\alpha_{i})r_{0}}{|\Lambda|(K+1)^{2}},
\end{equation}
we have
\begin{equation*}
\mathrm{dist}(\{z\in \partial \mathbb{D}: \varphi_{\Lambda}^{\beta,\eta}(\omega,z;x)=0\}, \{z\in \partial \mathbb{D}: |z-z_{0}|=r\})\geq\frac{r_{0}}{8(K+1)}.
\end{equation*}
\end{lemma}
\begin{proof}
Due to \eqref{Lem6.4-(1)},
\begin{equation*}
\|\mathcal{E}_{\Lambda}^{\beta,\eta}(\omega,x)-\mathcal{E}_{\Lambda}^{\beta,\eta}(\omega_{0},x_{0})\|\leq C(|x-x_{0}|+|\Lambda||\omega-\omega_{0}|)\leq \frac{r_{0}}{32(K+1)^{2}}.
\end{equation*}
Then the conclusion follows from Corollary \ref{Coro6.3}.
\end{proof}

\begin{remark}
As we know, the restriction of Schr\"{o}dinger operator to a finite interval is  a Hermitian matrix which possesses many good properties. While the restriction of CMV matrix is a unitary matrix after modifying the boundary conditions whose properties are worse than those of a Hermitian matrix. For example, there is a famous Cauchy Interlacing Theorem for the eigenvalues of Hermitian matrices. But we can not expect such a result for unitary matrices. As a result, we can not estimate the number of the roots for $\varphi_{\Lambda}^{\beta,\eta}(\omega,z;x)=0$ in some disk. Consequently, when we study the factorization of $\varphi_{\Lambda}^{\beta,\eta}(\omega,z;x)$, we can not obtain the property of the factor of it as the property (d) in \cite[Proposition 5.2]{GSV16-arXiv}. However, without this property, we still can obtain the last statement of the following lemma.
\end{remark}

We also use matrix perturbation theory to ensure that eigenvalues of NDR intervals are stable under small perturbations of frequency and phase.
\begin{lemma}\label{Lemma6.5}
Suppose $x_{0}\in \mathbb{T}^{d}$, $\omega_{0}\in \mathbb{T}^{d}(p,q)$, $z_{0}\in \partial \mathbb{D}$ and $L(\omega_{0},z_{0})>\gamma>0$. Let $\tau, \nu$ be as in LDT. Let $K\geq 1$, $l\geq l_{0}(\alpha_{i},p,q,z_{0},\gamma)$, $r_{0}=\exp(-l)$. Assume $\Lambda$ is $(K,l)$-NDR with respect to $x_{0}, \omega_{0}, z_{0}$. There exist
$$P(\omega,x,z)=z^{k}+a_{k-1}(\omega,x)z^{k-1}+\cdots+a_{0}(\omega,x)$$
with $a_{j}(\omega,x)$ analytic in the polydisk
$$\mathcal{P}:=\{(\omega,x):|\omega-\omega_{0}|<cr_{0}|\Lambda|^{-1}(K+1)^{-2},|x-x_{0}|<cr_{0}(K+1)^{-2}\},$$
and an analytic function $g(\omega,x,z)$ on $\mathcal{P}\times \mathcal{D}(z_{0},r)$, $r_{0}/2<r<r_{0}$ such that:

(a) $\varphi_{\Lambda}^{\beta,\eta}(\omega,z;x)=P(\omega,x,z)g(\omega,x,z)$ for any $(\omega,x, z)\in \mathcal{P}\times\mathcal{D}(z_{0},r)$;

(b) $g(\omega,x,z)\neq 0 $ for any $(\omega,x,z)\in \mathcal{P}\times \mathcal{D}(z_{0},r)$;

(c) for any $(\omega,x)\in\mathcal{P}$, $P(\omega,x,\cdot)$ has no zeros in $\mathbb{C}\backslash\mathcal{D}(z_{0},r)$;

(d) if $(\omega,x)\in\mathcal{P}\cap(\mathbb{T}^{d}(p,q)\times\mathbb{T}^{d})$, $z\in \mathcal{D}(z_{0},r)$, and $\frac{r_{0}}{16(K+1)}\geq\exp(-|\Lambda|^{\nu/2})$, \\
then
\begin{equation}\label{Lem6.5-(1)}
\log|g(\omega,x,z)|>|\Lambda|L_{|\Lambda|}(\omega,z)-|\Lambda|^{1-\tau/2}.
\end{equation}
\end{lemma}
\begin{proof}
By Lemma \ref{Lemma6.4}, all the necessary conditions for Lemma \ref{Weierstrass} are satisfied, which leads to the conclusions (a)-(c). Then we need to verify (d).

Note that $P(\omega,x,z)$ is a product of factors $z-z_{j}^{\Lambda}(\omega,x)$ with $z_{j}^{\Lambda}(\omega,x)\in\mathcal{D}(z_{0},r)\cap \partial \mathbb{D}$, where $z_{j}^{\Lambda}(\omega,x)$ is the eigenvalue of $\mathcal{E}_{\Lambda}^{\beta,\eta}$. Thus, for $z\in \mathcal{D}(z_{0},r)\cap \partial \mathbb{D}$, $|P(\omega,x,z)|\leq (2r)^{k}<1$. Then we obtain
$$\log|g(\omega,x,z)|>\log|\varphi_{\Lambda}^{\beta,\eta}(\omega,z;x)|$$
for any $z\in \mathcal{D}(z_{0},r)\cap \partial \mathbb{D}$.

Let $r-\frac{r_{0}}{16(K+1)}<r'<r$. From Lemma \ref{Lemma6.4}, for any $z\in \{z\in \partial \mathbb{D}: |z-z_{0}|=r'\}$ we have
$$\mathrm{dist}(z,\sigma(\mathcal{E}_{\Lambda}^{\beta,\eta}(\omega,x)))\geq \frac{r_{0}}{16(K+1)}\geq \exp(-|\Lambda|^{\nu/2}).$$
According to the spectral form of LDT, we have
$$\log|\varphi_{\Lambda}^{\beta,\eta}(\omega,z;x)|>|\Lambda|L_{|\Lambda|}(\omega,z)-|\Lambda|^{1-\tau/2}.$$
Therefore, \eqref{Lem6.5-(1)} holds. Since the left-hand side of \eqref{Lem6.5-(1)} is harmonic (in $z$) and the right-hand side is subharmonic, it follows that the estimate holds for all $z\in\mathcal{D}(z_{0},r')\cap \partial \mathbb{D}$. Since this is true for $r'$ arbitrarily close to $r$, the conclusion follows.
\end{proof}

Let us recall some basic theory about the resultant of two polynomials that we will use later.
\begin{definition 3}\cite[Section 2.9]{GSV16-arXiv}
Let $f(z)=a_{k}z^{k}+a_{k-1}z^{k-1}+\cdots+a_{0}$, $g(z)=b_{m}z^{m}+b_{m-1}z^{m-1}+\cdots+b_{0}$ be polynomials with complex coefficients. Let $\zeta_{i}$, $1\leq i\leq k$ and $\eta_{j}$, $1\leq j\leq m$ be the zeros of $f(z)$ and $g(z)$ respectively. The resultant of $f$ and $g$ is defined as follows:
\begin{equation}\label{resultant}
\mathrm{Res}(f,g)=a_{k}^{m}b_{m}^{k}\prod_{i,j}(\zeta_{i}-\eta_{j}).
\end{equation}
\end{definition 3}

We use the following property of resultants to link non-vanishing resultants to uniform lower bounds on the polynomials, ensuring no common zeros.
\begin{lemma}\cite[Lemma 2.29]{GSV16-arXiv}\label{Lemma8.1}
Let $f$, $g$ be polynomials as above and set
\begin{equation*}
s=\max(k, m),\,\,r=\max(\max|\zeta_{i}|,\max|\eta_{j}|).
\end{equation*}
Let $\delta\in (0,1)$. If $|\mathrm{Res}(f,g)|>\delta$ and $r\leq \frac{1}{2}$, then
$$\max(|f(z)|,|g(z)|)>(\delta/2)^{s}\,\,for\,\,all\,\, z.$$
\end{lemma}

Finally, we combine the NDR factorization, resultant bounds, and semialgebraic set tools to eliminate double resonances between two distant NDR intervals. We show that for most frequencies and phases, at least one of the intervals has a non-vanishing polynomial factor (no resonance), ruling out common eigenvalues.
\begin{lemma}\label{Lemma8.2}
Suppose $x_{0}\in \mathbb{T}^{d}$, $\omega_{0}\in \mathbb{T}^{d}(p,q)$, $z_{0}\in \partial \mathbb{D}$ and $L(\omega_{0},z_{0})>\gamma$. Let $\tau, \nu$ be as in LDT and $1\leq l\leq \bar{l}\leq\exp(l)$, $1\leq K\leq \bar{l}^{\frac{1-\tau}{10}}$, $\exp(l^{2})\leq|t|\leq\exp(\bar{l}^{\frac{\nu(1-\tau)}{4}})$. Let $\Lambda_{j}$, $j=0,1$, be intervals in $\mathbb{Z}$ such that $\bar{l}/10\leq |\Lambda_{j}|\leq 10\bar{l}$, $0\in \Lambda_{j}$, and $\Lambda_{j}$ is $(K,l)$-NDR with respect to $x_{j}$, $\omega_{0}$, $z_{0}$, $x_{j}:=x_{0}+jt\omega_{0}$, $j=0, 1$. For any $l\geq l_{0}(\alpha_{i},p,q,z_{0},\gamma)$ there exists a set $\Omega_{\bar{l}}$, $\mathrm{mes}(\Omega_{\bar{l}})<\exp(-l^{2}/4)$, such that if $\omega_{0}\notin\Omega_{\bar{l}}$ the following statement holds. There exists a set $\mathcal{B}_{\bar{l},t}(x_{0},\omega_{0},z_{0})$, such that
$$\mathrm{mes}(\mathcal{B}_{\bar{l},t})<\exp(-\bar{l}^{\frac{1-\tau}{6d}})/|t|^{d},$$
and for any $(\omega,x)\in(\mathbb{T}^{d}(p,q)\times\mathbb{T}^{d})\backslash \mathcal{B}_{\bar{l},t}$, $z\in \partial \mathbb{D}$,
\begin{equation}\label{Lem8.2-(1)}
|x-x_{0}|<\exp(-4l),\,\,|\omega-\omega_{0}|<\exp(-4l)/|t|,\,\,|z-z_{0}|<\exp(-l)/2,
\end{equation}
we have
\begin{equation*}
\max_{j=0,1}(\log|\varphi_{\Lambda_{j}}^{\beta,\eta}(\omega,z;x+jt\omega)|-|\Lambda_{j}|L_{|\Lambda_{j}|}(\omega,z)+2|\Lambda_{j}|^{1-\tau/2})>0.
\end{equation*}
\end{lemma}
\begin{proof}
According to Lemma \ref{Lemma3.8}, $L(\omega_{0},z)>\frac{\gamma}{2}$ for any $|z-z_{0}|<\exp(-l)$ (provided $l$ is large enough). Let $\Omega_{\bar{l},t_{0}}$, $\mathcal{B}_{\bar{l},t_{0},\omega_{0}}$ be the sets from Corollary \ref{Coro7.3}, then
$$\mathrm{mes}(\Omega_{\bar{l},t_{0}}), \mathrm{mes}(\mathcal{B}_{\bar{l},t_{0},\omega_{0}})<\exp(-l^{2}/4).$$
Set $\Omega_{\bar{l}}:=\Omega_{\bar{l},t_{0}}$. Suppose $\omega_{0}\notin \Omega_{\bar{l}}$. Let $x_{0}'\in \mathbb{T}^{d}$, $|x_{0}'-x_{0}|<\exp(-l^{2}/(5d))$ be such that $x_{0}'\notin \mathcal{B}_{\bar{l},t_{0},\omega_{0}}$. We can obtain
\begin{equation}\label{Lem8.2-(2)}
\min(\|(\mathcal{E}_{\Lambda_{0}}^{\beta,\eta}(\omega_{0},x_{0}')-z)^{-1}\|,\|(\mathcal{E}_{\Lambda_{1}}^{\beta,\eta}(\omega_{0},x_{0}'+t\omega_{0})-z)^{-1}\|)\leq \exp(\bar{l}^{\frac{1-\tau}{3}})
\end{equation}
for all $|z-z_{0}|<\exp(-l)$ and $\exp(l^{2})\leq|t|\leq\exp(c\bar{l}^{\nu(1-\tau)/3})$.

Let $\varphi_{\Lambda_{j}}^{\beta,\eta}(\omega,z;x)=P_{j}(\omega,x,z)g_{j}(\omega,x,z)$, $j=0,1$, be the factorizations from Lemma \ref{Lemma6.5}. We know that
$$P_{j}(\omega,x,z)=z^{k_{j}}+a_{j,k_{j}-1}(\omega,x)z^{k_{j}-1}+\cdots+a_{j,0}(\omega,x)$$
with $a_{j,i}(\omega,x)$ analytic in a polydisk containing
$$\mathcal{P}_{j}:=\{(\omega,x)\in \mathbb{C}^{2d}:|\omega-\omega_{0}|<\exp(-3l), |x-x_{j}|<\exp(-2l)\},$$
and all the zeros of $P_{j}(\omega,x,\cdot)$, $(\omega,x)\in\mathcal{P}_{j}$, are contained in $\mathcal{D}(z_{0},r)\cap \partial \mathbb{D}$.

Due to $\exp(l^{2})\leq|t|\leq\exp(\bar{l}^{\frac{\nu(1-\tau)}{4}})$, then $\bar{l}\geq l^{\frac{8}{\nu(1-\tau)}}$. In addition,
$$\frac{r_{0}}{16(K+1)}=\frac{\exp(-l)}{16(K+1)}\geq \frac{\exp(-|\Lambda_{j}|^{\frac{\nu(1-\tau)}{8}})}{16(K+1)}\geq \exp(-|\Lambda_{j}|^{\nu/2}).$$
Thus, the condition needed for (d) of Lemma \ref{Lemma6.5} is satisfied. Then we have
\begin{equation}\label{Lem8.2-(3)}
\log|g_{j}(\omega,x,z)|>|\Lambda_{j}|L_{|\Lambda_{j}|}(\omega,z)-|\Lambda_{j}|^{1-\tau/2}
\end{equation}
for any $(\omega,x)\in \mathcal{P}_{j}\cap(\mathbb{T}^{d}(p,q)\times\mathbb{T}^{d})$ and $z\in \mathcal{D}(z_{0},\exp(-l)/2)\cap \partial \mathbb{D}$.

Let $R(\omega,x)=\mathrm{Res}(P_{0}(\omega,x,\cdot),P_{1}(\omega,x+t\omega,\cdot))$. Then $R$ is defined on the polydisk
$$\mathcal{P}:=\{(\omega,x)\in \mathbb{C}^{2d}:|\omega-\omega_{0}|<\exp(-3l)/|t|, |x-x_{j}|<\exp(-3l)\}.$$
According to the definition of the resultant and the properties of $P_{j}$, we get
$$R(\omega,x)=\prod(z_{i}^{\Lambda_{0}}(\omega,x)-z_{j}^{\Lambda_{1}}(\omega,x+t\omega)),$$
where the product is only in terms of eigenvalues contained in $\mathcal{D}(z_{0},\exp(-l))\cap \partial \mathbb{D}$. Then we have that $\sup|R(\omega,x)|\leq (2\exp(-l))^{k_{0}k_{1}}<1$.

From \eqref{Lem8.2-(2)}, for $z\in (\sigma(\mathcal{E}_{\Lambda_{0}}^{\beta,\eta}(\omega_{0},x_{0}'))\cup\sigma(\mathcal{E}_{\Lambda_{1}}^{\beta,\eta}(\omega_{0},x_{0}'+t\omega_{0})))\cap\mathcal{D}(z_{0},\exp(-l))\cap \partial \mathbb{D}$ we have that
$$|z_{i}^{\Lambda_{0}}(\omega_{0},x_{0})-z_{j}^{\Lambda_{1}}(\omega_{0},x_{0}'+t\omega_{0})|\geq\exp(-\bar{l}^{\frac{1-\tau}{3}})$$
for the pairs of eigenvalues contained in $\mathcal{D}(z_{0},\exp(-l))\cap \partial \mathbb{D}$. Therefore,
$$|R(\omega_{0},x_{0}')|\geq(\exp(-\bar{l}^{\frac{1-\tau}{3}}))^{k_{0}k_{1}}\geq\exp(-|\Lambda_{0}||\Lambda_{0}|\bar{l}^{\frac{1-\tau}{3}})\geq\exp(-100\bar{l}^{2}\bar{l}^{\frac{1-\tau}{3}}).$$
Applying Cartan's estimate to $R(\omega,x)$ with $M=0$, $m=-100\bar{l}^{2}\bar{l}^{\frac{1-\tau}{3}}$, $H=\bar{l}^{\frac{1-\tau}{3}}$, one can obtain that
$$\log|R(\omega,x)|>-100C\bar{l}^{2}\bar{l}^{\frac{1-\tau}{3}},$$
which implies that $|R(\omega,x)|>\exp(-100C\bar{l}^{2}\bar{l}^{\frac{1-\tau}{3}})$ for all $(\omega,x)\in(\frac{1}{6}\mathcal{P}\cap \mathbb{R}^{2d})\backslash \mathcal{B}$, where $\mathcal{B}=\mathcal{B}(\Lambda_{0},\Lambda_{1})$ and
$$\mathrm{mes}(\mathcal{B})\leq \frac{C\exp(-6dl)}{|t|^{d}}\exp(-H^{\frac{1}{2d}}).$$
Note that $\frac{1}{6}\mathcal{P}\cap \mathbb{R}^{2d}$ contains the points $(\omega,x)$ satisfying \eqref{Lem8.2-(1)}. Let $\mathcal{B}_{\bar{l},t}=\cup \mathcal{B}(\Lambda_{0},\Lambda_{1})$. According to Lemma \ref{resultant}, one can obtain
$$\max(|P_{0}(\omega,x,z)|,|P_{1}(\omega,x+t\omega,z)|)\geq(\frac{1}{2}\exp(-100C\bar{l}^{2}\bar{l}^{\frac{1-\tau}{3}}))^{10\bar{l}}>\exp(-\bar{l}^{5})$$
for all $z$ and $(\omega,x)\notin \mathcal{B}_{\bar{l},t}$. Therefore, the conclusion follows by applying the factorizations of the determinants and \eqref{Lem8.2-(3)}.
\end{proof}

We extend the above result to all pairs of distant points in a large interval, ensuring that double resonances are excluded uniformly across the interval. This uniform exclusion is critical for extending finite scale localization to full scale localization.
\begin{lemma}\label{Lemma8.3}
Let $\tau, \nu$ be as in LDT, $1\leq l\leq \bar{l}\leq\exp(l)$, $1\leq K\leq\bar{l}^{\frac{1-\tau}{10}}$, and $\exp(l^{2})\leq n$. Let $\Lambda_{j}$, $j=0,1$, be intervals in $\mathbb{Z}$ such that $\bar{l}/10\leq |\Lambda_{j}|\leq 10 \bar{l}$, $0\in \Lambda_{j}$,  $j=0,1$. For $l\geq l_{0}(\alpha_{i},p,q,\gamma)$ there exists $\varrho=\varrho(p,q)$ such that if $n\leq \exp(\bar{l}^{\varrho})$, there exists a set $\Omega_{n,\bar{l}}$, $\mathrm{mes}(\Omega_{n,\bar{l}})<2\exp(-l^{2}/4)$, and for every $\omega\in \mathbb{T}^{d}(p,q)\backslash\Omega_{n,\bar{l}}$ there exists a set $\mathcal{B}_{n,\bar{l},\omega}$, $\mathrm{mes}(\mathcal{B}_{n,\bar{l},\omega})<\exp(-\bar{l}^{\varrho})$ such that the following statement holds. For any $\omega\in \mathbb{T}^{d}(p,q)\backslash\Omega_{n,\bar{l}}$, $x\in \mathbb{T}^{d}\backslash\mathcal{B}_{n,\bar{l},\omega}$, $z\in \partial \mathbb{D}$, and $m_{0}, m_{1}\in [0,n-1]$ such that $|m_{0}-m_{1}|\geq\exp(l^{2})$, we have that if $L(\omega,z)>\gamma$ and $\Lambda_{j}$, $j=0,1$, are $(K,l,\frac{1}{2})$-NDR intervals with respect to $x+m_{j}\omega$, $\omega$, $z$, then
$$\max_{j=0,1}(\log|\varphi_{\Lambda_{j}}^{\beta,\eta}(\omega,z;x+m_{j}\omega)|-|\Lambda_{j}|L_{|\Lambda_{j}|}(\omega,z)+2|\Lambda_{j}|^{1-\tau/2})>0.$$
\end{lemma}
\begin{proof}
Let $\Omega_{\bar{l}}$ be the set from Lemma \ref{Lemma8.2} (with $\frac{\gamma}{2}$ instead of $\gamma$). Let
$$\{x: |x-x_{k}|<\exp(-l^{3/2})\},\,\,k\lesssim \exp(dl^{3/2})$$
$$\{\omega: |\omega-\omega_{k'}|<\exp(-l^{3/2})/n\},\,\,k\lesssim n^{d}\exp(dl^{3/2})$$
$$\{z: |z-z_{k''}|<\exp(-l^{3/2})\},\,\,k''\lesssim \exp(l^{3/2})$$
be covers of $\mathbb{T}^{d}$, $\mathbb{T}^{d}\backslash \Omega_{\bar{l}}$ and $\partial \mathbb{D}$, respectively. Note that if $|\omega-\omega_{k'}|<\exp(-l^{3/2})/n$, $|z-z_{k''}|<\exp(-l^{3/2})$, and $L(\omega,z)>\gamma$, then $L(\omega_{k'},z_{k''})>\frac{\gamma}{2}$ (provided $l$ is large enough) by using Lemma \ref{Lemma3.8}. If $L(\omega_{k'},z_{k''})>\frac{\gamma}{2}$, we let $\mathcal{B}_{\bar{l},t}(x_{k},\omega_{k'},z_{k''})$ be the set from Lemma \ref{Lemma8.2} (with $\frac{\gamma}{2}$ instead of $\gamma$). Otherwise we let $\mathcal{B}_{\bar{l},t}(x_{k},\omega_{k'},z_{k''})$ be the empty set.

Let $\mathcal{B}_{n,\bar{l}}=\mathop{\cup}\limits_{k,k',k'',t,m}S_{m}(\mathcal{B}_{\bar{l},t}(x_{k},\omega_{k'},z_{k''}))$,
where $S_{m}(x,\omega)=(x-m\omega,\omega)$, $m\in [0,n-1]$, $|t|\in [\exp(l^{2}),n]$. Then
$$\mathrm{mes}(\mathcal{B}_{n,\bar{l}})\leq C c^{d+2}\exp((2d+1)l^{3/2})\exp(-\bar{l}^{\frac{\nu(1-\tau)}{6d}}).$$
From the restriction of $|t|$ in Lemma \ref{Lemma8.2}, we require that $n\leq\exp(\bar{l}^{\frac{\nu(1-\tau)}{4}})$. Take $\varrho=\varrho(\tau,\nu)=\varrho(p,q)$ such that if $n\leq\exp(\bar{l}^{\varrho})$, the above restriction is satisfied and
$$\mathrm{mes}(\mathcal{B}_{n,\bar{l}})<\exp(-2\bar{l}^{\varrho}).$$
According to Chebyshev's inequality, there exists a set $\tilde{\Omega}_{n,\bar{l}}$, $\mathrm{mes}(\tilde{\Omega}_{n,\bar{l}})<\exp(-\bar{l}^{\varrho})$ such that for $\omega\notin \tilde{\Omega}_{n,\bar{l}}$ we have $\mathrm{mes}(\mathcal{B}_{n,\bar{l},\omega})<\exp(-\bar{l}^{\varrho})$, where $\mathcal{B}_{n,\bar{l},\omega}=\{x:(x,\omega)\in\mathcal{B}_{n,\bar{l}}\}$.

Set $\Omega_{n,\bar{l}}=\Omega_{\bar{l}}\cup\tilde{\Omega}_{n,\bar{l}}$. Let $x, \omega, z, m_{j}, \Lambda_{j}$ be as in the assumptions. Then there exist $x_{k}, \omega_{k'}, z_{k''}$ such that
$$|x+m_{0}\omega-x_{k}|<\exp(-l^{3/2}),\,\,|\omega-\omega_{k'}|<\exp(-l^{3/2})/n,\,\,|z-z_{k''}|<\exp(-l^{3/2}).$$
Due to $L(\omega,z)>\gamma$, then we have $L(\omega_{k'},z_{k''})>\frac{\gamma}{2}$. Since $\Lambda_{j}$, $j=0,1$, are $(K,l,\frac{1}{2})$-NDR intervals with respect to $x+m_{j}\omega, \omega, z$, then $\Lambda_{j}$, $j=0,1$, are
$(K,l)$-NDR intervals with respect to $x_{k}+jt\omega_{k'}, \omega_{k'}, z_{k''}, t=m_{1}-m_{0}$.

The choice of above exceptional sets implies that $(x+,m_{0}\omega,\omega)\notin \mathcal{B}_{\bar{l},t}(x_{k},\omega_{k'},z_{k''})$. Then according to Lemma \ref{Lemma8.2}, we have
$$\max_{j=0,1}(\log|\varphi_{\Lambda_{j}}^{\beta,\eta}(\omega,z;x+m_{0}\omega+jt\omega)|-|\Lambda_{j}|L_{|\Lambda_{j}|}(\omega,z)+2|\Lambda_{j}|^{1-\tau/2})>0.$$
Due to $t=m_{1}-m_{0}$, then $x+m_{0}\omega+jt\omega=x+m_{0}\omega+j(m_{1}-m_{0})\omega=x+m_{j}\omega$, $j=0,1$.

Thus, the conclusion follows.
\end{proof}

\section{Full Scale Localization}\label{section6}
In this section, we formulate the full scale localization. To do that, we require the following lemmas.
\begin{lemma}\cite[Section 2.40]{GSV16-arXiv}
Let $A$ be an $N\times N$ Hermitian matrix. Let $E, \varepsilon\in \mathbb{R}$, $\varepsilon>0$ and suppose there exists $\phi\in \mathbb{R}^{N}$, $\|\phi\|=1$, such that
$$\|(A-E)\phi\|<\varepsilon.$$
Then the following statements hold:

(a) There exists a normalized eigenvector $\psi$ of $A$ with an eigenvalue $E_{0}$ such that
$$E_{0}\in (E-\varepsilon\sqrt{2},E+\varepsilon\sqrt{2}),$$
$$|\langle\phi,\psi\rangle|\geq(2N)^{-1/2}.$$

(b) If in addition there exists $\hat{\varepsilon}>\varepsilon$ such that the subspace of the eigenvectors of $A$ with eigenvalues falling into the interval $(E-\hat{\varepsilon},E+\hat{\varepsilon})$ is at most of dimension one, then there exists a normalized eigenvector $\psi$ of $A$ with an eigenvalue $E_{0}\in (E-\varepsilon,E+\varepsilon)$, such that
$$\|\phi-\psi\|<\sqrt{2}\hat{\varepsilon}^{-1}\varepsilon.$$
\end{lemma}

From the above lemma and its proof, it is obvious that this Lemma can also be applicable to a unitary matrix. Then we have the following corollary. Since the proof is similar to the above lemma, we no longer prove it.
\begin{coro}\label{eigenvector}
Let $A$ be an $N\times N$ unitary matrix. Let $z\in \partial \mathbb{D}$, $\tilde{\varepsilon}>0\in \mathbb{R}$, and suppose there exists $\phi\in \mathbb{C}^{N}$, $\|\phi\|=1$, such that
$$\|(A-z)\phi\|<\tilde{\varepsilon}.$$
Then the following statements hold:

(a) There exists a normalized eigenvector $\psi$ of $A$ with an eigenvalue $z_{0}$ such that
$$z_{0}\in \mathcal{D}(z,\tilde{\varepsilon}\sqrt{2})\cap \partial \mathbb{D},$$
$$|\langle\phi,\psi\rangle|\geq(2N)^{-1/2};$$

(b) If in addition there exists $\hat{\varepsilon}>\tilde{\varepsilon}$ such that the subspace of the eigenvectors of $A$ with eigenvalues falling into the interval $\mathcal{D}(z,\hat{\varepsilon})\cap \partial \mathbb{D}$ is at most of dimension one, then there exists a normalized eigenvector $\psi$ of $A$ with an eigenvalue $z_{0}\in \mathcal{D}(z,\tilde{\varepsilon})\cap \partial \mathbb{D}$, such that
$$\|\phi-\psi\|<\sqrt{2}\hat{\varepsilon}^{-1}\tilde{\varepsilon}.$$
\end{coro}

Next, we show that for most frequencies and phases, eigenfunctions of large finite intervals are localized to a small subinterval, with exponential decay away from this localization center.
\begin{lemma}\label{Lemma9.1}
Let $\varepsilon\in(0,\frac{1}{5})$ and $\tau, \nu$ be as in LDT. For $n\geq N_{0}(\alpha_{i},p,q,\gamma,\varepsilon)$, there exists a set $\Omega_{n}$,
$$\mathrm{mes}(\Omega_{n})<\exp(-(\log n)^{\varepsilon\nu}),$$
such that for $\omega\in \mathbb{T}^{d}(p,q)\backslash\Omega_{n}$ there exists a set $\mathcal{B}_{n,\omega}$,
$$\mathrm{mes}(\mathcal{B}_{n,\omega})<\exp(-(\exp(\log n)^{\varepsilon\nu})^{\varrho})$$
and the following holds for any $\omega\in \mathbb{T}^{d}(p,q)\backslash\Omega_{n}$, $x\in \mathbb{T}^{d}(p,q)\backslash\mathcal{B}_{n,\omega}$ and $z\in \partial \mathbb{D}$ such that $L(\omega,z)>\gamma$, where $\varrho$ is as in Lemma \ref{Lemma8.3}. Let $1\ll \underline{C}_{1}\ll \overline{C}_{1}\ll \cdots\ll \underline{C}_{(s-1)^{2}}\ll \overline{C}_{(s-1)^{2}}$, $s=2^{2d+1}$, be constants such that the sets
$$\mathcal{N}_{k}=\{N\in \mathbb{Z}: \exp(\underline{C}_{k}l^{\nu/2})\leq |N|\leq\exp(\overline{C}_{k}l^{\nu/2})\}$$
satisfy the assumptions of Corollary \ref{Coro5.4}, where $l=\lceil(\log n)^{2\varepsilon}\rceil$.

If $m_{0}\in [0,n-1]$ is such that
$$\log|\varphi_{\Lambda}^{\beta,\eta}(\omega,z;x)|\leq |\Lambda|L_{|\Lambda|}(\omega,z)-|\Lambda|^{1-\tau/4}$$
for all intervals $\Lambda\subset[0,n-1]$ satisfying $\mathrm{dist}(m_{0},[0,n-1]\backslash\Lambda)>|\Lambda|/100$, $|\Lambda|\geq(\log n)^{\varepsilon}$, then for any $\bar{l}\geq (\exp(\overline{C}_{(s-1)^{2}}(\log n)^{3\varepsilon\nu/2}))^{2/\nu}$ we have
\begin{equation}\label{Lem9.1-(1)}
\log|\varphi_{\Lambda}^{\beta,\eta}(\omega,z;x+(m-1)\omega)|> \bar{l}L_{\bar{l}}(\omega,z)-|\Lambda|^{1-\tau/2}
\end{equation}
for any $m\in [0,n-\bar{l}]$ such that $|m-m_{0}|\geq \bar{l}+2\exp(l^{2})$.
\end{lemma}
\begin{proof}
Let $\bar{l}_{k}=\lfloor\exp(\overline{C}_{k}l^{\nu/2}) \rfloor$, $K_{k}=\lceil\exp(\underline{C}_{k}l^{\nu/2})\rceil$. Taking the constants $\overline{C}_{k}, \underline{C}_{k}$ such that $K_{k}\leq \bar{l}_{k}^{\frac{1-\tau}{10}}$, then the condition needed for Lemma \ref{Lemma8.3} is satisfied. We choose
$$\Omega_{n}:=\Omega_{l}\cup (\mathop{\cup}\limits_{k}\Omega_{n,\bar{l}_{k}}),\,\,\mathcal{B}_{n,\omega}:=\mathop{\cup}\limits_{k}\mathcal{B}_{n,\bar{l}_{k},\omega},$$
where $\Omega_{l}$ is the set from Corollary \ref{Coro5.4} and $\Omega_{n,\bar{l}_{k}}$, $\mathcal{B}_{n,\bar{l}_{k},\omega}$ are the sets from Lemma \ref{Lemma8.3}. Then one can obtain that
\begin{equation*}
\mathrm{mes}(\Omega_{n})<l^{C(p,q)}\exp(-\underline{C}_{1}l^{\nu/2})+2(s-1)^{2}\exp(-l^{2}/4)<\exp(-l^{\nu/2})<\exp(-(\log n)^{\varepsilon\nu}),
\end{equation*}
\begin{equation*}
\mathrm{mes}(\mathcal{B}_{n,\omega})<(s-1)^{2}\exp(-\bar{l}_{k}^{\varrho})<(s-1)^{2}\exp(-\bar{l}_{1}^{\varrho})<\exp(-(\exp(\log n)^{\varepsilon\nu})\varrho).
\end{equation*}
Let $\Lambda_{0}=[m_{0}', m_{0}'']\subset [0,n-1]$, $|\Lambda_{0}|=l$, be an interval such that $\mathrm{dist}(m_{0},[0,n-1]\backslash\Lambda_{0})>|\Lambda_{0}|/100$. Then this implies that
\begin{equation*}
\log|\varphi_{\Lambda_{0}}^{\beta,\eta}(\omega,z;x)|=\log|\varphi_{l}^{\beta,\eta}(\omega,z;x+(m_{0}'-1)\omega)|> lL_{l}(\omega,z)-l^{1-\tau/2}.
\end{equation*}
According to Corollary \ref{Coro5.4}, for each $m'\in [0,n-1]$ we either have
\begin{equation*}
\log|\varphi_{l}^{\beta,\eta}(\omega,z;x+(m'-1)\omega)|> lL_{l}(\omega,z)-l^{1-\tau/2}
\end{equation*}
(there exists $k\in \{1,\ldots,(s-1)^{2}\}$ such that the above inequality holds for all $m'\in \mathcal{N}_{k}$) or there exists $k=k(m_{0}',m')$ such that $[-\bar{l}_{k},\bar{l}_{k}]$ is $(K_{k},l,\frac{1}{2})$-NDR with respect to both $x+(m_{0}'-1)\omega, \omega, z$ and $x+(m'-1)\omega, \omega, z$.

From Lemma \ref{Lemma8.3}, for any $m'\in [0,n-1]$ such that $|m'-m_{0}'|\geq\exp(l^{2})$, we have
\begin{equation}\label{Lem9.1-(2)}
\log|\varphi_{\Lambda(m')}^{\beta,\eta}(\omega,z;x+(m'-1)\omega)|> |\Lambda(m')|L_{|\Lambda(m')|}(\omega,z)-2|\Lambda(m')|^{1-\tau/2},
\end{equation}
where $\Lambda(m')$  is either $[0,l-1]$ or $[-\bar{l}_{k},\bar{l}_{k}]$ with $k=k(m_{0}',m')$ as above.

According to the definition of NDR intervals, if $[-\bar{l}_{k},\bar{l}_{k}]$ is NDR with respect to $x+(m_{0}'-1)\omega, \omega, z$, then $[-\bar{l}_{k},\bar{l}_{k}]\cap ([0,n-1]-m_{0}'+1)$ is also NDR with respect to $x+(m_{0}'-1)\omega, \omega, z$. Therefore, Lemma \ref{Lemma8.3} guarantees that if $\Lambda(m')=[-\bar{l}_{k},\bar{l}_{k}]$, then \eqref{Lem9.1-(2)} also holds for any subintervals of $[-\bar{l}_{k},\bar{l}_{k}]$ with length exceeding than $\bar{l}_{k}/10$.

Consider $m\in[0,n-\bar{l}]$ such that $|m-m_{0}|\geq \bar{l}+2\exp{l^{2}}$. Then $[m,m+\bar{l}-1]\subset\{m': |m'-m_{0}'|\geq\exp(l^{2})\}$. Every point in $[m,m+\bar{l}-1]$ is contained in an interval of the form $\Lambda(m')\cap[m,m+\bar{l}-1]$ for some $m'\in[m-1,m+\bar{l}-1-l]$, on which the large deviations estimate applies. Note that we aimed to ensure that $m'+[1,l]\subset[m,m+\bar{l}-1]$, since if a large deviations estimate holds on  $m'+[1,l]$, it does not necessarily hold on $m'+[1,l]\cap[m,m+\bar{l}-1]$. On the other hand, we already noted that when a large deviations estimate holds on $m'+[-\bar{l}_{k},\bar{l}_{k}]$ it also holds on $(m'+[-\bar{l}_{k},\bar{l}_{k}])\cap[m,m+\bar{l}-1]$. Since $\bar{l}\geq (\exp(\overline{C}_{(s-1)^{2}}(\log n)^{3\varepsilon\nu/2}))^{2/\nu}$, then
$$|\Lambda(m')|\leq \bar{l}_{(s-1)^{2}}\leq \bar{l}^{\nu/2}$$
as required for the covering form of LDT. Therefore, \eqref{Lem9.1-(1)} follows from \eqref{Lem9.1-(2)} and the covering form of LDT.
\end{proof}

Building on the above lemma, we prove that localized eigenfunctions of large finite intervals decay exponentially away from their localization center, with the decay rate independent of the interval size.
\begin{lemma}\label{Lemma9.2}
Let $\varepsilon\in(0,\frac{1}{5})$, $\nu$ as in LDT, and $\Omega_{n}$, $\mathcal{B}_{n,\omega}$ as in Lemma \ref{Lemma9.1}. For any $n\geq N_{0}(\alpha_{i},p,q,\gamma,\varepsilon)$, $\omega_{0}\in \mathbb{T}^{d}(p,q)\backslash \Omega_{n}$, $x_{0}\in\mathbb{T}^{d}\backslash\mathcal{B}_{n,\omega}$, and any eigenvalue $z_{0}=z_{j}^{[0,n-1]}(\omega_{0},x_{0})$, such that $L(\omega_{0},z_{0})>\gamma$, there exists an interval $I=I(\omega_{0},z_{0},x_{0})\subset[0,n-1]$,
$$|I|<\exp((\log n)^{5\varepsilon}),$$
such that for any $(\omega,x)\in \mathbb{T}^{d}(p,q)\times\mathbb{T}^{d}$,  $|\omega-\omega_{0}|, |x-x_{0}|<\exp(-\exp((\log n)^{2\varepsilon\nu}))$, we have
$$\Big|u_{j}^{[0,n-1]}(\omega,x;s)\Big|<\exp(-\frac{\gamma}{4}\mathrm{dist}(s,I)),$$
provided $\mathrm{dist}(s,I)>\exp((\log n)^{2\varepsilon\nu})$.
\end{lemma}
\begin{proof}
Let $u$ be a normalized eigenvector corresponding to $z_{0}$. Let $m_{0}$ be such that
$$|u(m_{0})|=\max\limits_{i\in(0,n-1)}|u(i)|.$$
Then $m_{0}$ satisfies the assumptions of Lemma \ref{Lemma9.1} because otherwise we can derive a contradiction based on the proof of Lemma \ref{covering-lemma}. Let $\bar{l}=\lceil (\exp(\overline{C}_{(s-1)^{2}}(\log n)^{3\varepsilon\nu/2}))^{2/\nu}\rceil$. According to Lemma \ref{Lemma9.1} (with $s$ instead of $m$), \eqref{Lem9.1-(1)} holds for $s\in [0,n-1]\backslash I$, $I=[0,n-1]\cap [m_{0}-\bar{l}-2\exp(l^{2}), m_{0}+\bar{l}+2\exp(l^{2})]$, $l=\lceil(\log n)^{2\varepsilon}\rceil$. Then we can obtain the conclusion by applying Theorem \ref{finite-scale-localization} (with $\bar{l}$ instead of $l$).
\end{proof}

To ensure eigenpair convergence, we show that eigenvalues corresponding to localized eigenfunctions are well-separated. This prevents eigenvalue crossings as the interval expands, ensuring each localized eigenfunction converges to a unique eigenfunction of the infinite matrix.
\begin{lemma}\label{Lemma9.3}
Let $\varepsilon\in (0,\frac{1}{5})$ and $\Omega_{n}$, $\mathcal{B}_{n,\omega}$ be as in Lemma \ref{Lemma9.2}. For any $n\geq N_{0}(\alpha_{i},p,q,\gamma,\varepsilon)$, $\omega_{0}\in \mathbb{T}^{d}(p,q)\backslash\Omega_{n}$, $x_{0}\in \mathbb{T}^{d}(p,q)\backslash\mathcal{B}_{n,\omega}$, and any eigenvalue $z_{0}=z_{j}^{[0,n-1](\omega_{0},x_{0})}$, such that $L(\omega_{0},z_{0})>\gamma$, if we take $I=I(\omega_{0},z_{0},x_{0})$ as in Lemma \ref{Lemma9.2}, then
$$\Big|z_{k}^{[0,n-1]}(\omega,x)-z_{j}^{[0,n-1]}(\omega,x)\Big|>\exp(-C|I|),$$
for any $k\neq j$ and any $(\omega,x)\in \mathbb{T}^{d}(p,q)\times\mathbb{T}^{d}$, with
$$|\omega-\omega_{0}|,\,\,|x-x_{0}|<\exp(-\exp((\log n)^{2\varepsilon\nu})).$$
\end{lemma}
\begin{proof}
According to the proof of Lemma \ref{Lemma9.2}, \eqref{Lem9.1-(1)} holds for $s\in [0,n-1]\backslash I$, $I=[0,n-1]\cap [m_{0}-\bar{l}-2\exp(l^{2}), m_{0}+\bar{l}+2\exp(l^{2})]$, $l=\lceil(\log n)^{2\varepsilon}\rceil$. Then the conclusion follows from Lemma \ref{Lemma4.2}.
\end{proof}

In what follows, when we deal with the difference between two vectors of different dimensions, we make the dimension of a low dimensional vector equal to the dimension of a high dimensional vector by adding the $0$ element. For example, $\xi=(x_{-5},x_{-4},\ldots,x_{-1},x_{0},x_{1},\ldots,x_{4},x_{5})$, $\upsilon=(y_{-7},y_{-6},\ldots,y_{-1},y_{0},y_{1},\ldots,y_{6},y_{7}),$ $n>m$, $\upsilon-\xi=(y_{-7},y_{-6},y_{-5}-x_{-5},\ldots,y_{-1}-x_{-1},y_{0}-x_{0},y_{1}-x_{1},\ldots,y_{5}-x_{5},y_{6},y_{7})$.
When the columns of a matrix are greater than the rows of a vector, we deal with their product in the same way.

\begin{lemma}\label{Lemma9.4}
Let $\varepsilon\in (0,\frac{1}{5})$. With $\Omega_{n}$, $\mathcal{B}_{n,\omega}$ as in Lemma \ref{Lemma9.2}, applied on $[-(n-1),n-1]$ instead of $[0,n-1]$, let $\hat{\Omega}_{N_{0}}=\mathop{\cup}\limits_{n\geq N_{0}}\Omega_{n}$, $\hat{\mathcal{B}}_{N_{0},\omega}=\mathop{\cup}\limits_{n\geq N_{0}}\mathcal{B}_{n,\omega}$. For $n\geq N_{0}(\alpha_{i},p,q,\gamma,\varepsilon)$, $\omega\in \mathbb{T}^{d}(p,q)\backslash\hat{\Omega}_{N_{0}}$, $x\in\mathbb{T}^{d}\backslash\hat{\mathcal{B}}_{N_{0},\omega}$, if $L(\omega,z_{j}^{[-(n-1),n-1]}(\omega,x))>\frac{3\gamma}{2}$ and $I=I(\omega,z_{j}^{[-(n-1),n-1]}(\omega,x),x)$ then for any $n\leq n'\leq \exp((\log n)^{\frac{1}{10\varepsilon}})$ there exists $j_{n'}$, such that
\begin{equation}\label{Lem9.4-(1)}
\Big|z_{j_{n'}}^{[-(n'-1),n'-1]}(\omega,x)-z_{j}^{[-(n-1),n-1]}(\omega,x)\Big|<\exp(-\frac{\gamma}{12}n),
\end{equation}
\begin{equation}\label{Lem9.4-(2)}
\Big\|u_{j_{n'}}^{[-(n'-1),n'-1]}(\omega,x)-u_{j}^{[-(n-1),n-1]}(\omega,x)\Big\|<\exp(-\frac{\gamma}{48}n).
\end{equation}
Furthermore, $I(\omega,z_{j_{n'}}^{[-(n'-1),n'-1]}(\omega,x),x)\subset[-\frac{3}{4}(n'-1),\frac{3}{4}(n'-1)]$ and
$$\Big|u_{j_{n'}}^{[-(n'-1),n'-1]}(\omega,x;s)\Big|<\exp(-\frac{\gamma}{17}\mathrm{dist}(s,I)),\,\,n\leq |s|\leq n'.$$
\end{lemma}
\begin{proof}
According to Lemma \ref{Lemma9.2}, we get that
$$\Big|u_{j}^{[-(n-1),n-1]}(\omega,x;\pm(n-1))\Big|<\exp(-\frac{\gamma}{4}(\frac{n-1}{2}-1)),$$
$$\Big|u_{j}^{[-(n-1),n-1]}(\omega,x;\pm(n-2))\Big|<\exp(-\frac{\gamma}{4}(\frac{n-1}{2}-2)),$$
$$\Big|u_{j}^{[-(n-1),n-1]}(\omega,x;\pm(n-3))\Big|<\exp(-\frac{\gamma}{4}(\frac{n-1}{2}-3)),$$
and
$$\Big|u_{j}^{[-(n-1),n-1]}(\omega,x;\pm(n-4))\Big|<\exp(-\frac{\gamma}{4}(\frac{n-1}{2}-4)).$$
Then we obtain that
$$\Big\|(\mathcal{E}_{[-(n'-1),n'-1]}^{\beta,\eta}(\omega,x)-z_{j}^{[-(n-1),n-1]}(\omega,x))u_{j}^{[-(n-1),n-1]}(\omega,x)\Big\|<8\exp(-\frac{\gamma}{4}(\frac{n-1}{2}-4)).$$
According to Corollary \ref{eigenvector} (with $\tilde{\varepsilon}=8\exp(-\frac{\gamma}{4}(\frac{n-1}{2}-4)))$, there exists $j_{n'}$ such that $z_{j_{n'}}^{[-(n'-1),n'-1]}(\omega,x)\in \mathcal{D}(z_{j}^{[-(n-1),n-1]}(\omega,x);\tilde{\varepsilon}\sqrt{2})\cap \partial \mathbb{D}$. Thus,
$$\Big|z_{j_{n'}}^{[-(n'-1),n'-1]}(\omega,x)-z_{j}^{[-(n-1),n-1]}(\omega,x)\Big|<2\tilde{\varepsilon}\sqrt{2}<\exp(-\frac{\gamma}{12}n).$$
Applying Corollary \ref{eigenvector} again (with $\hat{\varepsilon}=\exp(-\frac{\gamma}{12}n)$), one can obtain that
$$\Big\|u_{j_{n'}}^{[-(n'-1),n'-1]}(\omega,x)-u_{j}^{[-(n-1),n-1]}(\omega,x)\Big\|<\sqrt{2}\hat{\varepsilon}^{-1}\tilde{\varepsilon}<\exp(-\frac{\gamma}{48}n).$$
From \ref{Lem9.4-(2)}, we can get that
$$\mathop{\sum}\limits_{|s|>n-1}\Big|u_{j_{n'}}^{[-(n'-1),n'-1]}(\omega,x;s)\Big|^{2}<\Big\|u_{j_{n'}}^{[-(n'-1),n'-1]}(\omega,x)-u_{j}^{[-(n-1),n-1]}(\omega,x)\Big\|^{2}<\exp(-\frac{\gamma}{24}n),$$
which implies that $\mathop{\sum}\limits_{|s|\leq \frac{5(n-1)}{9}}\Big|u_{j_{n'}}^{[-(n'-1),n'-1]}(\omega,x;s)\Big|^{2}>\frac{1}{2}.$

It follows from Lemma \ref{Lemma9.2} that
$$I(\omega,z_{j_{n'}}^{[-(n'-1),n'-1]}(\omega,x),x)\subset[-\frac{3(n-1)}{4},\frac{3(n-1)}{4}],$$
and
$$\Big|u_{j_{n'}}^{[-(n'-1),n'-1]}(\omega,x;s)\Big|<\exp\Big(-\frac{\gamma}{4}\mathrm{dist}\big(s,[-\frac{3(n-1)}{4},\frac{3(n-1)}{4}]\big)\Big)<\exp(-\frac{\gamma}{17}\mathrm{dist}(s,I))$$
for any $n-1\leq |s|\leq n'-1$.
\end{proof}

Finally, we establish full scale localization by iterating the eigenpair convergence (Lemma 6.6) over a sequence of exponentially expanding intervals $n_{k}=n^{2^{k}}$. We show that the eigenpair sequence converges to an eigenfunction of the infinite MF-QP CMV matrix that decays exponentially everywhere, confirming pure point spectrum and completing the localization proof.
\begin{theorem}\label{main-result}
Let $\varepsilon\in (0,\frac{1}{5})$ and $\hat{\Omega}_{N_{0}}$, $\hat{\mathcal{B}}_{N_{0},\omega}$ be as in Lemma \ref{Lemma9.4}. For any $N_{0}\geq C(\alpha_{i},p,q,\gamma,\varepsilon)$, $\omega\in \mathbb{T}^{d}(p,q)\backslash\hat{\Omega}_{N_{0}}$, $x\in \mathbb{T}^{d}\backslash\hat{\mathcal{B}}_{N_{0},\omega}$, the following statements hold. Let $n_{k}=n^{2^{k}}$. If $L(\omega,z_{j}^{[-(n-1),n-1]}(\omega,x))>2\gamma$ and $I=I(\omega,z_{j}^{[-(n-1),n-1]}(\omega,x),x)\subset[-\frac{n-1}{2},\frac{n-1}{2}]$, then for each $k\geq 1$ there exist $j_{k}$ such that
\begin{equation}\label{main-(1)}
\begin{aligned}
&\Big|z_{j_{k}}^{[-(n_{k}-1),n_{k}-1]}(\omega,x)-z_{j}^{[-(n-1),n-1]}(\omega,x)\Big|<\exp(-\frac{\gamma}{16}n),\\
&\Big\|u_{j_{k}}^{[-(n_{k}-1),n_{k}-1]}(\omega,x)-u_{j}^{[-(n-1),n-1]}(\omega,x)\Big\|<\exp(-\frac{\gamma}{50}n),
\end{aligned}
\end{equation}
\begin{equation}\label{main-(2)}
\Big|u_{j_{k}}^{[-(n_{k}-1),n_{k}-1]}(\omega,x;s)\Big|<\exp(-\frac{\gamma}{60}\mathrm{dist}(s,I)),\,\,\frac{3(n-1)}{4}\leq |s|\leq n_{k}-1.
\end{equation}
Furthermore, for any $k'\geq k\geq 0$,
\begin{equation}\label{main-(3)}
\begin{aligned}
&\Big|z_{j_{k'}}^{[-(n_{k'}-1),n_{k'}-1]}(\omega,x)-z_{j_{k}}^{[-(n_{k}-1),n_{k}-1]}(\omega,x)\Big|<\exp(-\frac{\gamma}{16}n_{k}),\\
&\Big\|u_{j_{k'}}^{[-(n_{k'}-1),n_{k'}-1]}(\omega,x)-u_{j_{k}}^{[-(n_{k}-1),n_{k}-1]}(\omega,x)\Big\|<\exp(-\frac{\gamma}{50}n_{k}).
\end{aligned}
\end{equation}
Particularly, the limits
$$z(\omega,x):=\lim\limits_{k\rightarrow\infty}z_{j_{k}}^{[-(n_{k}-1),n_{k}-1]}(\omega,x),\,\,
u(\omega,x;s):=\lim\limits_{k\rightarrow\infty}u_{j_{k}}^{[-(n_{k}-1),n_{k}-1]}(\omega,x;s),\,\,s\in \mathbb{Z}$$
exist and
\begin{equation}\label{main-(4)}
|u(\omega,x;s)|<\exp(-\frac{\gamma}{60}\mathrm{dist}(s,I)),\,\, |s|\geq\frac{3(n-1)}{4}.
\end{equation}
\end{theorem}
\begin{proof}
According to Lemma \ref{Lemma9.4}, there exists $j_{k}$, $k\geq 0$, $j_{0}=j$, such that
$$\Big|z_{j_{k+1}}^{[-(n_{k+1}-1),n_{k+1}-1]}(\omega,x)-z_{j_{k}}^{[-(n_{k}-1),n_{k}-1]}(\omega,x)\Big|<\exp(-\frac{\gamma}{12}n_{k}),$$
$$\Big\|u_{j_{k+1}}^{[-(n_{k+1}-1),n_{k+1}-1]}(\omega,x)-u_{j_{k}}^{[-(n_{k}-1),n_{k}-1]}(\omega,x)\Big\|<\exp(-\frac{\gamma}{48}n_{k}),$$
$$L(\omega,z_{j_{k+1}}^{[-(n_{k+1}-1),n_{k+1}-1]}(\omega,x))>\frac{3\gamma}{2},$$
$$I_{k}=I(\omega,z_{j_{k+1}}^{[-(n_{k+1}-1),n_{k+1}-1]}(\omega,x),x)\subset[-\frac{3(n_{k}-1)}{4},\frac{3(n_{k}-1)}{4}],$$
\begin{equation}\label{main-(5)}
\Big|u_{j_{k+1}}^{[-(n_{k+1}-1),n_{k+1}-1]}(\omega,x;s)\Big|<\exp(-\frac{\gamma}{17}\mathrm{dist}(s,I_{k}))<\exp(-\frac{\gamma}{18}\mathrm{dist}(s,I))
\end{equation}
for all $n_{k}-1\leq |s|\leq n_{k+1}-1$. Then for any $k'\geq k \geq 0$, we have
$$\Big|z_{j_{k'}}^{[-(n_{k'}-1),n_{k'}-1]}(\omega,x)-z_{j_{k}}^{[-(n_{k}-1),n_{k}-1]}(\omega,x)\Big|<\sum_{i=k}^{k'-1}\exp(-\frac{\gamma}{12}n_{i})<\exp(-\frac{\gamma}{16}n_{k})$$
and
$$\Big\|u_{j_{k'}}^{[-(n_{k'}-1),n_{k'}-1]}(\omega,x)-u_{j_{k}}^{[-(n_{k}-1),n_{k}-1]}(\omega,x)\Big\|<\sum_{i=k}^{k'-1}\exp(-\frac{\gamma}{48}n_{i})<\exp(-\frac{\gamma}{50}n_{k}).$$
This proves \eqref{main-(1)}, \eqref{main-(3)}, and the existence of the limits $z(\omega,x)$, $u(\omega,x)$.

From \eqref{main-(5)}, we know that \eqref{main-(2)} holds for $n_{k-1}-1\leq |s|\leq n_{k}-1$. Then it is enough to prove that  \eqref{main-(2)} holds for $\frac{3(n-1)}{4}\leq |s|\leq n_{k-1}$. For $1\leq \hat{k}\leq k-2$, due to
$$\Big\|u_{j_{k}}^{[-(n_{k}-1),n_{k}-1]}(\omega,x)-u_{j_{\hat{k}}}^{[-(n_{\hat{k}}-1),n_{\hat{k}}-1]}(\omega,x)\Big\|<\sum_{i=k}^{\hat{k}-1}\exp(-\frac{\gamma}{48}n_{i})<\exp(-\frac{\gamma}{50}n_{k}),$$
we have
\begin{align*}
\big|u_{j_{k}}^{[-(n_{k}-1),n_{k}-1]}(\omega,x;s)\big|&<\big|u_{j_{\hat{k}}}^{[-(n_{\hat{k}}-1),n_{\hat{k}}-1]}(\omega,x;s)\big|+\exp(-\frac{\gamma}{50}n_{k})\\
&<\exp(-\frac{\gamma}{18}\mathrm{dist}(s,I))+\exp(-\frac{\gamma}{50}n_{k})\\
&<\exp(-\frac{\gamma}{60}\mathrm{dist}(s,I)),
\end{align*}
for $n_{\hat{k}-1}-1\leq |s|\leq n_{\hat{k}}-1$. Due to $n_{k}=n^{2^{k}}$, $1\leq \hat{k}\leq k-2$, we have $n-1\leq n_{\hat{k}-1}-1\leq |s|\leq n_{\hat{k}}-1\leq n_{k}-1$. Thus, \eqref{main-(2)} holds for $n-1\leq |s|\leq n_{k}-1$.
According to Lemma \ref{Lemma9.2},
$$\big|u_{j}^{[-(n-1),n-1]}(\omega,x;s)\big|<\exp(-\frac{\gamma}{17}\mathrm{dist}(s,I))\,\,\,\mathrm{for}\,\,\,\frac{3(n-1)}{4}\leq |s|\leq n-1.$$
Therefore,
\begin{align*}
\big|u_{j_{k}}^{[-(n_{k}-1),n_{k}-1]}(\omega,x;s)\big|&<\big|u_{j}^{[-(n-1),n-1]}(\omega,x;s)\big|+\exp(-\frac{\gamma}{50}n_{k})\\
&<\exp(-\frac{\gamma}{17}\mathrm{dist}(s,I))+\exp(-\frac{\gamma}{50}n_{k})\\
&<\exp(-\frac{\gamma}{60}\mathrm{dist}(s,I))
\end{align*}
for $\frac{3(n-1)}{4}\leq |s|\leq n-1$. In conclusion, \eqref{main-(2)} holds for $\frac{3(n-1)}{4}\leq |s|\leq n_{k}-1$.

Taking $k'\rightarrow \infty$ in \eqref{main-(3)}, one can obtain that
$$\Big\|u(\omega,x)-u_{j_{k}}^{[-(n_{k}-1),n_{k}-1]}(\omega,x)\Big\|<\exp(-\frac{\gamma}{50}n_{k}).$$
Similar to the proof of \eqref{main-(2)}, \eqref{main-(4)} follows.
\end{proof}

\section{Appendix}\label{section7}

In order to make the structure of this paper clearer, we put some useful definitions and lemmas in this section.

\subsection{Some Useful Notations}
\begin{definition 4}\cite[Definition 2.3]{BGS02-Acta}
For any positive numbers $a, b$ the notation $a\lesssim b$ means $Ca\leq b$ for some constant $C>0$. By $a\ll b$ we mean that the constant $C$ is very large. If both $a\lesssim b$ and $a\gtrsim b$, then we write $a\asymp b$.
\end{definition 4}
\subsection{Avalanche Principle}
\begin{lemma}\label{avalanche-principle}\cite[Proposition 2.2]{GS01-Annals}
Let $A_{1},\ldots,A_{m}$ be a sequence of arbitrary unimodular $2\times 2$-matrices. Suppose that
\begin{equation}\label{AP-1}
\min_{1\leq j\leq m}\|A_{j}\|\geq \mu\geq m
\end{equation}
and
\begin{equation}\label{AP-2}
\max_{1\leq j< m}\big[\log\|A_{j+1}\|+\log\|A_{j}\|-\log\|A_{j+1}A_{j}\|\big]<\frac{1}{2} \log\mu.
\end{equation}
Then
\begin{equation}\label{AP-3}
\Big|\log\|A_{m}\cdots A_{1}\|+\sum_{j=2}^{m-1}\log\|A_{j}\|-\sum_{j=1}^{m-1}\log\|A_{j+1}A_{j}\|\Big|<C_{A}\frac{m}{\mu},
\end{equation}
where $C_{A}$ is an absolute constant.
\end{lemma}

\subsection{Subharmonic Functions}

\begin{definition 5}\cite[Definition 2.3]{GS08-GAFA}
Suppose $u: \Omega_{1}\times\cdots\Omega_{d}\rightarrow \mathbb{R}\cup \{-\infty\}$ is continuous. Then $u$ is said to be separately subharmonic, if for any $1\leq j\leq d$ and $z_{k}\in\Omega_{k}$ for $k\neq j$ the function
\begin{equation*}
z\rightarrow u(z_{1},\ldots,z_{k-1},z,z_{k+1},\ldots,z_{d})
\end{equation*}
is subharmonic in $z\in \Omega_{j}$.
\end{definition 5}

\begin{lemma}\label{Lemma3.14}\cite[Lemma 2.4]{GS08-GAFA}
Let $u(x_{1},\ldots,x_{d})$ be a separately subharmonic function on $\mathbb{T}_{h}^{d}$. 
Suppose that $\sup\limits_{\mathbb{T}_{h}^{d}}u\leq M$. There are constants $C_{d}$, $C_{d,h}$ such that, if for some $0<\delta<1$ and some $L$
\begin{equation*}
\mathrm{mes}\{x\in \mathbb{T}^{d}:u<-L\}>\delta,
\end{equation*}
then
\begin{equation*}
\sup\limits_{\mathbb{T}^{d}}u\leq C_{h,d}M-\frac{L}{C_{h,d}\log^{d}(C_{d}/\delta)}.
\end{equation*}
\end{lemma}

\subsection{Cartan's Estimate}
In what follows, we denote $\mathcal{D}(z_{0},r)=\{z\in \mathbb{C}:|z-z_{0}|<r\}$.

\begin{definition 6}\cite[Definition 2.12]{GS08-GAFA}
Let $H\geq 1$. For an arbitrary set $\mathcal{B}\subset \mathcal{D}(z_{0},1)\subset \mathbb{C}$ we say that $\mathcal{B}\in \mathrm{Car}_{1}(H,K)$ if $\mathcal{B}\subset \mathop{\cup}\limits^{j_{0}}_{j=1}\mathcal{D}(z_{j},r_{j})$ with $j_{0}\leq K$, and
\begin{equation}\label{def2-(1)}
\sum_{j}r_{j}<e^{-H}.
\end{equation}
If $d\geq 2$ is an integer and $\mathcal{B}\subset\mathop{\prod}\limits_{j=1}^{d}\mathcal{D}(z_{j,0},1)\subset \mathbb{C}^{d}$, then we define inductively that $\mathcal{B}\in \mathrm{Car}_{d}(H,K)$ if for any $1\leq j\leq d$ there exists $\mathcal{B}_{j}\subset\mathcal{D}(z_{j,0},1)\subset \mathbb{C}$, $\mathcal{B}_{j}\in \mathrm{Car}_{1}(H,K)$ so that $\mathcal{B}^{(j)}_{z}\in\mathrm{Car}_{d-1}(H,K)$ for any $z\in \mathbb{C}\backslash\mathcal{B}_{j}$, here $\mathcal{B}^{(j)}_{z}=\{(z_{1},\ldots,z_{d})\in \mathcal{B}:z_{j}=z\}$.
\end{definition 6}

\begin{lemma}\label{Lemma3.23}\cite[Lemma 2.15]{GS08-GAFA}
Let $\varphi(z_{1},\ldots,z_{d})$ be an analytic function defined on a polydisk $\mathcal{P}=\mathop{\prod}\limits_{j=1}^{d}\mathcal{D}(z_{j,0},1)$, $z_{j,0}\in \mathbb{C}$. Let $M\geq\sup\limits_{z\in\mathcal{P}}\log|\varphi(z)|$, $m\leq\log|\varphi(z_{0})|$, $z_{0}=(z_{1,0},\ldots,z_{d,0})$. Given $H\gg 1$, there exists a set $\mathcal{B}\subset \mathcal{P}$, $\mathcal{B}\in \mathrm{Car}_{d}(H^{1/d},K)$, $K=C_{d}H(M-m)$, such that
\begin{equation}\label{Lem3.23-(1)}
\log|\varphi(z)|>M-C_{d}H(M-m)
\end{equation}
for any $z\in \frac{1}{6}\mathcal{P}\backslash \mathcal{B}$. Furthermore, when $d=1$ we can take $K=C(M-m)$ and keep only the disks of $\mathcal{B}$ containing a zero of $\varphi$ in them.
\end{lemma}

\section*{Acknowledgments}
\addcontentsline{toc}{section}{Acknowledgments}

D.P was supported in part by the NSFC (No. 11571327, 11971059). B.Z was supported by Nankai Zhide Foundation.



\begin{thebibliography}{[aa]}
\bibitem{Anderson-1958} P. W.\ Anderson, Absence of diffusion in certain random lattices, \textit{Phys.\ Rev.}\ \textbf{109} (1958), 1492-1505.

\bibitem{AJ10-GEMS}  A.\ Avila, S.\ Jitomirskaya,  Almost localization and almost reducibility, \textit{J.\ Eur.\ Math.\ Soc.}\ \textbf{12} (2010), 93-131.

\bibitem{Bourgain-GAFA-2007} J.\ Bourgain, Anderson localization for quasi-periodic lattice Schr\"{o}dinger Operators on $\mathbb{Z}^{d}$, $d$ arbitrary, \textit{Geom. Func. Anal.}\ \textbf{17} (2007), 682-706.



\bibitem{Bourgain-book}  J.\ Bourgain, Green's Function Estimates for Lattice Schr\"{o}dinger Operators and Applications, Annals of Mathematics Studies, 158, Princeton University Press, Princeton, NJ, 2005.

\bibitem{BG00-Annals} J.\ Bourgain, M.\ Goldstein, On nonperturbative localization with quasi-periodic potential, \textit{Ann.\ Math.}\ \textbf{152} (2000), 835-879.

\bibitem{BGS01-CMP} J.\ Bourgain, M.\ Goldstein, W.\ Schlag, Anderson localization for Schr\"{o}dinger operators on $\mathbb{Z}$ with potentials given by the skew-shift, \textit{Commun.\ Math.\ Phys.}\ \textbf{200} (2001), 583-621.

\bibitem{BGS02-Acta} J.\ Bourgain, M.\ Goldstein, W.\ Schlag, Anderson localization for Schr\"{o}dinger operators on $\mathbb{Z}^{2}$ with quasi-periodic potential, \textit{Acta.\ Math.}\ \textbf{188} (2002), 41-86.

\bibitem{BS01-CMP} J.\ Bourgain, W.\ Schlag, Anderson localization for Schr\"{o}dinger operators on $\mathbb{Z}$ with strongly mixing potentials, \textit{Commun.\ Math.\ Phys.}\ \textbf{215} (2000), 143-175.

\bibitem{Cedzich-2021} C.\ Cedzich, A.H. Werner, Anderson localization for electric quantum walks and skew-shift CMV matrices, \textit{Commun.\ Math.\ Phys.}\ \textbf{387} (2021), 1257-1279.

\bibitem{CS89-CMP} V.A. Chulaevsky, Y.G. Sina\v{\i}, Anderson localization for the $1$-D discrete Schr\"{o}dinger operator with two-frequency potential, \textit{Commun.\ Math.\ Phys.}\ \textbf{125} (1989), 91-112.
\bibitem{ChSu14}   V.A. Chulaevsky, Y. Suhov, A Brief History of Anderson Localization. In: Multi-scale Analysis for Random Quantum Systems with Interaction. Progress in Mathematical Physics, vol 65. Birkh\"{a}user, New York, NY. 2014.

\bibitem{GY20-GAFA} L.\ Ge, J.\ You, Arithmetic version of Anderson localization via reducibility, \textit{Geom. Func. Anal.}\ \textbf{30} (2020), 1370-1401.




\bibitem{GS08-GAFA} M.\ Goldstein, W.\ Schlag, Fine properties of the integrated density of states and a quantitative separation property of the Dirichlet eigenvalues, \textit{Geom. Func. Anal.}\  \textbf{18} (2008), 755-869.

\bibitem{GS01-Annals} M.\ Goldstein, W.\ Schlag, H\"{o}lder continuity of the integrated density of states for quasi-periodic Schr\"{o}dinger equations and averages of shifts of subharmonic functions, \textit{Ann.\ Math.}\ \textbf{154} (2001), 155-203.

\bibitem{GSV16-arXiv} M.\ Goldstein, W.\ Schlag, M.\ Voda, On localization and spectrum of multi-frequency quasiperiodic operators, arXiv:1610.00380 [math.SP].


\bibitem{GSV19-Inventiones} M.\ Goldstein, W.\ Schlag, M.\ Voda, On the spectrum of multi-frequency quasiperiodic Schr\"{o}dinger operators with large coupling, \textit{Invent.\ Math.}\ \textbf{217} (2019), 603-701.

\bibitem{Kruger13-IMRN} H.\ Kr\"{u}ger, Orthogonal polynomials on the unit circle with Verblunsky coefficients defined by the skew-shift, \textit{Int.\ Math.\ Res.\ Not.}\ \textbf{18} (2013), 4135-4169.

\bibitem{Li-Damanik-Zhou} L.\ Li, D.\ Damanik, Q.\ Zhou, Absolutely continuous spectrum for CMV matrices with small quasi-periodic Verblunsky coefficients, \textit{Trans.\ Am.\ Math.\ Soc.}\ \textbf{375} (2022), 6093-6125.

\bibitem{Li-Damanik-Zhou-C} L.\ Li, D.\ Damanik, Q.\ Zhou, Cantor spectrum for CMV matrices with almost periodic Verblunsky coefficients, \textit{J.\ Funct.\ Anal.}\ \textbf{283} (2022), 109709.


\bibitem{LPG24-arXiv} Y.\ Lin, S.\ Guo, D.\ Piao, Anderson localization for CMV matrices with Verblunsky coefficients
 defined by the hyperbolic toral automorphism, \textit{J.\ Funct.\ Anal.}\ (2025), 111103.

\bibitem{LPG23-JFA} Y.\ Lin, D.\ Piao, S.\ Guo, Anderson localization for the quasi-periodic CMV matrices with Verblunsky coefficients defined by the skew-shift, \textit{J.\ Func.\ Anal.}\ \textbf{285} (2023), 1-27.


\bibitem{LTW09} A.\ Lagendijk,\ B. Tiggelen,  D. Wiersma, Fifty years of Anderson localization, \textit{Phys. Tod.} \textbf{62}(8) (2009), 24-29.

\bibitem{MJ} C. A. Marx, S. Jitomirskaya, Dynamics and spectral theory of quasi-periodic Schr\"{o}dinger-type operators, \textit{Ergod. Theor. Dyn. Syst.} \textbf{37} (2015), 2353-2393.

\bibitem{Simon-book} B.\ Simon, Orthogonal Polynomials on the Unit Circle. Part 1. Classical Theory, Amer. Math. Soc. Colloq. Publ., vol. 54, Amer.\ Math.\ Soc., Providence, RI, 2005.
\bibitem{Simon-book2} B. Simon, Orthogonal Polynomials on the Unit Circle. Part 2. Spectral Theory, Amer. Math. Soc. Colloq. Publ., vol. 55, Amer.\ Math.\ Soc., Providence, RI, 2005.
\bibitem{Wang-JMAA} F.\ Wang, A formula related to CMV matrices and Szeg\H{o} cocycles, \textit{J.\ Math.\ Anal.\ Appl.}\ \textbf{464} (2018), 304-316.

\bibitem{WD19-JFA} F.\ Wang, D.\ Damanik, Anderson localization for quasi-periodic CMV matrices and quantum walks, \textit{J.\ Funct.\ Anal.}\ \textbf{276} (2019), 1978-2006.



\bibitem{ZL22-DCDS} Z.\ Zhang, X.\ Li, Anderson localization for MF-QPJacobi operators, \textit{Discrete.\ Contin.\ Dyn.\ Syst.\ Ser.\ A.}\ \textbf{42} (2022), 5869-5891.

\bibitem{Zhao20-JFA} X.\ Zhao, H\"{o}lder continuity of absolutely continuous spectral measure for multi-frequency Schr\"{o}dinger operators, \textit{J.\ Func.\ Anal.}\ \textbf{278} (2020), 108508.




\bibitem{Zhu-arXiv} X.\ Zhu, Localization for random CMV matrices, \textit{J.\ Approx.\ Theory.}\ \textbf{298} (2024), 1-20.



\end{thebibliography}
\end{document}